\documentclass[11pt]{amsart}
\usepackage[centertags]{amsmath}
\usepackage{amsfonts}
\usepackage{amssymb}
\usepackage{amsthm}
\usepackage{newlfont}
\usepackage{amscd}
\usepackage{mathrsfs}
\usepackage[all,cmtip]{xy}
\usepackage{tikz-cd}
\tikzcdset{arrow style=Latin Modern, row sep/normal=1.35em, column sep/normal=1.8em, cramped}
\usepackage{tikz}
\usepackage[pagebackref=false,colorlinks]{hyperref}

\definecolor{mycolor}{HTML}{F7F8E0}
\definecolor{myorange}{RGB}{245,156,74}
\definecolor{cadetgrey}{rgb}{0.57, 0.64, 0.69}
\definecolor{calpolypomonagreen}{rgb}{0.12, 0.3, 0.17}

\hypersetup{pdffitwindow=true,linkcolor=calpolypomonagreen,citecolor=calpolypomonagreen,urlcolor=calpolypomonagreen}
\usepackage{cite}
\usepackage{fancyhdr}
\usepackage{stmaryrd}

\usepackage[OT2,OT1]{fontenc}
\newcommand\cyr{%
\renewcommand\rmdefault{wncyr}%
\renewcommand\sfdefault{wncyss}%
\renewcommand\encodingdefault{OT2}%
\normalfont
\selectfont}
\DeclareTextFontCommand{\textcyr}{\cyr}

\usepackage[toc, page] {appendix}

\numberwithin{equation}{section}

\newtheorem{thm}{Theorem}[section]

\newtheorem{cor}[thm]{Corollary}
\newtheorem{lem}[thm]{Lemma}
\newtheorem{prop}[thm]{Proposition}

\newtheorem{conj}[thm]{Conjecture}

\theoremstyle{definition}
\newtheorem{defn}[thm]{Definition}

\newtheorem{rem}[thm]{Remark}

\newtheorem{ques}[thm]{Question}

\newcommand{\KS}{\mathbf{KS}}
\newcommand{\KSbar}{\overline{\mathbf{KS}}}

\newcommand{\ES}{\mathbf{ES}}
\newcommand{\ks}{\boldsymbol{\kappa}}
\newcommand{\kn}{\widetilde{\boldsymbol{\delta}}}
\newcommand{\widedelta}{\widetilde{\delta}}

\newcommand{\sha}{\textrm{{\cyr SH}}}

\begin{document}
\title[Selmer groups and the main conjecture]{The structure of Selmer groups and the Iwasawa main conjecture for elliptic curves}
\author[C.-H. Kim]{Chan-Ho Kim}
\address{
Department of Mathematics and Institute of Pure and Applied Mathematics,
Jeonbuk National University,
567 Baekje-daero, Deokjin-gu, Jeonju, Jeollabuk-do 54896, Republic of Korea
}
\email{chanho.math@gmail.com}
\thanks{This research was partially supported 
by a KIAS Individual Grant (SP054103) via the Center for Mathematical Challenges at Korea Institute for Advanced Study,
by the National Research Foundation of Korea(NRF) grant funded by the Korea government(MSIT) (No. 2018R1C1B6007009, 2019R1A6A1A11051177, RS-2024-00339824), 
by research funds for newly appointed professors of Jeonbuk National University in 2024, and
by Global-Learning \& Academic research institution for Master’s$\cdot$Ph.D. Students, and Postdocs (LAMP) Program of the National Research Foundation of Korea (NRF) funded by the Ministry of Education (No. RS-2024-00443714).
}
\date{\today}
\begin{abstract}
We reveal a new and refined application of (a \emph{weaker} statement than) the Iwasawa main conjecture for elliptic curves to the \emph{structure} of Selmer groups of elliptic curves of \emph{arbitrary} rank. 
For a large class of elliptic curves, we obtain the following arithmetic consequences:
\begin{itemize}
\item Kato's Kolyvagin system is non-trivial. It is the cyclotomic analogue of Kolyvagin's conjecture.
\item The structure of Selmer groups of elliptic curves over the rationals is completely determined in terms of certain modular symbols.
It is a structural refinement of Birch and Swinnerton-Dyer conjecture.
\item The rank zero $p$-converse, the $p$-parity conjecture, and a new upper bound of the ranks of elliptic curves are obtained. 
\item The conjecture of Kurihara on the semi-local description of mod $p$ Selmer groups is confirmed.
\item An application of the $p$-adic Birch and Swinnerton-Dyer conjecture to the structure of Iwasawa modules is discussed.
\end{itemize}
\end{abstract}
\maketitle

\setcounter{tocdepth}{1}

\section{Introduction}
\subsection{Overview}
\subsubsection{}
In modern number theory, one of the most important themes is to understand the arithmetic meaning of special values of $L$-functions by developing the connection with the \emph{size} of arithmetically interesting groups.
We go beyond this philosophy by giving a description of the \emph{structure} of Selmer groups of elliptic curves in terms of a certain discrete variation of their special $L$-values.

Let $p$ be a prime and $E$ an elliptic curve over $\mathbb{Q}$.
Without a doubt, the Selmer group $\mathrm{Sel}(\mathbb{Q}, E[p^\infty])$ of the $p$-power torsion points of $E$ plays a central role in studying the arithmetic of elliptic curves.
The Selmer group encodes the information of the Mordell--Weil group $E(\mathbb{Q})$ and the Tate--Shafarevich group $\sha(E/\mathbb{Q})$ via the fundamental exact sequence
\[
\xymatrix{
0 \ar[r] & E(\mathbb{Q}) \otimes \mathbb{Q}_p/\mathbb{Z}_p \ar[r] &
\mathrm{Sel}(\mathbb{Q}, E[p^\infty]) \ar[r] & \sha(E/\mathbb{Q})[p^\infty] \ar[r] & 0 .
}
\]
The celebrated Birch and Swinnerton-Dyer (BSD) conjecture predicts that the rank of an elliptic curve $E$ equals the vanishing order of the complex $L$-function of $E$ at $s=1$ and the leading term of the $L$-function knows the size of $\sha(E/\mathbb{Q})$, which is conjecturally finite, and other arithmetic invariants.

The $p$-adic BSD conjecture \`{a} la Mazur--Tate--Teitelbaum \cite{mtt} also predicts that 
the rank of an elliptic curve equals the vanishing order of its $p$-adic $L$-function at the trivial character and the sizes of similar arithmetic invariants are also encoded in the leading term of the $p$-adic $L$-function.

Here, the key common observation is that a certain \emph{complex or $p$-adic variation of $L$-values} detects the rank of an elliptic curve and the leading term of the Taylor expansion of the variation is related to the size of the Tate--Shafarevich group and other arithmetic invariants.

In this article, we focus on the \emph{discrete variation of $L$-values} naturally arising from Kato's Kolyvagin systems and establish the corresponding refined BSD type conjecture, which determines the structure of Selmer groups. 
We call the discrete variation the \emph{collection of Kurihara numbers}, which are explicitly built out from modular symbols (as defined in $\S$\ref{subsec:kurihara-numbers}) and are also realized as the image of Kato's Kolyvagin system under a refinement of the dual exponential map. 
The collection of Kurihara numbers exactly plays the role of $L$-functions in the context of \emph{refined} Iwasawa theory for elliptic curves \`{a} la Kurihara \cite{kurihara-munster, kurihara-iwasawa-2012}. 
\subsubsection{}
The most interesting features of this article are the cyclotomic analogue of Kolyvagin's conjecture and the structural refinement of BSD conjecture. These are special cases of Corollaries \ref{cor:main-ks} and \ref{cor:main-kn} with help of Theorem \ref{thm:refined-bsd}.
\begin{thm}[Corollaries \ref{cor:main-ks} and \ref{cor:main-kn}] \label{thm:main-intro}
Let $E$ be an elliptic curve over $\mathbb{Q}$ and $p \geq 5$ a semi-stable reduction prime for $E$ such that the mod $p$ representation $\overline{\rho} : \mathrm{Gal}(\overline{\mathbb{Q}}/\mathbb{Q}) \to \mathrm{Aut}_{\mathbb{F}_p}(E[p])$ is surjective.
If the Iwasawa main conjecture (Conjecture \ref{conj:IMC}) inverting $p$ holds or
$\mathrm{ord}_{s=1}L(E,s) \leq 1$,
then
\begin{enumerate}
\item Kato's Kolyvagin system is non-trivial, and
\item the structure of Selmer group
$\mathrm{Sel}(\mathbb{Q}, E[p^\infty])$
as a co-finitely generated $\mathbb{Z}_p$-module  is completely determined by the collection of Kurihara numbers as described in Theorem \ref{thm:refined-bsd}.
\end{enumerate}
If we assume the finiteness of $\sha(E/\mathbb{Q})[p^\infty]$ as well, then the explicit formulas for the rank of $E(\mathbb{Q})$ and the exact size of $\sha(E/\mathbb{Q})[p^\infty]$ are also provided in Theorem \ref{thm:refined-bsd}.
\end{thm}
Here, the semi-stable reduction means good \emph{or} multiplicative reduction.
Conclusion (2) of Theorem \ref{thm:main-intro} says that all the non-trivial (moreover, linearly independent) elements of the Selmer group can be detected in terms of modular symbols. It provides us with more information on Selmer groups than BSD conjecture. For example, See Example (\ref{exam:625}) in $\S$\ref{sec:examples}. 
It is believed that the Iwasawa main conjecture inverting $p$ does not help to know the exact sizes of Selmer and Tate--Shafarevich groups due to the $p$-power subtlety. However, our result shows that it is not true at all.
Also,  when $\mathrm{ord}_{s=1}L(E,s) \leq 1$, the Iwasawa main conjecture is \emph{not} even used to obtain the consequences of Theorem \ref{thm:main-intro}. Note that all the known results on the $p$-part of the BSD formula depend heavily on the \emph{full} Iwasawa main conjecture.

\subsubsection{}
We recall the key principles of the standard applications of the Euler system argument and Iwasawa theory for elliptic curves:
\begin{enumerate}
\item If the $L$-value does not vanish, then the Selmer group is finite and bounded by the valuation of the $L$-value \cite{rubin-book, kato-euler-iwasawa-selmer, perrin-riou-euler-systems}.
\item If the Iwasawa main conjecture holds, then the $p$-part of the Birch and Swinnerton-Dyer formula for elliptic curves of analytic rank $\leq 1$ holds (as listed in $\S$\ref{subsubsec:bsd}).
\end{enumerate}
In order to prove Theorem \ref{thm:main-intro}, we improve the key principles as follows:
\begin{enumerate}
\item[(1')] If the collection of Kurihara numbers does not vanish identically, then the structure of the Selmer group is completely determined by the collection of Kurihara numbers (Theorem \ref{thm:refined-bsd}).
\item[(2')] If the Iwasawa main conjecture localized at the augmentation ideal of the Iwasawa algebra holds, then the collection of Kurihara numbers does not vanish identically (Theorem \ref{thm:main-technical} and Proposition \ref{prop:ks-kn-equi-non-vanishing}).
\end{enumerate}
The following flowchart explains how Theorem \ref{thm:main-intro} follows from these improvements with the comparison with well-known applications of the full Iwasawa main conjecture
\[
\xymatrix@R=1.3em{
\textrm{Iwasawa main conjecture} \ar@{=>}[rr]^-{ {\substack{\textrm{standard} \\ \textrm{applications} } }}_-{  \textrm{listed in $\S$\ref{subsubsec:bsd} }  }    \ar@{=>}[d] & &
 {\substack{\textrm{The $p$-part of the BSD formula} \\ \textrm{when the analytic rank $\leq 1$} } } \\
 {\substack{\textrm{Iwasawa main conjecture} \\\textrm{inverting $p$} } } 
  \ar@{=>}[d]   & &  {\substack{\textrm{elliptic curves} \\\textrm{with good ordinary reduction at $p$} } }   \ar@{=>}[ll]^-{ {\substack{ \textrm{ \cite{kato-euler-systems, skinner-urban, wan_hilbert} } \\ \textrm{(Cor. \ref{cor:main-ks})}}}}  \\
 {\substack{\textrm{Iwasawa main conjecture} \\\textrm{``localized at $X\Lambda$"} } } 
  \ar@{<=>}[d]_-{\textrm{Thm. \ref{thm:main-technical}}}    & &  {\substack{\textrm{elliptic curves} \\\textrm{of analytic rank $\leq 1$} } }   \ar@{=>}[dll]_-{ {\substack{ \textrm{ \cite{bertolini-darmon-venerucci, kazim-pollack-sasaki, burungale-skinner-tian-wan} } \\ \textrm{(Cor. \ref{cor:main-ks})}}}}  \\
 {\substack{\textrm{The non-triviality of} \\\textrm{Kato's Kolyvagin systems} } }  \ar@{<=>}[d]_-{\textrm{Prop. \ref{prop:ks-kn-equi-non-vanishing}}}
\ar@{=>}[rr]_-{ {\substack{\textrm{ \cite{mazur-rubin-book} } \\ \textrm{(Thm. \ref{thm:kato-kolyvagin-main})} } }}
& &   {\substack{\textrm{The structure of $p$-strict Selmer groups} \\ \textrm{with no (analytic) rank restriction} } } \\
 {\substack{\textrm{The non-vanishing of} \\\textrm{the collection of Kurihara numbers} } } 
\ar@{=>}[rr]^-{\textrm{Thm. \ref{thm:refined-bsd}}} & &   {\substack{\textrm{The structure of Selmer groups} \\ \textrm{with no (analytic) rank restriction.} } }
}
\]
The improvements provides us with a \emph{better} consequence on Selmer groups from a \emph{weaker} input.
The main idea of (1') is the extension of Mazur--Rubin's structure theorem of $p$-strict Selmer groups (in the above diagram) to the case of classical Selmer groups. Although this idea sounds simple, the proof requires various technical input (studied in $\S$\ref{sec:kolyvagin-systems} and $\S$\ref{sec:kurihara-numbers}) and contains substantial computations (given in $\S$\ref{sec:proof-refined-bsd}).
Also, (2') can be viewed as a bridge from Iwasawa theory to refined Iwasawa theory for elliptic curves, and it is proved in $\S$\ref{sec:Lambda-adic-KS}.

We expect that our strategy generalizes to various settings.
We refer to \cite{kim-gross-zagier} for the cases of Heegner point Kolyvagin systems and bipartite Euler systems.
In particular, Kolyvagin's conjecture becomes a trivial consequence of the Heegner point main conjecture.
\subsubsection{}
In the mod $p$ situation, we confirm the conjecture of Kurihara \cite{kurihara-iwasawa-2012, kurihara-analytic-quantities} on the semi-local description of mod $p$ Selmer groups $\mathrm{Sel}(\mathbb{Q}, E[p])$ (Theorem \ref{thm:kurihara-conjecture}).
More precisely, when $p$ does not divide any Tamagawa factor, we obtain the following implications 
\[
\xymatrix@R=1.3em{
\textrm{Iwasawa main conjecture} \ar@{<=>}[r]  \ar@{<=>}[d] & {\substack{\textrm{The non-vanishing of} \\ \textrm{the collection of mod $p$ Kurihara numbers} } } \ar@{<=>}[ld] \ar@{=>}[d]^-{\textrm{the conjecture of Kurihara}}    \\
  {\substack{\textrm{The non-triviality of} \\\textrm{mod $p$ Kato's Kolyvagin systems} } }  & {\substack{\textrm{The semi-local decription of} \\ \textrm{mod $p$ Selmer groups $\mathrm{Sel}(\mathbb{Q}, E[p])$} \\ \textrm{via one non-zero mod $p$ Kurihara number. } } } 
}
\]
This diagram generalizes \cite{kks, kim-nakamura, sakamoto-p-selmer} in various ways.
\subsubsection{}
Regarding the applications to Birch and Swinnerton-Dyer conjecture, we obtain the following ``rank zero converse" result \emph{without using} any control theorem or the interpolation formula of $p$-adic $L$-functions (cf. \cite[Thm. 2]{skinner-urban}, \cite[Thm. 7]{wan_hilbert}).
\begin{thm}[Corollary \ref{cor:rank-zero-converse}] \label{thm:main-intro-p-converse}
Let $E$ be an elliptic curve over $\mathbb{Q}$. Then the following statements are equivalent.
\begin{enumerate}
\item Both $E(\mathbb{Q})$ and $\sha(E/\mathbb{Q})$ are finite. 
\item $L(E,1) \neq 0$.
\end{enumerate}
\end{thm}
\subsubsection{}
Last, we discuss how the Iwasawa module structure of the fine Selmer groups over the cyclotomic tower can be understood via a version of the $p$-adic BSD conjecture (Theorem \ref{thm:main-structure-fine-iwasawa}).
The rest of this section is devoted to state the precise statements of all the results mentioned above.
\subsection{Kato's Kolyvagin systems and the Iwasawa main conjecture}
We quickly review the notion of Kato's Kolyvagin systems and the Iwasawa main conjecture for elliptic curves.
\subsubsection{}
Let $E$ be an elliptic curve over $\mathbb{Q}$ of conductor $N$ and $p \geq 3$ a prime.
Let $T$ be the $p$-adic Tate module of $E$ and denote by
$\mathbf{z}^{\mathrm{Kato}} = \left\lbrace z^{\mathrm{Kato}}_{F} \in \mathrm{H}^1(F, T) \right\rbrace_{F}$
Kato's Euler system for $E$ where $F$ runs over finite abelian extensions of $\mathbb{Q}$.
See $\S$\ref{subsubsec:kato-euler-systems} for the precise convention.
\subsubsection{} \label{subsubsec:kolyvagin_primes}
Let $k \geq 1$ be an integer.
Let 
$$\mathcal{P}_k = \left\lbrace \ell, \textrm{ a prime} : (\ell, Np) = 1, \ell \equiv 1 \pmod{p^k}, a_\ell(E) \equiv \ell +1 \pmod{p^k} \right\rbrace$$
and $\mathcal{N}_k$ the set of square-free products of primes in $\mathcal{P}_k$.
For $n \in \mathcal{N}_k$, write
 $I_n = \sum_{\ell \vert n} (\ell - 1, a_\ell - \ell -1) \subseteq \mathbb{Z}_p$.
\subsubsection{} \label{subsubsec:cyclotomic-setup}
Let $\mathbb{Q}_{\infty}$ be the cyclotomic $\mathbb{Z}_p$-extension of $\mathbb{Q}$, 
$\mathbb{Q}_m$ the cyclic subextention of $\mathbb{Q}$ of degree $p^m$ in $\mathbb{Q}_\infty$, and $\Lambda = \mathbb{Z}_p \llbracket \mathrm{Gal}(\mathbb{Q}_\infty/\mathbb{Q}) \rrbracket$ the Iwasawa algebra.
By using the $\Lambda$-adic version of the Euler-to-Kolyvagin system map (recalled in Theorem \ref{thm:euler-to-kolyvagin-Lambda-adic}), we obtain
the $\Lambda$-adic Kato's Kolyvagin system 
$\ks^{\mathrm{Kato}, \infty} = \left\lbrace \kappa^{\mathrm{Kato},\infty}_{n} \in \mathrm{H}^1(\mathbb{Q}, T/I_nT \otimes \Lambda) \right\rbrace_{n \in \mathcal{N}_1}$
 from $\mathbf{z}^{\mathrm{Kato}}$.
In particular, we have 
$\kappa^{\mathrm{Kato}, \infty}_{1} = z^{\mathrm{Kato}}_{\mathbb{Q}_\infty} = \varprojlim_m z^{\mathrm{Kato}}_{\mathbb{Q}_m}$
where the projective limit is taken with respect to the corestriction map.
Of course, this element lies in the first Iwasawa cohomology group $\mathrm{H}^1_{\mathrm{Iw}}(\mathbb{Q}, T) = \varprojlim_m \mathrm{H}^1(\mathbb{Q}_m, T) \simeq \mathrm{H}^1(\mathbb{Q}, T \otimes \Lambda)$.
See $\S$\ref{sec:kolyvagin-systems} and $\S$\ref{sec:Lambda-adic-KS} for more details.
\subsubsection{} 
Denote by $\mathrm{Sel}_0(\mathbb{Q}_{\infty}, E[p^\infty] )$ the $p$-strict (``fine") Selmer group of $E[p^\infty]$ over $\mathbb{Q}_\infty$. Write $(-)^\vee = \mathrm{Hom}(-, \mathbb{Q}_p/\mathbb{Z}_p)$.
We recall the Iwasawa main conjecture without $p$-adic $L$-functions \`{a} la Kato \cite[Conj. 12.10]{kato-euler-systems}.
\begin{conj}[IMC] \label{conj:IMC}
The following equality as (principal) ideals of $\Lambda$ holds
\begin{equation} \label{eqn:IMC}
 \mathrm{char}_{\Lambda} \left( \dfrac{\mathrm{H}^1_{\mathrm{Iw}}(\mathbb{Q}, T)}{\Lambda \kappa^{\mathrm{Kato}, \infty}_1 }  \right) 
=
 \mathrm{char}_{\Lambda} \left( \mathrm{Sel}_0(\mathbb{Q}_{\infty}, E[p^\infty] )^\vee  \right) .
\end{equation}
\end{conj}

\subsection{The non-triviality of Kato's Kolyvagin systems}
\subsubsection{} \label{subsubsec:blind-spot-Lambda-primitivity}
Let $\mathfrak{P}$ be a height one prime ideal of $\Lambda$.
We first recall the following concepts from the theory of Kolyvagin systems \cite{mazur-rubin-book}.
\begin{itemize}
\item
$\mathfrak{P}$ is a \textbf{blind spot of $\ks^{\mathrm{Kato}, \infty}$} if $\ks^{\mathrm{Kato}, \infty}$ vanishes modulo $\mathfrak{P}$.
\item $\ks^{\mathrm{Kato}, \infty}$ is \textbf{$\Lambda$-primitive} if any height one prime ideal of $\Lambda$ is not a blind spot of $\ks^{\mathrm{Kato}, \infty}$.
\end{itemize}
These play essential roles in the new and refined applications of Kolyvagin systems to the structure of Selmer groups and the Iwasawa main conjecture, which are not observed directly from the theory of Euler systems \cite{kks}.
It is known that Kato's Kolyvagin system $\ks^{\mathrm{Kato}}$ is non-trivial if and only if $X\Lambda$ is not a blind spot of $\ks^{\mathrm{Kato}, \infty}$ (Proposition \ref{prop:non-triviality-equivalence}).
Here, $\ks^{\mathrm{Kato}}$ is defined by the image of $\mathbf{z}^{\mathrm{Kato}}$ under the Euler-to-Kolyvagin system map (Theorem \ref{thm:euler-to-kolyvagin}).
\subsubsection{}
Fix an isomorphism $\Lambda \simeq \mathbb{Z}_p\llbracket X \rrbracket$ by sending a topological generator of $\mathrm{Gal}(\mathbb{Q}_\infty/\mathbb{Q})$ to $X+1$.
Let $M$ be a finitely generated torsion $\Lambda$-module.
Write 
$\mathrm{char}_{\Lambda} \left( M  \right) =  p^\mu \cdot \prod_i g_i(X)^{m_i}  \cdot \Lambda$
where $g_i(X)$ is an irreducible distinguished polynomial and $g_i(X) \neq g_j(X)$ if $i \neq j$.
We define
$\mathrm{ord}_{ \mathfrak{P}_{i} } \left( \mathrm{char}_{\Lambda} \left( M  \right) \right) = m_i$ and
$\mathrm{ord}_{ p\Lambda } \left( \mathrm{char}_{\Lambda} \left( M  \right) \right) = \mu$
where $\mathfrak{P}_{i}$ is the height one prime ideal generated by $g_i(X)$.

\subsubsection{}
One of the technical innovation of this article is the following statement. 
\begin{thm} \label{thm:main-technical}
Let $E$ be an elliptic curve over $\mathbb{Q}$ and $p \geq 5$ a prime such that $\overline{\rho}$ is surjective.
Let $\mathfrak{P}$ be a height one prime ideal of $\Lambda$.
The following statements are equivalent.
\begin{enumerate}
\item $\ks^{\mathrm{Kato}, \infty}$ does not vanish modulo $\mathfrak{P}$, i.e. $\mathfrak{P}$ is not a blind spot of $\ks^{\mathrm{Kato}, \infty}$.
\item The ``Iwasawa main conjecture localized at $\mathfrak{P}$" holds; in other words,
\begin{equation} \label{eqn:IMC-localized}
\mathrm{ord}_{\mathfrak{P}} \left( \mathrm{char}_{\Lambda} \left( \dfrac{\mathrm{H}^1_{\mathrm{Iw}}(\mathbb{Q}, T)}{\Lambda \kappa^{\mathrm{Kato}, \infty}_1 }  \right) \right)
=
\mathrm{ord}_{\mathfrak{P}} \left( \mathrm{char}_{\Lambda} \left( \mathrm{Sel}_0(\mathbb{Q}_{\infty}, E[p^\infty] )^\vee  \right) \right) .
\end{equation}
\end{enumerate}
In particular, $\ks^{\mathrm{Kato}, \infty}$ is $\Lambda$-primitive if and only if the Iwasawa main conjecture (Conjecture \ref{conj:IMC}) holds.
\end{thm}
\begin{proof}
See $\S$\ref{subsec:proof-of-main-technical}.
\end{proof}
The one direction $(1) \Rightarrow (2)$ and $\Lambda$-primitivity $\Rightarrow$ IMC are established in the work of Mazur--Rubin \cite[Thm. 5.3.10.(2) and (3)]{mazur-rubin-book} (Theorem \ref{thm:mazur-rubin-main-conjecture}). Our contribution is to prove the converse, which has many interesting arithmetic consequences.
In particular, we observe that the non-triviality of Kato's Kolyvagin system $\ks^{\mathrm{Kato}}$ is \emph{strictly weaker} than the Iwasawa main conjecture.

One small disadvantage of our structural approach is that the exact bound of the Tamagawa factors is missing. See Conjecture \ref{conj:kn-non-zero-quantitative} for this aspect.

\subsubsection{}
After having Theorem \ref{thm:main-technical}, we are able to confirm the cyclotomic analogue of Kolyvagin's conjecture \cite{kolyvagin-selmer, wei-zhang-mazur-tate} for many cases.
\begin{cor} \label{cor:main-ks}
Let $E$ be an elliptic curve over $\mathbb{Q}$ and $p\geq 5$ a prime such that $\overline{\rho}$ is surjective.
If the Iwasawa main conjecture localized at $\mathfrak{P} = X\Lambda$ (\ref{eqn:IMC-localized}) holds, then Kato's Kolyvagin system $\ks^{\mathrm{Kato}}$ is non-trivial.
In particular, if one of the following holds:
\begin{enumerate}
\item $E$ has good ordinary reduction at $p$,
\item $E$ has analytic rank zero, or
\item $E$ has analytic rank one and $E$ has semi-stable reduction at $p$,
\end{enumerate}
then $\ks^{\mathrm{Kato}}$ is non-trivial.
In this case, the structure of $\mathrm{Sel}_0(\mathbb{Q}, E[p^\infty])$ is completely determined by $\ks^{\mathrm{Kato}}$ as described in 
Theorem \ref{thm:kato-kolyvagin-main}.
\end{cor}
\begin{proof}
Since the Iwasawa main conjecture \emph{inverting $p$} for elliptic curves with good ordinary reduction at $p$ is established under our setting \cite{kato-euler-systems, skinner-urban, wan_hilbert}, Theorem \ref{thm:main-technical} immediately implies the non-triviality of $\ks^{\mathrm{Kato}}$.

When the analytic rank is zero, the non-triviality of $\kappa^{\mathrm{Kato}}_1 = z^{\mathrm{Kato}}_{\mathbb{Q}}$ for \emph{every} prime $p$ follows from Kato's explicit reciprocity law \cite[Thm. 12.5]{kato-euler-systems}.

When the analytic rank is one, the non-triviality of $\kappa^{\mathrm{Kato}}_1 = z^{\mathrm{Kato}}_{\mathbb{Q}}$ for \emph{every semi-stable} reduction prime $p > 2$ follows from the recent settlement of Perrin-Riou's conjecture \cite{perrin-riou-rational-pts, bertolini-darmon-venerucci, kazim-pollack-sasaki, burungale-skinner-tian-wan}.

The Kolyvagin system description of the structure of $p$-strict Selmer groups follows from the work of Mazur--Rubin \cite{mazur-rubin-book} (Theorem \ref{thm:kato-kolyvagin-main}).
\end{proof}

\subsubsection{} \label{subsubsec:main-kn}
With help of Theorem \ref{thm:refined-bsd}, we also obtain a structural refinement of BSD conjecture.
\begin{cor} \label{cor:main-kn}
Let $E$ be an elliptic curve over $\mathbb{Q}$ and $p\geq 5$ a prime such that
 $\overline{\rho}$ is surjective and
 the Manin constant is prime to $p$.
If the Iwasawa main conjecture localized at $\mathfrak{P} = X\Lambda$ (\ref{eqn:IMC-localized}) holds, then the collection of Kurihara numbers, denoted by $\kn$, does not vanish identically.
In particular, if one of the following holds:
\begin{enumerate}
\item $E$ has good ordinary reduction at $p$,
\item $E$ has analytic rank zero, or
\item $E$ has analytic rank one and $E$ has semi-stable reduction at $p$,
\end{enumerate}
then $\kn$ does not vanish identically.
In this case, the structure of $\mathrm{Sel}(\mathbb{Q}, E[p^\infty])$ is completely determined by $\kn$ as described in Theorem \ref{thm:refined-bsd}.
\end{cor}
\begin{proof}
By Proposition \ref{prop:ks-kn-equi-non-vanishing}, the non-triviality of $\ks^{\mathrm{Kato}}$ is equivalent to the non-vanishing of $\kn$ under our running hypotheses.
The modular symbol description of the structure of Selmer groups follows from Theorem \ref{thm:refined-bsd}.
\end{proof}

\subsection{Modular symbols and Kurihara numbers} \label{subsec:kurihara-numbers}
In order to give a precise statement of Theorem \ref{thm:refined-bsd}, we quickly review the notion of modular symbols and Kurihara numbers.
\subsubsection{}
Let $p \geq 5$ be a prime and $E$ an elliptic curve over $\mathbb{Q}$ such that the residual representation $\overline{\rho}$ is irreducible and the Manin constant is prime to $p$.
The Manin constant assumption becomes vacuous when $E$ has semi-stable reduction at $p$ \cite[Cor. 4.1]{mazur-rational-isogenies}.

Let $f = \sum_{n \geq 1} a_n(E) q^n \in S_2(\Gamma_0(N))$ be the newform corresponding to $E$ \cite{bcdt}.
For each $\dfrac{a}{b} \in \mathbb{Q}$, the modular symbol $\left[\dfrac{a}{b}\right]^+$ is defined by equality
$$2 \pi \cdot \int^{\infty}_0 f(\dfrac{a}{b}+ iy )dy = \left[\dfrac{a}{b}\right]^+ \cdot \Omega^+_E + \left[\dfrac{a}{b}\right]^- \cdot \sqrt{-1} \cdot \Omega^-_E$$
where $\left[\dfrac{a}{b}\right]^+$ and $\left[\dfrac{a}{b}\right]^-$ are rational numbers \cite[(1.1)]{mazur-tate}.
Especially, the real N\'{e}ron period $\Omega^{+}_E$  of $E$ is taken as the absolute value of the integral of an invariant differential of a global minimal Weierstrass model of $E$ over $E(\mathbb{R})$. (cf. \cite[(1.1)]{mazur-tate}.)
Under our assumptions, we have $\left[\dfrac{a}{b}\right]^+ \in \mathbb{Z}_{(p)}$, i.e. the modular symbols are $p$-integral.

\subsubsection{}
We follow the convention in $\S$\ref{subsubsec:kolyvagin_primes}.
For each prime $\ell \in \mathcal{P}_k$, we fix a primitive root $\eta_\ell$ mod $\ell$ and define
 $\mathrm{log}_{\eta_\ell}(a) \in \mathbb{Z}/(\ell-1)$ by $\eta^{ \mathrm{log}_{\eta_\ell}(a)}_\ell \equiv a \pmod{\ell}$.
\subsubsection{}
For each $n \in \mathcal{N}_1$, the \textbf{(mod $I_n$) Kurihara number at $n$} is defined by
$$\widetilde{\delta}_n  = \sum_{a \in (\mathbb{Z}/n\mathbb{Z})^\times } \left( \overline{\left[ \dfrac{a}{n} \right]^+}  \cdot   \prod_{\ell \mid n} \overline{\mathrm{log}_{\eta_\ell} (a)} \right) \in \mathbb{Z}_p/I_n\mathbb{Z}_p$$
where $\overline{\left[ \dfrac{a}{n} \right]^+}$ is the mod $I_n$ reduction of $\left[ \dfrac{a}{n} \right]^+$ and 
$\overline{\mathrm{log}_{\eta_\ell} (a)}$ is also the mod $I_n$ reduction of $\mathrm{log}_{\eta_\ell} (a)$.
Note that $\widetilde{\delta}_n$ is well-defined up to $(\mathbb{Z}_p/I_n\mathbb{Z}_p)^\times$.
When $n \in \mathcal{N}_{k}$, we write
$\widetilde{\delta}^{(k)}_n = \widetilde{\delta}_n \pmod{p^k} \in \mathbb{Z}/p^k\mathbb{Z}$.
When $n=1$, we have
$$\widetilde{\delta}_1 = [0]^+ = \dfrac{L(E, 1)}{\Omega^+_E} \in \mathbb{Z}_{(p)} .$$
\subsubsection{}
The \textbf{collection of Kurihara numbers} is defined by 
$$\widetilde{\boldsymbol{\delta}} = \left\lbrace \widetilde{\delta}_n \in \mathbb{Z}_p/I_n \mathbb{Z}_p : n \in \mathcal{N}_1 \right\rbrace .$$
For a square-free integer $n$, denote by $\nu(n)$ the number of prime factors of $n$ with convention $\nu(1)=0$.
The \textbf{vanishing order of $\widetilde{\boldsymbol{\delta}}$} is defined by
$$\mathrm{ord} (\widetilde{\boldsymbol{\delta}}) = \mathrm{min} \left\lbrace \nu(n) : n \in \mathcal{N}_1, \widetilde{\delta}_n \neq 0 \right\rbrace .$$
We write $\mathrm{ord} (\widetilde{\boldsymbol{\delta}}) = \infty$ if $\widetilde{\delta}_n = 0$ for all $n \in \mathcal{N}_1$, i.e. the collection of Kurihara numbers vanishes identically.
The following conjecture on the non-vanishing of $\kn$ follows from the Iwasawa main conjecture (Conjecture \ref{conj:IMC}) thanks to Theorem \ref{thm:main-technical} and Proposition \ref{prop:ks-kn-equi-non-vanishing}.
\begin{conj} \label{conj:kn-non-zero}
Let $E$ be an elliptic curve over $\mathbb{Q}$ and $p \geq 5$ a prime such that
 $\overline{\rho}$ is surjective and the Manin constant is prime to $p$.
Then $\mathrm{ord}(\kn) <\infty$.
\end{conj}
\subsection{The structure of Selmer groups and Kurihara numbers}
\subsubsection{}
Denote by $\partial^{(0)} (\widetilde{\boldsymbol{\delta}})$ the $p$-adic valuation of $\widetilde{\delta}_1$.
 We also define
\begin{align*}
\partial^{(i)}(\widetilde{\boldsymbol{\delta}}) & = 
\mathrm{min} \left\lbrace \mathrm{max} \left\lbrace j_n : \widetilde{\delta}_n \in p^{j_n} \mathbb{Z}_p/I_n\mathbb{Z}_p \right\rbrace : n \in \mathcal{N}_1 \textrm{ with } \nu(n) = i  \right\rbrace , \\
\partial^{(\infty)}(\widetilde{\boldsymbol{\delta}}) & = \mathrm{min}  \lbrace \partial^{(i)}(\widetilde{\boldsymbol{\delta}}) : 0 \leq i \rbrace.
\end{align*}
\subsubsection{}
Denote by $M_{/\mathrm{div}}$ the quotient of $M$ by its maximal divisible submodule.
\begin{thm} \label{thm:refined-bsd}
Let $E$ be an elliptic curve over $\mathbb{Q}$ and $p \geq 5$ a prime such that
 $\overline{\rho}$ is surjective and the Manin constant is prime to $p$.
If $\mathrm{ord} (\kn) < \infty$,
then
\begin{enumerate}
\item $ \mathrm{cork}_{\mathbb{Z}_p}\mathrm{Sel}(\mathbb{Q}, E[p^\infty]) = \mathrm{ord}(\kn)$, i.e. 
$\mathrm{Fitt}_{i, \mathbb{Z}_p} \left(  \mathrm{Sel}(\mathbb{Q}, E[p^\infty])^\vee \right) = 0$
 for all $0 \leq i \leq \mathrm{ord}(\kn)-1$, 
\item $\mathrm{Fitt}_{i, \mathbb{Z}_p} \left(  \mathrm{Sel}(\mathbb{Q}, E[p^\infty])^\vee \right) = p^{\partial^{(i)} ( \kn ) - \partial^{(\infty)}(\kn)} \mathbb{Z}_p$
for all $i \geq \mathrm{ord}(\kn)$ with $i \equiv \mathrm{ord}(\kn) \pmod{2}$, and
\item $\mathrm{length}_{\mathbb{Z}_p} \left( \mathrm{Sel}(\mathbb{Q}, E[p^\infty])_{/\mathrm{div}} \right) = \partial^{(\mathrm{ord}(\kn))} (\widetilde{\boldsymbol{\delta}} ) - \partial^{(\infty)} (\widetilde{\boldsymbol{\delta}} )$.
\end{enumerate}
In other words, we have
$$\mathrm{Sel}(\mathbb{Q}, E[p^\infty])  \simeq \left( \mathbb{Q}_p/\mathbb{Z}_p \right)^{\oplus \mathrm{ord}(\kn)} \oplus \bigoplus_{i\geq 1} \left( \mathbb{Z}/p^{(  \partial^{(\mathrm{ord}(\kn) + 2(i-1))} (\widetilde{\boldsymbol{\delta}} ) - \partial^{(\mathrm{ord}(\kn) + 2i)} (\widetilde{\boldsymbol{\delta}} ) )/2}\mathbb{Z} \right)^{\oplus 2}  .$$
If we further assume the finiteness of $\sha(E/\mathbb{Q})[p^\infty]$, then we have
\begin{enumerate}
\item[(4)] $\mathrm{rk}_{\mathbb{Z}} E( \mathbb{Q} ) = \mathrm{ord}(\kn)$,
\item[(5)] $\sha(E/\mathbb{Q})[p^\infty] \simeq \bigoplus_{i\geq 1} \left( \mathbb{Z}/p^{(  \partial^{(\mathrm{ord}(\kn) + 2(i-1))} (\widetilde{\boldsymbol{\delta}} ) - \partial^{(\mathrm{ord}(\kn) + 2i)} (\widetilde{\boldsymbol{\delta}} ) )/2}\mathbb{Z} \right)^{\oplus 2} $, and so
\item[(6)] $\mathrm{length}_{\mathbb{Z}_p} \left( \sha(E/\mathbb{Q})[p^\infty] \right) = \partial^{(\mathrm{ord}(\kn))} (\widetilde{\boldsymbol{\delta}} ) - \partial^{(\infty)} (\widetilde{\boldsymbol{\delta}} )$.
\end{enumerate}
\end{thm}
\begin{proof}
See $\S$\ref{sec:proof-refined-bsd}. The first two statements imply all the other statements.
\end{proof}
When  $\partial^{(\infty)}(\widetilde{\boldsymbol{\delta}})=0$, the ideal $p^{\partial^{(i)} (\widetilde{\boldsymbol{\delta}} )} \mathbb{Z}_p$ is closely related to the \textbf{$i$-th higher Stickelberger ideal} $\Theta_i = \sum_{0\leq i_0 \leq i} p^{\partial^{(i_0)} (\widetilde{\boldsymbol{\delta}} )} \mathbb{Z}_p$ defined by Kurihara \cite{kurihara-munster, kurihara-iwasawa-2012}.
The conclusion of Theorem \ref{thm:refined-bsd} should be viewed as the structural refinement of Birch and Swinnerton-Dyer conjecture, and it was first observed by Kurihara from a completely different perspective under various technical assumptions. See $\S$\ref{subsubsec:refined-iwasawa-theory} for details.

We are able to obtain an approximate structural upper bound of $\mathrm{Sel}(\mathbb{Q}, E[p^\infty])$ in practice since each $i$-th higher Fitting ideal 
can be approximated by computing the valuations of $\widedelta_n$'s with fixed $\nu(n) = i$.
\subsubsection{} \label{subsubsec:numerical-verification}
Considering the compatibility with the conjectural classical BSD formula, we expect the following equality, which is a quantitative refinement of Conjecture \ref{conj:kn-non-zero}. See \cite[Conj. 4.5]{wei-zhang-cdm} for the Heegner point version.
\begin{conj} \label{conj:kn-non-zero-quantitative}
Let $E$ be an elliptic curve over $\mathbb{Q}$ of conductor $N$ and $p \geq 5$ a prime such that
 $\overline{\rho}$ is surjective and the Manin constant is prime to $p$.
Then
$$ \partial^{(\infty)}(\kn) = \sum_{\ell \mid N} \mathrm{ord}_p(c_\ell)$$
where $c_\ell$ is the Tamagawa factor of $E$ at $\ell$.
\end{conj}
We expect that the only obstructions to observe the non-vanishing of $\kn$ are the Selmer corank (Theorem \ref{thm:refined-bsd}), Tamagawa factors (Conjecture \ref{conj:kn-non-zero-quantitative}), and the functional equation (Proposition \ref{prop:vanishing-delta-n}).
In this sense, the numerical verification of Conjecture \ref{conj:kn-non-zero} is not very difficult.
 It is easy to compute each Kurihara number numerically, at least in the mod $p$ situation \cite{kurihara-iwasawa-2012, kks, kim-survey, kurihara-analytic-quantities}. An efficient algorithm to compute $\widetilde{\delta}_n$ is available at \href{https://github.com/aghitza/kurihara_numbers}{https://github.com/aghitza/kurihara\_numbers} thanks to Alexandru Ghitza.

After the first version of this article becomes available (March 2022), Burungale--Castella--Grossi--Skinner also proved the non-triviality of $\ks^{\mathrm{Kato}}$ (Corollary \ref{cor:main-ks}) and Conjecture \ref{conj:kn-non-zero-quantitative} for elliptic curves with good ordinary reduction even without assuming the surjectivity of $\overline{\rho}$ \cite{burungale-castella-grossi-skinner-indivisibility}. Their approach uses the full Iwasawa main conjecture and is based on a certain $p$-adic approximation of the specialization of $\Lambda$-adic Kato's Kolyvagin system with careful control of error terms. Their work does not have any direct application to the exact size of Selmer groups yet.
In our approach, Mazur--Rubin's approximation argument is used to connect the Iwasawa main conjecture localized at the augmentation ideal with the non-triviality of $\ks^{\mathrm{Kato}}$ and $\widetilde{\boldsymbol{\delta}}$ as in $\S$\ref{subsec:proof-of-main-technical}. 
One of the key features of our work is to reveal the structure of Selmer groups.

\subsection{The conjecture of Kurihara}
\subsubsection{}
In the mod $p$ situation, we confirm the conjecture of Kurihara \cite[Conj. 2]{kurihara-iwasawa-2012},
which says that a single $n$ with $\widedelta^{(1)}_n \neq 0$ entirely determines the structure of mod $p$ Selmer groups in terms of purely local data when all the Tamagawa factors are prime to $p$.
In this case, the non-vanishing of $\kn^{(1)}$ is equivalent to the Iwasawa main conjecture \cite[Conj. 12.10]{kato-euler-systems}.
\begin{thm} \label{thm:kurihara-conjecture}
Let $E$ be an elliptic curve over $\mathbb{Q}$ and $p \geq 5$ a prime
such that
\begin{itemize}
\item $\overline{\rho}$ is surjective,
\item the Manin constant is prime to $p$,
\item $E(\mathbb{Q}_p)[p] = 0$, and
\item all the Tamagawa factors are prime to $p$.
\end{itemize}
Then the following statements are equivalent.
\begin{enumerate}
\item $\widedelta^{(1)}_n \neq 0$ in $\mathbb{F}_p$ for some $n \in \mathcal{N}_1$ with $\nu(n) = \mathrm{ord}(\kn^{(1)})$.
\item The mod $p$ Kato's Kolyvagin system $\ks^{\mathrm{Kato},(1)}$ is non-trivial.
\item The Iwasawa main conjecture holds.
\end{enumerate}
In this case, the canonical homomorphism
$$\mathrm{Sel}(\mathbb{Q}, E[p]) \to \bigoplus_{\ell \mid n} ( E(\mathbb{Q}_\ell) \otimes \mathbb{Z}/p\mathbb{Z} ) \simeq \bigoplus_{\ell \mid n} ( E(\mathbb{F}_\ell) \otimes \mathbb{Z}/p\mathbb{Z} )$$
 is an isomorphism, and Conjecture \ref{conj:kn-non-zero-quantitative} follows with value zero.
If we further assume that $\sha(E/\mathbb{Q})[p]$ is trivial, then we have a mod $p$ exact rank formula 
$$\mathrm{rk}_{\mathbb{Z}} E(\mathbb{Q})  = \mathrm{ord}(\kn^{(1)}) = \nu(n).$$
\end{thm}
\begin{proof}
See $\S$\ref{sec:kurihara-conjecture}.
\end{proof}
The implication (1) $\Rightarrow$ (2) $\Rightarrow$ (3) is the main result of \cite{kks}.
The implication (3) $\Rightarrow$ (1) is proved by Kurihara \cite{kurihara-munster} (under the assumption on the non-degeneracy of $p$-adic height pairings) and Sakamoto \cite{sakamoto-p-selmer} when $E$ has good ordinary reduction at $p$ via Kolyvagin systems of Gauss sum type and Kolyvagin systems of rank zero, respectively.
In particular, Theorem \ref{thm:kurihara-conjecture} recovers the main result of \cite{sakamoto-p-selmer}.
See also \cite{kurihara-iwasawa-2012, kurihara-analytic-quantities}.

\subsection{Applications to Birch and Swinnerton-Dyer conjecture}
We discuss some consequences of Theorem \ref{thm:refined-bsd} towards Birch and Swinnerton-Dyer conjecture.
\subsubsection{}
If $\kn$ does not vanish, then the rank zero $p$-converse to a theorem of Gross--Zagier--Kolyvagin follows.
\begin{cor} \label{cor:rank-zero-converse}
Let $E$ be an elliptic curve over $\mathbb{Q}$ and $p \geq 5$ a prime.
We assume that
 $\overline{\rho}$ is surjective,
 the Manin constant is prime to $p$, and
 $\mathrm{ord} (\kn) < \infty$. 
Then $\mathrm{Sel}(\mathbb{Q}, E[p^\infty])$ is finite if and only if $L(E, 1)\neq 0$.
In particular, Theorem \ref{thm:main-intro-p-converse} follows.
\end{cor}
\begin{proof}
It immediately follows from Theorem \ref{thm:refined-bsd} except the `In particular' part.
For CM elliptic curves, See \cite[Thm. 11.1]{rubin-main-conj-cm}. See also \cite{burungale-tian-jnt, burungale-tian-rank-zero-p-converse}.
For non-CM elliptic curves, we are always possible to choose a good ordinary prime $p$ such that $\overline{\rho}$ is surjective.
By Theorem \ref{thm:main-technical} and Proposition \ref{prop:ks-kn-equi-non-vanishing}, the Iwasawa main conjecture inverting $p$ \cite{kato-euler-systems, skinner-urban, wan_hilbert} implies  $\mathrm{ord} (\kn) < \infty$. 
By Theorem \ref{thm:refined-bsd}, $\mathrm{Sel}(\mathbb{Q}, E[p^\infty])$ is finite if and only if $\widedelta_1 = \dfrac{L(E, 1)}{\Omega^+_E} \neq 0$.
It is known that $L(E, 1)\neq 0$ implies the finiteness of both $E(\mathbb{Q})$ and $\sha(E/\mathbb{Q})$ thanks to the work of Kato \cite[Thm. 3.5.4]{rubin-book}, \cite[Cor. 14.3]{kato-euler-systems}.
Thus, Theorem \ref{thm:main-intro-p-converse} follows.
\end{proof}
The reader is invited to compare Corollary \ref{cor:rank-zero-converse} with \cite[Thm. 2]{skinner-urban}, \cite[Thm. 7]{wan_hilbert}, and \cite[Thm. E in $\S$4]{skinner-converse}, which use the control theorem and the interpolation formula of $p$-adic $L$-functions instead of Theorem \ref{thm:refined-bsd}.

\subsubsection{}
By using the functional equation for $\widedelta_n$ discussed in (\ref{eqn:functional-equation-delta_n}), we immediately obtain the $p$-parity conjecture from Theorem \ref{thm:refined-bsd}. Our approach is completely independent of others. 
\begin{cor} \label{cor:p-parity}
Let $E$ be an elliptic curve over $\mathbb{Q}$ and $p \geq 5$ a prime such that
 $\overline{\rho}$ is surjective and the Manin constant is prime to $p$.
If $\mathrm{ord} (\kn) < \infty$, then
we have
$$\mathrm{cork}_{\mathbb{Z}_p}\mathrm{Sel}(\mathbb{Q}, E[p^\infty]) \equiv
\mathrm{ord}_{s=1} L(E,s) \pmod{2} .$$
In particular, the $p$-parity conjecture holds for elliptic curves with good ordinary reduction at $p$.
\end{cor}
\subsubsection{}
We obtain the following ``easy and practical" upper bound of the ranks of elliptic curves.
\begin{cor} \label{cor:upper-bound}
Let $E$ be an elliptic curve over $\mathbb{Q}$ and $p \geq 5$ a prime such that
 $\overline{\rho}$ is surjective and the Manin constant is prime to $p$.
If $\widetilde{\delta}_n \neq 0$ for some $n \in \mathcal{N}_1$, 
then
$$ \mathrm{rk}_{\mathbb{Z}}E(\mathbb{Q}) \leq \nu(n).$$
\end{cor}

\subsection{Iwasawa modules and a $p$-adic Birch and Swinnerton-Dyer conjecture}
\subsubsection{}
Compared with the $p$-adic BSD conjecture for elliptic curves \cite{mtt}, we consider the following variant, namely the $p$-adic BSD conjecture for Kato's zeta elements.
\begin{conj}[$p$-adic BSD for Kato's zeta elements] \label{conj:p-adic-BSD-kato-zeta}
$$\mathrm{cork}_{\mathbb{Z}_p} \mathrm{Sel}_0(\mathbb{Q}, E[p^\infty]) = \mathrm{ord}_{X\Lambda} \left( 
\mathrm{char}_{\Lambda} \left( \dfrac{\mathrm{H}^1_{\mathrm{Iw}}(\mathbb{Q}, T)}{\kappa^{\mathrm{Kato}, \infty}_1} \right) \right) .$$
\end{conj}
This type of $p$-adic BSD conjecture can be found in \cite[Conj. 4.16]{burns-kurihara-sano}.
Indeed, it turns out that Conjecture \ref{conj:p-adic-BSD-kato-zeta} implies the usual $p$-adic BSD conjecture \cite[Cor. 6.6]{burns-kurihara-sano}.
\subsubsection{}
The interesting feature we observed is that Conjecture \ref{conj:p-adic-BSD-kato-zeta} has an application to the Iwasawa module structure of fine Selmer groups. In particular, Conjecture \ref{conj:p-adic-BSD-kato-zeta} is slightly more refined than IMC localized at $\mathfrak{P} = X\Lambda$ (\ref{eqn:IMC-localized}) from the viewpoint of Iwasawa modules.
Fix a pseudo-isomorphism
\begin{equation} \label{eqn:pseudo-isom-XLambda}
\mathrm{Sel}_0(\mathbb{Q}_{\infty}, E[p^\infty] )^\vee \to \bigoplus_i \Lambda / X^{m_i} \Lambda \oplus \bigoplus_j \Lambda / f_j \Lambda
\end{equation}
where each $f_j$ is prime to $X\Lambda$.
\begin{thm} \label{thm:main-structure-fine-iwasawa}
Conjecture \ref{conj:p-adic-BSD-kato-zeta} implies $m_i =1$ for every $i$ in (\ref{eqn:pseudo-isom-XLambda}).
\end{thm}
\begin{proof}
See $\S$\ref{sec:iwasawa-modules}.
\end{proof}

\subsection{Comparison with related conjectures and results} \label{subsec:comparisons}

\subsubsection{BSD} \label{subsubsec:bsd}
The rank part of BSD conjecture is completely open when $\mathrm{rk}_{\mathbb{Z}} E(\mathbb{Q})$ or $\mathrm{ord}_{s=1} L(E, s)$ is larger than 1.
Thanks to the work of Coates--Wiles,  Rubin, Gross--Zagier, Kolyvagin, and Kato \cite{coates-wiles-bsd-1977, rubin-tate-shafarevich, gross-zagier-original, kolyvagin-euler-systems, kato-euler-systems} based on the method of Euler systems, it is known that
if $\mathrm{ord}_{s=1}L(E,s) \leq 1$, then the rank part is true and $\sha(E/\mathbb{Q})$ is finite.
Also, its $p$-converse is recently developed enormously; see \cite{skinner-urban, skinner-converse, wei-zhang-mazur-tate, burungale-tian-p-converse, kim-p-converse, burungale-skinner-tian-wan} for example.

Furthermore, when $\mathrm{ord}_{s=1}L(E,s) \leq 1$, a large amount of the $p$-part of the BSD formula is resolved through the establishment of various Iwasawa main conjectures and complex and $p$-adic Gross--Zagier formulas thanks to the work of Skinner--Urban, Kobayashi, W. Zhang, X. Wan, Berti--Bertolini--Venerucci, Jetchev--Skinner--Wan, Castella, Castella--Grossi--Lee--Skinner, B\"{u}y\"{u}kboduk--Pollack--Sasaki, Burungale--Castella--Skinner , and Burungale--Skinner--Tian--Wan  \cite{skinner-urban, kobayashi-gross-zagier, wei-zhang-mazur-tate, wan_hilbert, wan-rankin-selberg, berti-bertolini-venerucci, jetchev-skinner-wan, castella-cambridge, castella-grossi-lee-skinner, kazim-pollack-sasaki, burungale-castella-skinner-gl2, burungale-skinner-tian-wan}. See also \cite{burungale-skinner-tian-survey} for a survey on this topic.

Corollary \ref{cor:main-kn} provides us with the structural information of $\sha(E/\mathbb{Q})[p^\infty]$ for semi-stable elliptic curve $E$ of analytic rank $\leq 1$ and every prime $p \geq 5$ with the surjectivity condition on $\overline{\rho}$ \emph{but with no use of} the Iwasawa main conjecture (cf. \cite{jetchev-skinner-wan, castella-cambridge}).
\subsubsection{$p$-adic BSD}
We focus only on  the good ordinary reduction case here and do not discuss the exceptional zero case or the supersingular variant. We refer to \cite{mtt, stein-wuthrich, balakrishnan-muller-stein} for details.

Regarding the rank part of the $p$-adic BSD conjecture \`{a} la Mazur--Tate--Teitelbaum, Kato proved $\mathrm{cork}_{\mathbb{Z}_p} \mathrm{Sel}(\mathbb{Q}, E[p^\infty]) \leq  \mathrm{ord}_{X=0} L_p(E)$ based on the one-sided divisibility of the Iwasawa main conjecture where $L_p(E)$ is the $p$-adic $L$-function of $E$ \cite{rubin-main-conj-cm, kato-euler-systems}.
In order to prove the equality of the rank part, we need the Iwasawa main conjecture, the non-degeneracy of the $p$-adic height pairing on $E$, and the finiteness of $\sha(E/\mathbb{Q})[p^\infty]$. 
The leading term part also follows from the combination of these conjectures. See \cite[Thm. 2']{schneider-height-2}, \cite[Prop. 3.4.6]{perrin-riou-rational-pts}, and \cite[Thm. 1.7]{balakrishnan-muller-stein}.

Our structural refinement (Theorem \ref{thm:refined-bsd}) works for elliptic curves of arbitrary rank and depends only on Conjecture \ref{conj:kn-non-zero}, which is strictly weaker than the Iwasawa main conjecture. 
Also, no regulator is involved in our approach.
Furthermore, we may ask the following na\"{i}ve questions.
\begin{ques} \label{ques:refining-bsd}
Let $r = \mathrm{rk}_{\mathbb{Z}}E(\mathbb{Q})$ and assume that $\sha(E/\mathbb{Q})$ is finite.
Denote by $s$ the number of generators of $\sha(E/\mathbb{Q})$, and by $s_p$ the number of generators of $\sha(E/\mathbb{Q})[p^\infty]$.
\begin{enumerate}
\item
Can the structure of $\sha(E/\mathbb{Q})$ be determined by the values $L^{(r)}(E,1), \cdots , L^{(r+s)}(E,1)$?
\item
Can the structure of $\sha(E/\mathbb{Q})[p^\infty]$ be determined by the values $L^{(r)}_p(E), \cdots , L^{(r+s_p)}_p(E)$?
\end{enumerate}
\end{ques}

\subsubsection{Refined BSD type conjectures} \label{subsubsec:refined-bsd-type}
In \cite{mazur-tate}, Mazur and Tate proposed several conjectures on equivariant refinements of the BSD conjecture.
Theorem \ref{thm:refined-bsd} refines their weak vanishing conjecture \cite[Conj. 1]{mazur-tate} by making their inequality into the equality when the character is trivial and also refines their weak main conjecture \cite[Conj. 3]{mazur-tate} and a ``Birch--Swinnerton-Dyer type" conjecture \cite[Conj. 4]{mazur-tate} by considering all the higher Fitting ideals (but with no equivariant variation).

\subsubsection{Refined Iwasawa theory} \label{subsubsec:refined-iwasawa-theory}
In a series of his papers \cite{kurihara-fitting, kurihara-documenta, kurihara-plms, kurihara-iwasawa-2012, kurihara-munster}, Kurihara developed refined Iwasawa theory and the theory of Kolyvagin systems of Gauss sum type to study the structure of Selmer groups.
Iwasawa himself was also interested in Kurihara's refinement of the Iwasawa main conjecture \cite{kato-survey}.

In \cite[Thm. B]{kurihara-munster}, Kurihara obtained Theorem \ref{thm:refined-bsd} 
for elliptic curves with good ordinary reduction
via the theory of Kolyvagin systems of Gauss sum type under various assumptions including $\mu = 0$ for $p$-adic $L$-functions and the non-degeneracy of the $p$-adic height pairing.

His argument is completely different from ours, and this difference can be compared with the difference between cyclotomic unit Kolyvagin systems and Kolyvagin systems of Gauss sums in classical Iwasawa theory.
In \cite{kurihara-munster}, Kurihara explicitly constructed a certain skew-Hermitian relation matrix to present the structure of Selmer groups and the matrix is essentially equivalent to the organizing matrix constructed by Mazur--Rubin \cite{mazur-rubin-organizing}.
In our approach, we do not construct such matrices, and no $p$-adic height pairing is involved.
Recently, Kurihara also obtained the same result for the supersingular reduction case with $a_p(E) = 0$ under similar assumptions \cite{kurihara-analytic-quantities}.
In this sense, Theorem \ref{thm:refined-bsd} generalizes Kurihara's results by removing all the serious Iwasawa-theoretic assumptions.

\section{Kolyvagin systems for elliptic curves} \label{sec:kolyvagin-systems}
We summarize some materials in \cite{mazur-rubin-book} with an emphasis on the structure of fine Selmer groups of elliptic curves.

Let $E$ be an elliptic curve over $\mathbb{Q}$ and $p \geq 5$ a prime such that $\overline{\rho}$ is surjective.
Let $T$ be the $p$-adic Tate module of $E$.

\subsection{Selmer structures and Selmer groups}
\subsubsection{}
Following \cite[$\S$1]{mazur-rubin-book}, we recall the standard local conditions for Selmer groups of $T/p^kT$ (with $k \geq 1$) and review their properties. Since $p$ is odd, we ignore the infinite place.

Let $K$ be a non-archimedean local field. We recall the local conditions we use.
\begin{itemize}
\item the finite condition:
$\mathrm{H}^1_{f}(K, T/p^kT) = E(K) \otimes \mathbb{Z}/p^k\mathbb{Z}$ via the Kummer map.
\item the relaxed condition:
$\mathrm{H}^1_{\mathrm{rel}}(K, T/p^kT) =  \mathrm{H}^1(K, T/p^kT)$.
\item the strict condition:
$\mathrm{H}^1_{\mathrm{str}}(K, T/p^kT)  = 0$.
\item the unramified condition:
$\mathrm{H}^1_{\mathrm{ur}}(K, T/p^kT) =  \mathrm{H}^1(K^{\mathrm{ur}}/K, \mathrm{H}^0(K^{\mathrm{ur}}, T/p^kT ) )$.
\item the transverse condition (when $K = \mathbb{Q}_\ell$ and $\ell \equiv 1 \pmod{p^k}$):
$$\mathrm{H}^1_{\mathrm{tr}}(\mathbb{Q}_\ell, T/p^kT) = \mathrm{H}^1(\mathbb{Q}_\ell(\zeta_\ell)/\mathbb{Q}_\ell, \mathrm{H}^0(\mathbb{Q}_\ell(\zeta_\ell), T/p^kT ) ).$$
\end{itemize}
The first four local conditions are also defined on $T$ and $E[p^\infty]$.
We write
$\mathrm{H}^1_{/f}(-) = \dfrac{\mathrm{H}^1(-)}{\mathrm{H}^1_{f}(-)}$.
If $T/p^kT$ is unramified as a representation of $G_K$,
then we have
$\mathrm{H}^1_{f}(K, T/p^kT) = \mathrm{H}^1_{\mathrm{ur}}(K, T/p^kT)$.

\subsubsection{}
The \textbf{Selmer structure $\mathcal{F}$ on $T/p^kT$} consists of
\begin{itemize}
\item a finite set $\Sigma$ of places of $\mathbb{Q}$ containing the primes where $E$ has bad reduction, $p$, and the infinite place, and
\item the choices of local conditions $\mathrm{H}^1_{\mathcal{F}}(\mathbb{Q}_\ell, T/p^kT)$ at primes in $\Sigma$.
\end{itemize}
For a prime $q \not \in \Sigma$, we fix $\mathrm{H}^1_{\mathcal{F}}(\mathbb{Q}_q, T/p^kT)  = \mathrm{H}^1_{f}(\mathbb{Q}_q, T/p^kT)$.

For $\ell \in \mathcal{N}_k$,  we have isomorphism
$\phi^{\mathrm{fs}}_\ell : \mathrm{H}^1_{f}(\mathbb{Q}_\ell, T/p^kT) \simeq \mathrm{H}^1_{/f}(\mathbb{Q}_\ell, T/p^kT)$
and we choose a generator of $\mathrm{Gal}(\mathbb{Q}(\zeta_\ell)\mathbb{Q})$ for this isomorphism.
We also identify
$\mathrm{H}^1_{/f}(\mathbb{Q}_\ell, T/p^kT) = \mathrm{H}^1_{\mathrm{tr}}(\mathbb{Q}_\ell, T/p^kT)$.
See \cite[$\S$1.2]{mazur-rubin-book} for details.

\subsubsection{}
For a given Selmer structure $\mathcal{F}$ on $T/p^kT$ and $n \in \mathcal{N}_k$, 
the Selmer structure $\mathcal{F}(n)$ is defined by
\begin{itemize}
\item $\mathrm{H}^1_{\mathcal{F}(n)}(\mathbb{Q}_\ell, T/p^kT) = \mathrm{H}^1_{\mathcal{F}}(\mathbb{Q}_\ell, T/p^kT)$ for $\ell$ not dividing $n$, and
\item $\mathrm{H}^1_{\mathcal{F}(n)}(\mathbb{Q}_\ell, T/p^kT) = \mathrm{H}^1_{\mathrm{tr}}(\mathbb{Q}_\ell, T/p^kT)$ for $\ell$ dividing $n$.
\end{itemize}
\subsubsection{}
The \textbf{Selmer group $\mathrm{Sel}_{\mathcal{F}}(\mathbb{Q}, T/p^kT)$ of $T/p^kT$} with respect to $\mathcal{F}$ is defined by the exact sequence
$$
0 \to \mathrm{Sel}_{\mathcal{F}}(\mathbb{Q}, T/p^kT) \to
\mathrm{H}^1(\mathbb{Q}_{\Sigma}/\mathbb{Q}, T/p^kT) \to \bigoplus_{\ell \in \Sigma}
\dfrac{ \mathrm{H}^1(\mathbb{Q}_\ell, T/p^kT) }{ \mathrm{H}^1_{\mathcal{F}}(\mathbb{Q}_\ell, T/p^kT) }$$
where $\mathbb{Q}_{\Sigma}$ is the maximal extension of $\mathbb{Q}$ unramified outside $\Sigma$.
This definition is independent of the choice of $\Sigma$.

\subsubsection{}
Let $(-)^* = \mathrm{Hom}( -, \mu_{p^\infty})$ be the Cartier dual 
and then we have $(T/p^kT)^* \simeq E[p^k] \simeq T/p^kT$  via the Weil pairing.
The corresponding dual Selmer structure $\mathcal{F}^*$ on $(T/p^kT)^*$ is defined by the choices of local conditions $\mathrm{H}^1_{\mathcal{F}^*}(\mathbb{Q}_\ell, (T/p^kT)^*) = \mathrm{H}^1_{\mathcal{F}}(\mathbb{Q}_\ell, T/p^kT)^\perp$ under the local Tate pairing with the same $\Sigma$.
The Selmer group $\mathrm{Sel}_{\mathcal{F}^*}(\mathbb{Q}, E[p^k])$ with respect to $\mathcal{F}^*$ is defined in a similar way.

\subsubsection{}
We recall two natural Selmer structures on $T/p^kT$.

The \textbf{classical Selmer structure $\mathcal{F}_{\mathrm{cl}}$} is defined by
$\mathrm{H}^1_{\mathcal{F}_{\mathrm{cl}}}(\mathbb{Q}_\ell, T/p^kT) = \mathrm{H}^1_{f}(\mathbb{Q}_\ell, T/p^kT)$, the image of of $E(\mathbb{Q}_\ell)/p^k E(\mathbb{Q}_\ell)$ under the Kummer map
for every prime $\ell$ so that we have
\[
\xymatrix{
\mathrm{Sel}(\mathbb{Q}, T/p^kT) = \mathrm{Sel}_{\mathcal{F}_{\mathrm{cl}}}(\mathbb{Q}, T/p^kT), & \mathrm{Sel}(\mathbb{Q}, E[p^k]) = \mathrm{Sel}_{\mathcal{F}^*_{\mathrm{cl}}}(\mathbb{Q}, E[p^k]).
}
\]
Thus, the classical Selmer structure recovers classical Selmer groups.

The \textbf{canonical Selmer structure $\mathcal{F}_{\mathrm{can}} $}
is defined by 
$\mathrm{H}^1_{\mathcal{F}_{\mathrm{can}}}(\mathbb{Q}_\ell, T/p^kT) = \mathrm{H}^1_{f}(\mathbb{Q}_\ell, T/p^kT)$ for every prime $\ell \neq p$
and 
$\mathrm{H}^1_{\mathcal{F}_{\mathrm{can}}}(\mathbb{Q}_p, T/p^kT) = \mathrm{H}^1(\mathbb{Q}_p, T/p^kT)$.
We write
\[
\xymatrix{
\mathrm{Sel}_{\mathrm{rel}}(\mathbb{Q}, T/p^kT) = \mathrm{Sel}_{\mathcal{F}_{\mathrm{can}}}(\mathbb{Q}, T/p^kT), & \mathrm{Sel}_0(\mathbb{Q}, E[p^k]) = \mathrm{Sel}_{\mathcal{F}^*_{\mathrm{can}}}(\mathbb{Q}, E[p^k]).
}
\]
In other words, the canonical Selmer structure defines the $p$-relaxed Selmer group and the dual canonical Selmer structure defines the $p$-strict Selmer group.

For $n \in \mathcal{N}_k$, we write
$\mathrm{Sel}_{\mathrm{rel}, n} = \mathrm{Sel}_{\mathcal{F}_{\mathrm{can}}(n)}$,
$\mathrm{Sel}_{n} = \mathrm{Sel}_{\mathcal{F}_{\mathrm{cl}}(n)}$, and
$\mathrm{Sel}_{0, n} = \mathrm{Sel}_{\mathcal{F}^*_{\mathrm{can}}(n)}$ for convenience.

When $n =1$, the compact Selmer groups of $T$ are defined by the projective limits of Selmer groups of $T/p^kT$, and the discrete Selmer groups of $E[p^\infty]$ are defined by the direct limits of Selmer groups of $E[p^k]$.

\subsection{Euler and Kolyvagin systems}
\subsubsection{} \label{subsubsec:euler-systems}
We recall the notion of Euler systems following \cite[Def. 3.2.2]{mazur-rubin-book}.
Let $\mathcal{P} = \mathcal{P}_1$ be the set of primes defined in $\S$\ref{subsubsec:kolyvagin_primes} and $\mathcal{K}$ a possibly infinite abelian extension of $\mathbb{Q}$.
An \textbf{Euler system $\mathbf{z}$ for $(T, \mathcal{K}, \mathcal{P})$} is a collection of cohomology classes
$$\mathbf{z} = \left\lbrace z_F \in \mathrm{H}^1(F, T) :  F/\mathbb{Q} \textrm{ finite}, F \subseteq \mathcal{K} \right\rbrace$$
such that whenever $F'/F$ are finite subextenions of $\mathbb{Q}$ in $\mathcal{K}$,
satisfying the norm relation
$$\mathrm{Nm}_{F'/F}(z_{F'}) = \left( \prod_{\ell} P_\ell(\mathrm{Fr}^{-1}_\ell) \right) \cdot z_{F}$$
where $\mathrm{Nm}_{F'/F}$ is the corestriction map from $F'$ to $F$,
the product runs over primes in $\mathcal{P}$ which are ramified in $F'/\mathbb{Q}$ but not in $F/\mathbb{Q}$, 
$P_\ell(X) = \mathrm{det}(1 - \mathrm{Fr}_\ell \cdot X | T)$, and $\mathrm{Fr}_\ell$ is the arithmetic Froenius at $\ell$.
Let $\ES(T) = \ES(T, \mathcal{K}, \mathcal{P})$ denote the module of Euler systems over $\mathbb{Z}_p\llbracket \mathrm{Gal}(\overline{\mathbb{Q}}/\mathbb{Q}) \rrbracket$.

\subsubsection{} \label{subsubsec:kolyvagin-systems}
A \textbf{Kolyvagin system $\ks$ for $(T, \mathcal{F}, \mathcal{P})$}
is a collection of cohomology classes
$$\ks = \left\lbrace 
\kappa_n \in \mathrm{Sel}_{\mathcal{F}(n)}(\mathbb{Q}, T/I_nT) : n \in \mathcal{N}_1 
\right\rbrace$$
such that
\begin{equation} \label{eqn:axiom-kolyvagin-systems}
\mathrm{loc}^s_\ell ( \kappa_{n\ell} )
= \phi^{\mathrm{fs}}_\ell \circ \mathrm{loc}_\ell (\kappa_n) \in \mathrm{H}^1_{/f}(\mathbb{Q}_\ell, T/I_{n\ell}T)
\end{equation}
 for $\ell \in \mathcal{N}_1$ with $(n,\ell) =1$ where 
$ \mathrm{loc}_\ell : \mathrm{H}^1(\mathbb{Q}, T/I_nT) \to \mathrm{H}^1(\mathbb{Q}_\ell, T/I_nT)$, and
$\mathrm{loc}^s_\ell : \mathrm{H}^1(\mathbb{Q}, T/I_nT) \to \mathrm{H}^1_{/f}(\mathbb{Q}, T/I_nT)$.
Let $\KS(T) = \KS(T, \mathcal{F}, \mathcal{P})$ denote the module of Kolyvagin systems
and $\KSbar(T) = \KSbar(T, \mathcal{F}, \mathcal{P})$ denote the generalized module of Kolyvagin systems defined by the completion. See \cite[$\S$3.1]{mazur-rubin-book} for details.

Our convention of Kolyvagin systems depends on the choice of generators of $\mathrm{Gal}(\mathbb{Q}(\zeta_\ell)/\mathbb{Q})$ for each prime $\ell$ dividing $n$, and it corresponds to the choice of the primitive roots in the definition of Kurihara numbers.

The natural map $\KS(T) \to \KSbar(T)$ is an isomorphism when the core rank is one ($\S$\ref{subsubsec:core-rank} and Theorem \ref{thm:core-rank-elliptic-curves}). See \cite[Cor. 4.5.3, Prop. 5.2.9, and Prop. 6.2.2]{mazur-rubin-book} for details.
\subsubsection{}
Under our working hypotheses, all the conditions in Theorem \ref{thm:euler-to-kolyvagin} below are satisfied.
\begin{thm}[Mazur--Rubin] \label{thm:euler-to-kolyvagin}
Suppose that $\mathcal{K}$ contains the maximal abelian $p$-extension of $\mathbb{Q}$ which is unramified outside $p$ and $\mathcal{P}$, and
\begin{enumerate}
\item $T/(\mathrm{Fr}_\ell - 1)T$ is a cyclic $\mathbb{Z}_p$-module for every $\ell \in \mathcal{P}$,
\item $\mathrm{Fr}^{p^k}_\ell - 1$ is injective on $T$ for every $\ell \in \mathcal{P}$ and every $k \geq 0$. 
\end{enumerate}
Then there exists a canonical homomorphism $\ES(T, \mathcal{K}, \mathcal{P}) \to \KSbar(T, \mathcal{F}_{\mathrm{can}}, \mathcal{P})$
sending $\mathbf{z}$ to $\ks$
such that
$\kappa_1 = z_{\mathbb{Q}}$.
If we further assume that $\mathrm{H}^0(\mathbb{Q}_p, E[p^\infty])$ is a divisible $\mathbb{Z}_p$-module, then $\KSbar(T, \mathcal{F}_{\mathrm{can}}, \mathcal{P})$ can be replaced by $\KS(T, \mathcal{F}_{\mathrm{can}}, \mathcal{P})$.
\end{thm}
\begin{proof}
See \cite[Thm. 3.2.4 and $\S$6.2]{mazur-rubin-book}.
\end{proof}

\subsection{Kato's Euler systems and Kato's Kolyvagin systems}
\subsubsection{} \label{subsubsec:kato-euler-systems}
Let $\mathbf{z}^{\mathrm{Kato}} = (z^{\mathrm{Kato}}_{F})_F$ be Kato's Euler system for $E$ associated to the real N\'{e}ron period $\Omega^+_E$
where $z^{\mathrm{Kato}}_F \in \mathrm{H}^1(F, T)$ and
$F$ runs over abelian extensions of $\mathbb{Q}$ following the convention of \cite[Thm. 6.1]{kataoka-thesis}.
More explicitly, each $z^{\mathrm{Kato}}_F \in \mathrm{H}^1(F, T)$ is characterized by the interpolation formula
$$\sum_{\sigma \in \mathrm{Gal}(F/\mathbb{Q})}\sigma \left( \mathrm{exp}^* \circ \mathrm{loc}^s_p (z^{\mathrm{Kato}}_F) \right) \cdot \chi(\sigma) = \dfrac{ L_{(Sp)}(E, \chi, 1) }{ \Omega^+_E } \cdot \omega_E$$
where $\chi$ is any even character of $\mathrm{Gal}(F/\mathbb{Q})$, $S$ is the product of the ramified primes of $F/\mathbb{Q}$, $L_{(Sp)}(E, \chi, 1)$ is the $Sp$-imprimitive $\chi$-twisted $L$-value of $E$ at $s=1$, 
and $\omega_E$ is the N\'{e}ron differential. For odd characters of $\mathrm{Gal}(F/\mathbb{Q})$, the interpolation formula is zero. 
Kato's Kolyvagin system $\ks^{\mathrm{Kato}}$ is defined by the image of $\mathbf{z}^{\mathrm{Kato}}$ under the map in Theorem \ref{thm:euler-to-kolyvagin}.

\subsubsection{}
For each $n \in \mathcal{N}_1$, the Kolyvagin derivative operator at $n$ is defined by
$D_{\mathbb{Q}(\zeta_n)}  = \prod_{\ell \vert n} \sum_{i=1}^{\ell-2} i \cdot \sigma^i_{\eta_\ell}$
where $\eta_\ell$ is a primitive root mod $\ell \in \mathcal{P}_1$.
Then each $\kappa^{\mathrm{Kato}}_n$ is defined by the image of $D_{\mathbb{Q}(\zeta_n)} z^{\mathrm{Kato}}_{\mathbb{Q}(\zeta_n)} \in \mathrm{H}^1(\mathbb{Q}(\zeta_n), T)$ in $\mathrm{H}^1(\mathbb{Q}, T/I_nT)$. 
We do not need any modification of the image of $D_{\mathbb{Q}(\zeta_n)} z^{\mathrm{Kato}}_{\mathbb{Q}(\zeta_n)}$ to obtain $\kappa^{\mathrm{Kato}}_n$ (cf. \cite[(33) in pp. 80]{mazur-rubin-book}).

\subsection{Properties of Kolyvagin systems}
\subsubsection{}
When we work with Kolyvagin systems, we always assume that $p \geq 5$ and $\overline{\rho}$ is surjective.
This is strong enough to satisfy all the working hypotheses for the theory of Kolyvagin systems \cite[\S3.5 and Lem. 6.2.3]{mazur-rubin-book}.
More precisely, the $p \geq 5 $ condition is used only in Proposition \ref{prop:chebotarev} below and its consequences. The $p=3$ case is studied by Sakamoto recently \cite{sakamoto-p-3}.

Also, all the argument works when $\overline{\rho}$ is irreducible and there exists a prime $\ell$ exactly dividing the conductor of $E$ such that $\overline{\rho}$ is ramified at $\ell$ \cite[$\S$2.5]{skinner-pacific}. The absolute irreducibility and the irreducibility of $\overline{\rho}$ are equivalent for the case of elliptic curves over $\mathbb{Q}$ \cite[Lem. 5]{rubin-modularity-mod-5}.

\begin{prop}[Mazur--Rubin] \label{prop:chebotarev}
Assume that $p \geq  5$ and $\overline{\rho}$ is surjective.
Let $c_1, c_2 \in \mathrm{H}^1(\mathbb{Q}, T/p^kT)$ and $c_3, c_4 \in \mathrm{H}^1(\mathbb{Q},E[p^k])$ be non-zero elements.
For every $k \in \mathbb{Z}_{>0}$, there exists a set $S \subseteq \mathcal{P}_k$ of positive density such that
for every $\ell \in S$, the localizations $\mathrm{loc}_\ell (c_i)$ are all non-zero. 
\end{prop} 
\begin{proof}
See \cite[Prop. 3.6.1]{mazur-rubin-book}.
\end{proof}
Following \cite[$\S$3.6 and Thm. 4.4.1]{mazur-rubin-book}, a Kolyvagin prime $\ell \in \mathcal{P}_k$ is said to be \textbf{useful for (non-zero) $\kappa_n$} with $n \in \mathcal{N}_k$ if $(\ell ,n)=1$ and $\mathrm{loc}_\ell(\kappa_n) \neq 0$.
\begin{prop} \label{prop:non-vanishing-kappa-n}
Assume that $p \geq 5$ and $\overline{\rho}$ is surjective.
Then there are infinitely many useful primes for a given non-zero $\kappa_n$.
In particular, if $\kappa_n \neq 0$ and $\ell$ is a useful prime for $\kappa_n$, then $\kappa_{n\ell} \neq 0$. 
\end{prop}
\begin{proof}
This follows from Proposition \ref{prop:chebotarev} and (\ref{eqn:axiom-kolyvagin-systems}).
\end{proof}

\subsubsection{} \label{subsubsec:core-rank}
Following \cite[$\S$4]{mazur-rubin-book}, we write
\begin{align*}
 \lambda(n, E[I_n]) & = \mathrm{length}_{\mathbb{Z}_p} \mathrm{Sel}_{0, n}(\mathbb{Q}, E[I_n]), \\
 \lambda(n, T/I_nT)  & = \mathrm{length}_{\mathbb{Z}_p} \mathrm{Sel}_{\mathrm{rel}, n}(\mathbb{Q}, T/I_nT)
\end{align*}
where $n \in \mathcal{N}_1$.
We say $n  \in \mathcal{N}_1$ is a \textbf{core vertex} if $\lambda(n, E[I_n])$ or $\lambda(n, T/I_nT)$ is zero \cite[Def. 4.1.8]{mazur-rubin-book} and
the \textbf{core rank $\chi(T)$} (with the canonical Selmer structure) is defined by $\mathrm{rk}_{ \mathbb{Z}_p/I_n \mathbb{Z}_p } \mathrm{Sel}_{\mathrm{rel}, n}(\mathbb{Q}, T/I_nT)$ for any core vertex $n$ \cite[Def. 4.1.11 and Def. 5.2.4]{mazur-rubin-book}.

\subsubsection{}
The following lemma is important for our proof of the main theorem.
\begin{lem} \label{lem:surjectivity-at-ell}
Let $n \in \mathcal{N}_k$ and $\ell \in \mathcal{P}_k$ with $(n, \ell) = 1$, and  `$\ell\textrm{-str}$' denotes the strict local condition at $\ell$.
Let $\mathcal{F} = \mathcal{F}_{\mathrm{cl}}$ or $\mathcal{F}_{\mathrm{can}}$.
If the localization map $\mathrm{Sel}_{\mathcal{F}(n)}(\mathbb{Q}, T/p^kT) \to E(\mathbb{Q}_\ell) \otimes  \mathbb{Z}_p /p^k \mathbb{Z}_p $ is surjective, then we have
$\mathrm{Sel}_{\mathcal{F}^*(n\ell)}(\mathbb{Q}, E[p^k]) = \mathrm{Sel}_{\mathcal{F}(n),\ell\textrm{-}\mathrm{str}}(\mathbb{Q}, E[p^k])$.
\end{lem}
\begin{proof}
See \cite[Lem. 4.1.7]{mazur-rubin-book}.
\end{proof}

\subsubsection{}
The following theorem plays an important role to have the equivalence between the non-triviality of $\ks^{\mathrm{Kato}}$ and the non-vanishing of $\kn$ (Proposition
\ref{prop:ks-kn-equi-non-vanishing}).
\begin{thm} \label{thm:core-rank-elliptic-curves}
Let $E$ be an elliptic curve over $\mathbb{Q}$ and $p\geq 5$ a prime such that $\overline{\rho}$ is surjective.
Let $T$ be the $p$-adic Tate module of $E$.
\begin{enumerate}
\item $\KS(T, \mathcal{F}_{\mathrm{cl}}, \mathcal{P}) = 0$.
\item $\KS(T, \mathcal{F}_{\mathrm{can}}, \mathcal{P})$ is free of rank one over $\mathbb{Z}_p$.
\end{enumerate}
\end{thm}
\begin{proof}
See \cite[Thm. 4.2.2, Thm. 5.2.10, and Prop. 6.2.2]{mazur-rubin-book}.
The first case lies in the core rank zero case and the second case lies in the core rank one case.
\end{proof}
Theorem \ref{thm:core-rank-elliptic-curves} explains why Kato's Kolyvagin systems cannot control classical Selmer groups directly.

\subsubsection{}
We discuss the precise location of Kolyvagin system classes in $n$-transverse $p$-relaxed Selmer groups.
\begin{thm} \label{thm:splitting-mazur-rubin}
For every $k \geq 1$ and $n \in \mathcal{N}_k$, there exists a non-canonical isomorphism
\begin{equation} \label{eqn:sel-sel0-decomposition}
\mathrm{Sel}_{\mathrm{rel},n}(\mathbb{Q}, T/I_nT ) \simeq \mathbb{Z}_p/I_n\mathbb{Z}_p \oplus \mathrm{Sel}_{0,n}(\mathbb{Q}, E[I_n] ) .
\end{equation}
\end{thm}
\begin{proof}
See \cite[Thm. 4.1.13.(i)]{mazur-rubin-book} and \cite[Thm. 5.2.5]{mazur-rubin-book}.
\end{proof}

For each $n \in \mathcal{N}_1$, write
\[
\xymatrix{
 \mathcal{H}(n)  = \mathrm{Sel}_{\mathrm{rel}, n}(\mathbb{Q}, T/I_nT), & \mathcal{H}'(n)  =  p^{ \lambda(n, E[I_n])}\mathrm{Sel}_{\mathrm{rel}, n}(\mathbb{Q}, T/I_nT),
}
\]
and the latter is said to be the \textbf{stub Selmer submodule at $n$}.
\begin{thm} \label{thm:stub-selmer}
For every $n \in \mathcal{N}_1$, $\kappa_n \in \mathcal{H}'(n)$.
\end{thm}
\begin{proof}
See \cite[Thm. 4.4.1]{mazur-rubin-book} with Theorem \ref{thm:core-rank-elliptic-curves}.
\end{proof}

\begin{rem} \label{rem:stub-selmer}
Theorem \ref{thm:stub-selmer} says that
$\kappa_n \in \mathrm{Sel}_{\mathrm{rel},n}(\mathbb{Q}, T/I_nT ) $
indeed lies in $p^{ \lambda(n, E[I_n])} \mathbb{Z}_p/I_n\mathbb{Z}_p \subseteq \mathbb{Z}_p/I_n\mathbb{Z}_p$, the first factor in the decomposition (\ref{eqn:sel-sel0-decomposition}).
\end{rem}
The following proposition illustrates the precise location of $\kappa_n$ in $\mathrm{Sel}_{\mathrm{rel}, n}(\mathbb{Q}, T/I_nT)$.
\begin{prop} \label{prop:kolyvagin-system-location}
Suppose that $\kappa_n \neq 0$ for some $n \in \mathcal{N}_1$.
Let $j \geq 0$ such that $\kappa_n$ generates $p^j\mathcal{H}'(n)$.
Then $\kappa_{n'}$ generates $p^j\mathcal{H}'(n') = p^{j+\lambda(n', E[I_{n'}])}\mathcal{H}(n') \simeq p^{j+\lambda(n', E[I_{n'}])} \mathbb{Z}_p/I_{n'}\mathbb{Z}_p$ for all $n'  \in \mathcal{N}_1$.
\end{prop}
\begin{proof}
See \cite[Cor. 4.5.2.(ii)]{mazur-rubin-book}.
\end{proof}
We say that a Kolyvagin system \textbf{$\ks$ is primitive} if the mod $p$ Kolyvagin system $\ks^{(1)}  = \left( \kappa_n \pmod{p} \right)_{n \in \mathcal{N}_1}$ is non-zero. This is equivalent to $j=0$ in Proposition \ref{prop:kolyvagin-system-location}. 
See \cite[Cor. 4.5.4 and Def. 4.5.5]{mazur-rubin-book}.

\subsection{Kolyvagin systems over $\mathbb{Z}_p$}
\subsubsection{}
Let $\ks^{(k)}$ be a Kolyvagin system over $T/p^kT$.
Suppose that $\kappa^{(k)}_n \neq 0$ for some $n \in \mathcal{N}_1$.
Following \cite[Ex. 3.1.12]{mazur-rubin-book}, we have another Kolyvagin system $\ks^{n, (k)}$
defined by 
$\kappa^{n, (k)}_m = \kappa^{(k)}_{n \cdot m}$
where $m \in \mathcal{N}_k$ with $(m,n)=1$.
The following proposition is fundamental to investigate the structure for $p$-strict Selmer groups.
\begin{prop} \label{prop:structure-fine-selmer-p^k}
Suppose that $\mathrm{ord}(\ks^{(k)}) < \infty$, i.e $\kappa^{(k)}_n \neq 0$ for some $n \in \mathcal{N}_1$.
Write
$$\mathrm{Sel}_0(\mathbb{Q}, E[p^k]) \simeq \bigoplus_{i \geq 1}  \mathbb{Z} / p^{d_i} \mathbb{Z}$$
with non-negative integers
$d_1 \geq d_2 \geq \cdots $, and fix $j \geq 0$ such that
$\kappa^{(k)}_n$ generates $p^j \mathcal{H}'(n)$ (Proposition \ref{prop:kolyvagin-system-location}).
Then for every integer $r \geq 0$, 
$$\partial^{(r)} (\ks^{(k)}) = \mathrm{min}\lbrace k, j+\sum_{i > r} d_i \rbrace .$$
\end{prop}
\begin{proof}
See \cite[Prop. 4.5.8]{mazur-rubin-book}
\end{proof}
\subsubsection{}
Let $\ks$ be a Kolyvagin system for $T$.
Define 
$$\partial^{(0)} (\ks) = \mathrm{max} \left\lbrace j : \kappa_1 \in p^j \mathrm{Sel}_{\mathrm{rel}}(\mathbb{Q}, T) \right\rbrace $$
and we allow $\partial^{(0)} (\ks)  = \infty$ (when $\kappa_1=0$).

\begin{thm} \label{thm:euler-system-divisibility}
Let $\ks$ be a Kolyvagin system for $T$. Then
$\mathrm{length}_{\mathbb{Z}_p} \mathrm{Sel}_0(\mathbb{Q}, E[p^\infty]) \leq \partial^{(0)} (\ks)$.
In particular, if $\kappa_1 \neq 0$, then $\mathrm{Sel}_0(\mathbb{Q}, E[p^\infty])$ is finite.
\end{thm}
\begin{proof}
See \cite[Thm. 5.2.2]{mazur-rubin-book}.
\end{proof}

\begin{defn}
\begin{enumerate}
\item 
The \textbf{vanishing order} of a non-zero Kolyvagin system $\ks = ( \kappa_n )_{n \in \mathcal{N}_1}$ is defined by
$\mathrm{ord}(\ks) = \mathrm{min} \left\lbrace \nu(n) : n \in \mathcal{N}_1 , \kappa_n \neq 0 \right\rbrace$.
\item 
For $r \in \mathbb{Z}_{>0}$, define
\begin{align*}
& \partial^{(r)}(\ks)  \\
 = \ &
\mathrm{min} \left\lbrace \mathrm{max} \left\lbrace j_n : \kappa_n \in p^{j_n} \mathrm{Sel}_{\mathrm{rel},n}(\mathbb{Q}, T/I_nT) \right\rbrace : n \in \mathcal{N}_1 \textrm{ with } \nu(n) = r  \right\rbrace \\
 = \ & \lim_{k \to \infty} \mathrm{min} \left\lbrace k, \mathrm{max} \left\lbrace j_n : \kappa^{(k)}_n \in p^{j_n} \mathrm{Sel}_{\mathrm{rel},n}(\mathbb{Q}, T/p^kT)  \right\rbrace : n \in \mathcal{N}_k \textrm{ with } \nu(n) = r  \right\rbrace
\end{align*}
and the second equality follows from Theorem \ref{thm:kato-kolyvagin-main}.(1) below.
\item 
The \textbf{sequence of elementary divisors} is defined by
$e_i(\ks) = \partial^{(i)}(\ks) - \partial^{(i+1)}(\ks)$
where $i \geq \mathrm{ord}(\ks)$.
\end{enumerate}
\end{defn}
\subsubsection{}
The following theorem illustrates how a non-trivial Kolyvagin system determines the structure of $p$-strict Selmer groups.
\begin{thm} \label{thm:kato-kolyvagin-main}
Let  $\ks$ be a non-trivial Kolyvagin system for $T$. Then the following statements hold.
\begin{enumerate}
\item For every $s \geq 0$,
$\partial^{(s)}(\ks) = \lim_{k \to \infty} \partial^{(s)}(\ks^{(k)}) $.
\item The sequence $\partial^{(s)}(\ks)$ is non-increasing, and finite for $s \geq \mathrm{ord}(\ks)$.
\item The sequence $e_i(\ks)$ is non-increasing, non-negative, and finite for $i  \geq \mathrm{ord}(\ks)$.
\item $\mathrm{ord}(\ks)$ and the $e_i(\ks)$ are independent of the choice of non-zero $\ks \in \KS(T)$.
\item $\mathrm{cork}_{\mathbb{Z}_p} \mathrm{Sel}_0(\mathbb{Q},E[p^\infty]) = \mathrm{ord}(\ks)$.
\item 
$\mathrm{Sel}_0(\mathbb{Q}, E[p^\infty])_{/\mathrm{div}} \simeq \bigoplus_{i \geq \mathrm{ord}(\ks)} \dfrac{\mathbb{Z}_p}{p^{e_i(\ks)}\mathbb{Z}_p}$
\item  $\mathrm{length}_{\mathbb{Z}_p} \mathrm{Sel}_0(\mathbb{Q}, E[p^\infty])_{/\mathrm{div}} =
\partial^{(\mathrm{ord}(\ks))}(\ks) - \partial^{(\infty)}(\ks)$
\item $\ks$ is primitive if and only if $\partial^{(\infty)}(\ks) = 0$.
\end{enumerate}
\end{thm}
\begin{proof}
See \cite[Thm. 5.2.12]{mazur-rubin-book}.
\end{proof}
\begin{cor} \label{cor:kato-kolyvagin-main}
Let  $\ks$ be a non-trivial Kolyvagin system for $T$. 
\begin{enumerate}
\item $\mathrm{length}_{\mathbb{Z}_p} \mathrm{Sel}_0(\mathbb{Q}, E[p^\infty])$ is finite if and only if $\kappa_1 \neq 0$.
\item $\mathrm{length}_{\mathbb{Z}_p} \mathrm{Sel}_0(\mathbb{Q}, E[p^\infty]) \leq \partial^{(0)} (\ks)$, with equality if and only if $\ks$ is primitive. 
\end{enumerate}
\end{cor}
\begin{proof}
See \cite[Cor. 5.2.13]{mazur-rubin-book}.
\end{proof}

\section{A refined explicit reciprocity law and Kurihara numbers} \label{sec:kurihara-numbers}
The goal of this section is to describe the precise connection between $\ks^{\mathrm{Kato}}$ and $\kn$.
We first investigate when the local torsion $E(\mathbb{Q}_p)[p]$ is trivial.
Then we compute the integral image of the Bloch--Kato dual exponential map.
Using this computation, we extend the the Bloch--Kato dual exponential map to torsion coefficients.
By using the compatibility between the classical Bloch--Kato dual exponential map and the torsion Bloch--Kato dual exponential map, we obtain the derivative of Kato's explicit reciprocity law.
It significantly refines the computation in \cite{kks}.

\subsection{Local torsions of elliptic curves: a digression} \label{subsec:local-torsion}
\subsubsection{}
Consider exact sequence
\begin{equation} \label{eqn:elliptic-curve-mod-p}
\xymatrix{
0 \ar[r] & \widehat{E}(p\mathbb{Z}_p) \ar[r] & E(\mathbb{Q}_p) \ar[r] & \widetilde{E}(\mathbb{F}_p) \ar[r] & 0 
}
\end{equation}
where $\widehat{E}$ is the formal group associated to $E$ and $\widetilde{E}$ is the mod $p$ reduction of $E$.
Since $\widehat{E}(p\mathbb{Z}_p)$ is isomorphic to $\mathbb{Z}_p$ as additive groups, $E(\mathbb{Q}_p)[p]$ is non-trivial if and only if $\widetilde{E}(\mathbb{F}_p)[p]$ is non-trivial and (\ref{eqn:elliptic-curve-mod-p}) splits.
Thus, if $E$ has non-anomalous good reduction at $p$, then $E(\mathbb{Q}_p)[p]$ is trivial.

Now we consider the case that $E$ has anomalous good reduction at $p$, equivalently, $\widetilde{E}(\mathbb{F}_p)[p]$ is non-trivial. Indeed, $a_p(E) \equiv 1 \pmod{p}$ implies $a_p(E) = 1$ if $p \geq  7$, and it is known that the set $\lbrace p : a_p(E) = 1 \rbrace$ has zero density \cite[Prop. 5.1]{greenberg-lnm}.

Suppose that $E(\mathbb{Q}_p)[p]$ is non-trivial, i.e. (\ref{eqn:elliptic-curve-mod-p}) splits.
By applying Gross' tameness criterion \cite[Prop. 13.2]{gross-tameness}, this is equivalent to that 
the image of 
$\overline{\rho} \vert_{\mathrm{Gal}(\overline{\mathbb{Q}}_p/\mathbb{Q}_p)}$
 is diagonalizable. 
  See \cite{bryden-cais-thesis} for the unchecked compatibility in \cite{gross-tameness}.

Recall that $f = \sum_{n \geq 1} a_n(E)q^n$ \textbf{admits a mod $p$ companion form} if
there exists a mod $p$ eigenform
$g = \sum_{n \geq 1} b_n q^n$ of weight $p-1$
such that
$n^2 \cdot b_n \equiv n \cdot a_n(E) \pmod{p}$
for all $n \geq 1$.
By \cite{gross-tameness}, the above diagonalizability is also equivalent to the existence of a mod $p$ companion form of $f$.
\subsubsection{}
The following result is known.
\begin{prop} \label{prop:local-torsions}
Let $E$ be an elliptic curve over $\mathbb{Q}_p$ with an odd prime $p$.
Then $E(\mathbb{Q}_p)[p] \neq  0$ if and only if
\begin{enumerate}
\item $E$ has anomalous good ordinary reduction at $p$ and admits a mod $p$ companion form,
\item $E$ has split multiplicative reduction at $p$ and the $j$-invariant satisfies $p \mid \mathrm{ord}_p(j)$, or
\item $E$ has additive reduction at $p$ with minimal Weierstrass equation
\begin{equation*}
y^2 + a_1 xy + a_3 y = x^3 + a_2 x^2 + a_4 x + a_6
\end{equation*}
where $a_i \in p\mathbb{Z}_p$ for each $i$ satisfying
\begin{enumerate}
\item $p = 3$ and $a_2 \equiv 6 \pmod{9}$,
\item $p = 5$ and $a_4 \equiv 10 \pmod{25}$, or
\item $p = 7$ and $a_6 \equiv 14 \pmod{49}$.
\end{enumerate} 
\end{enumerate}
\end{prop}
\begin{proof}
The first two cases are well-known. For the additive reduction case, see \cite[$\S$2]{kim-nakamura} and \cite{kosters-pannekoek}.
\end{proof}
\begin{rem}
There are several results on the non-triviality of $E(\mathbb{Q}_p)[p]$. See \cite[Lem. 5.2]{greenberg-lnm}, \cite[Prop. 2.1.(3)]{david-weston}, and \cite[Prop. 6]{ghate-2005}.
It is conjectured that there are only finitely many primes $p$ such that $E(\mathbb{Q}_p)[p] \neq 0$ when $E$ is a non-CM elliptic curve \cite[Conj. 1.1]{david-weston}, and its average version is proved in \cite[Thm. 1.2]{david-weston}.
\end{rem}

\subsection{The computation of the integral image of  Bloch--Kato's dual exponential map} \label{subsec:computing-integral-image}
We recall the computation of the integral image of Bloch--Kato's dual exponential map following \cite[Prop. 3.5.1]{rubin-book}, \cite[Lem. 14.18]{kato-euler-systems} and \cite[$\S$6 and $\S$7]{kks} with some refinement in order to handle the $E(\mathbb{Q}_p)[p^\infty] \neq 0$ case.
\subsubsection{}
Following \cite[$\S$3.5]{rubin-book}, fix a minimal Weierstrass model of $E$ over $\mathbb{Q}_p$ and $\omega_E$ is the corresponding holomorphic differential. Let
$\omega^*_E$ be the basis of the tangent space of the minimal Weierstrass model such that the natural pairing $\langle \omega^*_E ,  \omega_E \rangle = 1$.
Write $\mathrm{exp}^*_{\omega_E} (-) = \langle \omega^*_E , \mathrm{exp}^*(-) \rangle$ where $\mathrm{exp}^*$ is the Bloch--Kato dual exponential map.
We write
$$t = 
\left\lbrace
\begin{array}{ll}
 0 & \textrm{ if  $E$ has split multiplicative reduction at $p$,} \\
 \mathrm{length}_{\mathbb{Z}_p} \left( E(\mathbb{Q}_p)[p^\infty] \right)  & \textrm{ otherwise.}
\end{array} \right.$$
\subsubsection{}
We first consider the good reduction case.
\begin{lem} \label{lem:dual-exponential}
If $p > 2$, we have
$$\mathrm{exp}^*_{\omega_E} :
\dfrac{ \mathrm{H}^1(\mathbb{Q}_p, T) }{ E(\mathbb{Q}_p) \otimes \mathbb{Z}_p }  \simeq \dfrac{\#\widetilde{E}(\mathbb{F}_p)}{\#\mathrm{H}^0(\mathbb{Q}_p,E[p^\infty])} \cdot \dfrac{1}{p} \cdot \mathbb{Z}_p .$$
If $E$ has good reduction at $p > 2$, then we have
$$\mathrm{exp}^*_{\omega_E} :
\dfrac{ \mathrm{H}^1(\mathbb{Q}_p, T) }{ E(\mathbb{Q}_p) \otimes \mathbb{Z}_p }  \simeq  \dfrac{\#\widetilde{E}(\mathbb{F}_p)}{p^{1+t}} \cdot \mathbb{Z}_p = \dfrac{p-a_p(E)+1}{p^{1+t}} \cdot \mathbb{Z}_p  .$$
\end{lem}
\begin{proof}
See \cite[Prop. 3.5.1]{rubin-book} and \cite[Lem. 14.18]{kato-euler-systems}.
\end{proof}
Note that $\dfrac{p-a_p(E)+1}{p} = 1- a_p(E)p^{-1} + p^{-1}$ is the Euler factor at $p$. 
\begin{cor} \label{cor:dual-exponential-unramified}
Suppose that $E$ has good reduction at $p \geq 5$.
Let $K$ be a finite unramified extension of $\mathbb{Q}_p$ with residue field $k$.
We assume either
\begin{enumerate}
\item $p$ does not divide $\#\widetilde{E}(k)$, or
\item the corresponding modular form does not admit a mod $p$ companion form.
\end{enumerate} 
Then we have $E(K)[p] = 0$ and
$$\mathrm{exp}^*_{\omega_E} :
\dfrac{ \mathrm{H}^1(K, T) }{ E(K) \otimes \mathbb{Z}_p }  \simeq \dfrac{ \#\widetilde{E}(k)}{p} \cdot  \mathcal{O}_K .$$
\end{cor}
\begin{proof}
The first isomorphism in Lemma \ref{lem:dual-exponential} works for any finite unramified extension $K$.
The first case is immediate. For the second case, we know the restriction of $\overline{\rho}$ to $\mathrm{Gal}(\overline{\mathbb{Q}}_p/\mathbb{Q}_p)$ is wildly ramified thanks to \cite[Prop. 13.7]{gross-tameness}.
Since $K/\mathbb{Q}_p$ is unramified, the restriction of $\overline{\rho}$ to $\mathrm{Gal}(\overline{K}/K)$ is still wildly ramified.
Thus, it cannot be a direct sum of two characters.
Thus, we have $E(K)[p] = 0$ even when $\widetilde{E}(k)[p]$ is non-trivial.
\end{proof}
\begin{rem}
In \cite{otsuki, ota-thesis, kataoka-thesis}, the $E(K)[p] = 0$ assumption is used seriously in the construction of integral Mazur--Tate elements over tamely ramified abelian extensions $K$ of $\mathbb{Q}_p$.
\end{rem}
\subsubsection{}
We move to the bad reduction cases.
Consider the following exact sequences
\[
\xymatrix{
0 \ar[r] & E_1(\mathbb{Q}_p) \ar[r] & E(\mathbb{Q}_p) \ar[r] & \widetilde{E}(\mathbb{F}_p) \ar[r] & 0 \\
0 \ar[r] & E_1(\mathbb{Q}_p) \ar[r] \ar@{=}[u] & E_0(\mathbb{Q}_p) \ar[r] \ar@{^{(}->}[u] & \widetilde{E}_{\mathrm{ns}}(\mathbb{F}_p) \ar[r] \ar@{^{(}->}[u] & 0
}
\]
where 
 $\widetilde{E}(\mathbb{F}_p)$ is the mod $p$ reduction of $E(\mathbb{Q}_p)$ and $E_1(\mathbb{Q}_p)$ is the kernel of the mod $p$ reduction map.
 Also, $\widetilde{E}_{\mathrm{ns}}(\mathbb{F}_p)$ is the non-singular locus of $\widetilde{E}(\mathbb{F}_p)$ and $E_0(\mathbb{Q}_p)$ is the preimage of $\widetilde{E}_{\mathrm{ns}}(\mathbb{F}_p) $ in $E(\mathbb{Q}_p)$.
As in Lemma \ref{lem:dual-exponential}, 
if $p >2$, we have
$$\mathrm{exp}^*_{\omega_E} :
\dfrac{ \mathrm{H}^1(\mathbb{Q}_p, T) }{ E(\mathbb{Q}_p) \otimes \mathbb{Z}_p }  \simeq \dfrac{\widetilde{E}(\mathbb{F}_p)}{\#\mathrm{H}^0(\mathbb{Q}_p,E[p^\infty])} \cdot \dfrac{1}{p} \cdot \mathbb{Z}_p \simeq \dfrac{\widetilde{E}_{\mathrm{ns}}(\mathbb{F}_p)}{\#\mathrm{H}^0(\mathbb{Q}_p,E[p^\infty])} \cdot [ E(\mathbb{Q}_p) : E_0(\mathbb{Q}_p) ] \cdot \dfrac{1}{p} \cdot \mathbb{Z}_p .$$

\begin{thm} \label{thm:multiplicative-bad-good-index}
Let $E$ be an elliptic curve over $\mathbb{Q}_p$  (with minimal Weierstrass equation), and $\Delta$ the discriminant of $E$.
\begin{enumerate}
\item 
If $E$ admits split multiplicative reduction at $p$, then $ [ E(\mathbb{Q}_p) : E_0(\mathbb{Q}_p) ] = \mathrm{ord}_p(\Delta) = -\mathrm{ord}_p(j)$.
\item 
If $E$ admits non-split multiplicative reduction at $p$, then $ [ E(\mathbb{Q}_p) : E_0(\mathbb{Q}_p) ] \leq 2$.
\item Otherwise, $ [ E(\mathbb{Q}_p) : E_0(\mathbb{Q}_p) ] \leq 4$. 
\end{enumerate}
\end{thm}
\begin{proof}
It follows from Tate's algorithm \cite{tate-algorithm}. See \cite[Thm. VII.6.1]{silverman} and \cite[$\S$2.2]{lorenzini-torsion-tamagawa}.
\end{proof}
\begin{rem} \label{rem:multiplicative-bad-good-index}
When $E$ has additive reduction at $p$ with $p \leq 3$, then $ [ E(\mathbb{Q}_p) : E_0(\mathbb{Q}_p) ]$ is not divisible by $p$.
See \cite[Table 4.1, Page 365]{silverman2}
\end{rem}

Thus, when $E$ admits non-split multiplicative reduction at  $p \geq 3$, we have
$$\mathrm{exp}^{*}_{\omega_E} ( \mathrm{H}^1(\mathbb{Q}_p, T) ) = \dfrac{\widetilde{E}_{\mathrm{ns}}(\mathbb{F}_p)}{\#\mathrm{H}^0(\mathbb{Q}_p,E[p^\infty])} \cdot [ E(\mathbb{Q}_p) : E_0(\mathbb{Q}_p) ] \cdot \dfrac{1}{p} \cdot \mathbb{Z}_p = \dfrac{1}{p} \cdot \mathbb{Z}_p .$$
Now we suppose that $E$ admits split multiplicative reduction at a prime $p >2$.
Then the $p$-adic uniformization by the Tate curve yields the following exact sequence of representations of $\mathrm{Gal}(\overline{\mathbb{Q}}_p/\mathbb{Q}_p)$
\[
\xymatrix{
0 \ar[r] & T_1 \ar[r] & T \ar[r] & T_2 \ar[r] & 0
}
\]
where $T_1 \simeq \mathbb{Z}_p(1)$ and $T_2 \simeq \mathbb{Z}_p$. By tensoring $\mathbb{Q}_p$, we also have
\[
\xymatrix{
0 \ar[r] & V_1 \ar[r] & V \ar[r] & V_2 \ar[r] & 0 
}
\]
where $V_{\bullet} = T_{\bullet} \otimes \mathbb{Q}_p$ and $\bullet \in \lbrace \emptyset, 1, 2 \rbrace$.
Following \cite[$\S$4]{kobayashi-elementary}, we have the commutative diagram
\[
\xymatrix{
\mathrm{H}^1(\mathbb{Q}_p, V) \ar[d]_-{\pi} \ar[r]^-{\mathrm{exp}^{*}_{E}} & \mathbb{Q}_p \omega_E \ar[r]^-{\omega^{*}_E} & \mathbb{Q}_p \ar[d]^-{\simeq} \\
\mathrm{H}^1(\mathbb{Q}_p, V_2)  \ar[r]^-{\mathrm{exp}^{*}_{\mathbb{G}_m}} & \mathbb{Q}_p \omega_{\mathbb{G}_m} \ar[r]^-{\omega^{*}_{\mathbb{G}_m}} & \mathbb{Q}_p
}
\]
where $\pi$ is induced from the quotient map $V \to V_2$,
$\mathrm{exp}^{*}_{E}$ is the dual exponential map for $E$,
$\mathrm{exp}^{*}_{\mathbb{G}_m}$ is the dual exponential map for $\mathbb{G}_m$,
$\omega_{\mathbb{G}_m}$ is the invariant differential of $\mathbb{G}_m$, which is $\dfrac{dX}{1+X}$ on the formal multiplicative group $\widehat{\mathbb{G}}_m$,
and $\omega^{*}_{\mathbb{G}_m}$ is the dual basis for $\omega_{\mathbb{G}_m}$ with $\omega^{*}_{\mathbb{G}_m} ( \omega_{\mathbb{G}_m} ) = 1$.
Write
$\mathrm{exp}^{*}_{\omega_E} = \omega^*_E \circ \mathrm{exp}^{*}_{E}$
and
$\mathrm{exp}^{*}_{\omega_{\mathbb{G}_m}} = \omega^*_{\mathbb{G}_m} \circ \mathrm{exp}^{*}_{\mathbb{G}_m}$.
In order to compute the lattice $\mathrm{exp}^{*}_{\omega_E} ( \mathrm{H}^1(\mathbb{Q}_p, T) ) \subseteq \mathbb{Q}_p$,
it suffices to compute the lattice $\mathrm{exp}^{*}_{\omega_{\mathbb{G}_m}}  ( \mathrm{H}^1(\mathbb{Q}_p, \mathbb{Z}_p) ) \subseteq \mathbb{Q}_p$.
Considering the local Tate duality, we compute the lattice $\mathrm{log}  ( \mathbb{G}_m(\mathbb{Z}_p) ) \subseteq \mathbb{Q}_p$
where  $\mathrm{log} : \mathbb{G}_m(\mathbb{Z}_p) \otimes \mathbb{Q}_p \to \mathbb{Q}_p$ is the linear extension of the formal logarithm map on $\widehat{\mathbb{G}}_m(p\mathbb{Z}_p)$.
Note that $\mathbb{Z}_p(1)$ is the Tate module of the multiplicative group $\mathbb{G}_m$.
We have an exact sequence
\[
\xymatrix{
0
\ar[r] &
\widehat{\mathbb{G}}_m(p\mathbb{Z}_p)
\ar[r] &
\mathbb{G}_m(\mathbb{Z}_p)
\ar[r] &
\widetilde{\mathbb{G}}_m(\mathbb{F}_p)
\ar[r] & 
0 .
}
\]
By using the formal logarithm, $\widehat{\mathbb{G}}_m(p\mathbb{Z}_p)$ maps to $p\mathbb{Z}_p \subseteq \mathbb{Q}_p$.
Also, $\widetilde{\mathbb{G}}_m(\mathbb{F}_p) = \mathbb{F}^\times_p$ has size prime to $p$.
Thus, we also have
$$\mathrm{exp}^{*}_{\omega_E} ( \mathrm{H}^1(\mathbb{Q}_p, T) ) = \mathrm{exp}^{*}_{\omega_{\mathbb{G}_m}}  ( \mathrm{H}^1(\mathbb{Q}_p, \mathbb{Z}_p) ) = \dfrac{1}{p} \mathbb{Z}_p .$$
To sum up, we have the following statement.
\begin{lem} \label{lem:dual-exponential-multiplicative}
If $E$ has multiplicative reduction at $p \geq 3$, then
$$\mathrm{exp}^{*}_{\omega_E} ( \mathrm{H}^1(\mathbb{Q}_p, T) ) = \dfrac{1}{p} \mathbb{Z}_p .$$
\end{lem}
It is remarkable that Lemma \ref{lem:dual-exponential-multiplicative} does not involve $t$.
When $E$ has non-split multiplicative reduction at $p$, we have $t=0$.
When $E$ has split multiplicative reduction at $p$, $\#\mathrm{H}^0(\mathbb{Q}_p,E[p^\infty])$ and $[ E(\mathbb{Q}_p) : E_0(\mathbb{Q}_p) ]$ cancel each other. See \cite[Table 4.1, pp. 365]{silverman2}.

\begin{lem} \label{lem:dual-exponential-additive}
If $E$ has additive reduction at $p \geq 3$, then
$$\mathrm{exp}^*_{\omega_E} (  \mathrm{H}^1(\mathbb{Q}_p, T) )  = \dfrac{1}{p^t}\mathbb{Z}_p . $$
\end{lem}
\begin{proof}
It follows from Theorem \ref{thm:multiplicative-bad-good-index} and Remark \ref{rem:multiplicative-bad-good-index}.
\end{proof}

\subsection{The extension of the dual exponential map to torsion coefficients} \label{subsec:extension-dual-exp}
\subsubsection{}
Let $n \in \mathcal{N}_1$. From the exact sequence
\[
\xymatrix{
0 \ar[r] &
\dfrac{\mathrm{H}^1(\mathbb{Q}_p, T)}{I_n \mathrm{H}^1(\mathbb{Q}_p, T)} \ar[r] &
\mathrm{H}^1(\mathbb{Q}_p, T/I_nT) \ar[r] &
\mathrm{H}^2(\mathbb{Q}_p, T)[I_n] \ar[r] & 0 ,
}
\]
we obtain the following exact sequence
\begin{equation} \label{eqn:extending-dual-exp}
\xymatrix{
0 \ar[r] &
 \dfrac{\mathrm{H}^1(\mathbb{Q}_p, T)}{I_n \mathrm{H}^1(\mathbb{Q}_p, T)  + \mathrm{H}^1_f(\mathbb{Q}_p, T) } \ar[r]^-{\phi} &
\dfrac{\mathrm{H}^1(\mathbb{Q}_p, T/I_nT)}{\mathrm{H}^1_f(\mathbb{Q}_p, T/I_nT)}  \ar[r] &
\mathrm{H}^2(\mathbb{Q}_p, T)[I_n] \ar[r] & 0 
}
\end{equation}
where the injectivity of $\phi$ follows from $ \mathrm{H}^1_f(\mathbb{Q}_p, T/I_nT) = \dfrac{\mathrm{H}^1_f(\mathbb{Q}_p, T)}{ I_n\mathrm{H}^1_f(\mathbb{Q}_p, T)}$.
\begin{prop} \label{prop:extending-dual-exp-splits}
The sequence (\ref{eqn:extending-dual-exp}) splits as $\mathbb{Z}_p/I_n\mathbb{Z}_p$-modules.
\end{prop}
\begin{proof}
We first recall the isomorphism
$\mathrm{exp}^{*}_{\omega_E} : \dfrac{\mathrm{H}^1(\mathbb{Q}_p, T)}{\mathrm{H}^1_f(\mathbb{Q}_p, T) } \simeq \mathrm{exp}^{*}_{\omega_E} \left(\mathrm{H}^1(\mathbb{Q}_p, T) \right) (\simeq \mathbb{Z}_p) $.
Then its na\"{i}ve mod $I_n$ reduction also induces an isomorphism
$$\dfrac{\mathrm{H}^1(\mathbb{Q}_p, T)}{I_n \mathrm{H}^1(\mathbb{Q}_p, T)  + \mathrm{H}^1_f(\mathbb{Q}_p, T) } \simeq  \dfrac{ \mathrm{exp}^{*}_{\omega_E} \left(\mathrm{H}^1(\mathbb{Q}_p, T) \right) }{ I_n \mathrm{exp}^{*}_{\omega_E} \left(\mathrm{H}^1(\mathbb{Q}_p, T) \right)  }  (\simeq \mathbb{Z}_p/I_n \mathbb{Z}_p ) .$$
In particular, $\dfrac{\mathrm{H}^1(\mathbb{Q}_p, T)}{I_n \mathrm{H}^1(\mathbb{Q}_p, T)  + \mathrm{H}^1_f(\mathbb{Q}_p, T) }$ is free of rank one over $\mathbb{Z}_p/I_n \mathbb{Z}_p$.
Thus, (\ref{eqn:extending-dual-exp}) splits.
\end{proof}
We define the dual exponential map on $\mathrm{H}^1(\mathbb{Q}_p, T/I_nT)$ by the composition
$$\mathrm{exp}^{*}_{\omega_E} : \mathrm{H}^1(\mathbb{Q}_p, T/I_nT) \to \dfrac{\mathrm{H}^1(\mathbb{Q}_p, T/I_nT)}{\mathrm{H}^1_f(\mathbb{Q}_p, T/I_nT)}
\to
\dfrac{\mathrm{H}^1(\mathbb{Q}_p, T)}{I_n \mathrm{H}^1(\mathbb{Q}_p, T)  + \mathrm{H}^1_f(\mathbb{Q}_p, T) } 
\to
\dfrac{ \mathrm{exp}^{*}_{\omega_E}\left(\mathrm{H}^1(\mathbb{Q}_p, T) \right) }{ I_n \mathrm{exp}^{*}_{\omega_E}\left(\mathrm{H}^1(\mathbb{Q}_p, T) \right)  }$$
where the first map is the natural quotient map, the second map comes from the splitting of (\ref{eqn:extending-dual-exp}) (as proved in Proposition \ref{prop:extending-dual-exp-splits}), and the third map is the na\"{i}ve mod $I_n$ reduction of the integral dual exponential map again.

\subsection{``Kolyvagin Derivatives" of Kato's explicit reciprocity law}

\subsubsection{}
Let $n \in \mathcal{N}_1$ and fix an isomorphism
$\xi : \dfrac{\#\widetilde{E}_{\mathrm{ns}}(\mathbb{F}_p)}{p^{1+t}} \mathbb{Z}_p / \dfrac{\#\widetilde{E}_{\mathrm{ns}}(\mathbb{F}_p)}{p^{1+t}} I_n \mathbb{Z}_p  \simeq \mathbb{Z}_p /  I_n \mathbb{Z}_p $.
Then we have the map
\begin{equation} \label{eqn:exp-loc}
\begin{split}
\xymatrix{
\mathrm{Sel}_{\mathrm{rel}, n}(\mathbb{Q}, T/I_nT) \ar[d]^-{\mathrm{loc}^s_p} \ar[r]^-{ \xi \circ \mathrm{exp}^*_{\omega_E} \circ \mathrm{loc}^s_p } &  \mathbb{Z}_p /  I_n \mathbb{Z}_p \\
\dfrac{ \mathrm{H}^1(\mathbb{Q}_p, T/I_nT) }{ E(\mathbb{Q}_p) \otimes \mathbb{Z}_p/I_n \mathbb{Z}_p } \ar[r]^-{ \mathrm{exp}^*_{\omega_E} }_-{\simeq} & \dfrac{\#\widetilde{E}_{\mathrm{ns}}(\mathbb{F}_p)}{p^{1+t}} \mathbb{Z}_p / \dfrac{\#\widetilde{E}_{\mathrm{ns}}(\mathbb{F}_p)}{p^{1+t}} I_n \mathbb{Z}_p  \ar[u]^-{\simeq}_-{\xi}
}
\end{split}
\end{equation}
thanks to Lemma \ref{lem:dual-exponential}, Lemma  \ref{lem:dual-exponential-multiplicative},  Lemma \ref{lem:dual-exponential-additive}, and the extension of the dual exponential map to torsion coefficients. 
Note that $\#\widetilde{E}_{\mathrm{ns}}(\mathbb{F}_p)$ is prime to $p$ if $E$ has multiplicative reduction and $\#\widetilde{E}_{\mathrm{ns}}(\mathbb{F}_p) = p$ if $E$ has additive reduction.

\begin{thm} \label{thm:from-ks-to-kn}
Let $E$ be an elliptic curve over $\mathbb{Q}$ and $p\geq 5$ is a prime such that
 $\overline{\rho}$ is surjective and the Manin constant is prime to $p$.
Then we have formula
$$\xi \circ \mathrm{exp}^*_{\omega_E} \circ \mathrm{loc}^s_p (\kappa^{\mathrm{Kato}}_n) = u \cdot p^t \cdot \widetilde{\delta}_n \in \mathbb{Z}_p/I_n\mathbb{Z}_p $$
where $u \in (\mathbb{Z}_p /  I_n \mathbb{Z}_p)^\times$.
\end{thm}
\begin{proof}
The following computation is essentially done in \cite[$\S$6 and $\S$7]{kks}
\begin{align*}
\mathrm{exp}^*_{\omega_E} \circ \mathrm{loc}^s_p ( D_{\mathbb{Q}(\zeta_n)} z^{\mathrm{Kato}}_{\mathbb{Q}(\zeta_n)} ) & = E_p(\sigma_p)  \cdot D_{\mathbb{Q}(\zeta_n)} \left( \sum_{a \in (\mathbb{Z}/n\mathbb{Z})^\times } \zeta^a_n \cdot \left[ \dfrac{a}{n} \right]^+ \right)  \in \mathbb{Q}_p \otimes \mathbb{Q}(\zeta_n)
\end{align*}
where 
$E_p(\sigma_p) = \left\lbrace \begin{array}{ll}
1-a_p(E) \cdot p^{-1}\cdot \sigma_p + p^{-1} \cdot \sigma^{2}_p & \textrm{if } p \nmid N \\
1-a_p(E)  \cdot p^{-1} \cdot \sigma_p  &\textrm{if } p \Vert N \\
1 & \textrm{if } p^2 \vert N
\end{array} \right.$
and
$\sigma_p \in \mathrm{Gal}(\mathbb{Q}(\zeta_n)/\mathbb{Q})$ is the arithmetic Frobenius at $p$.
We also have
$$D_{\mathbb{Q}(\zeta_n)} \left( \sum_{a \in (\mathbb{Z}/n\mathbb{Z})^\times } \zeta^a_n \cdot \left[ \dfrac{a}{n} \right]^+ \right)  \equiv \widedelta_n \in \mathbb{Z}_p/I_n\mathbb{Z}_p $$
as in \cite[Thm. 7.5]{kks} and \cite[pp. 190]{kurihara-munster}.

By the extension of the dual exponential map to torsion coefficients discussed in $\S$\ref{subsec:extension-dual-exp}, we have
\begin{equation} \label{eqn:dual-exponential-mod-I_n}
\mathrm{exp}^*_{\omega_E} : \mathrm{H}^1_{/f}(\mathbb{Q}_p, T/I_nT ) \simeq \dfrac{\#\widetilde{E}_{\mathrm{ns}}(\mathbb{F}_p)}{p^{1+t}}\mathbb{Z}_p/ \dfrac{\#\widetilde{E}_{\mathrm{ns}}(\mathbb{F}_p)}{p^{1+t}} I_n\mathbb{Z}_p .
\end{equation}
Because $\kappa^{\mathrm{Kato}}_n \in \mathrm{Sel}_{\mathrm{rel},n}(\mathbb{Q}, T/I_nT) \subseteq \mathrm{H}^1(\mathbb{Q}, T/I_nT)  $
comes from the mod $I_n$ reduction of $D_{\mathbb{Q}(\zeta_n)} z^{\mathrm{Kato}}_{\mathbb{Q}(\zeta_n)}$, 
 we regard
$E_p(\sigma_p)  \cdot D_{\mathbb{Q}(\zeta_n)} \left( \sum_{a \in (\mathbb{Z}/n\mathbb{Z})^\times } \zeta^a_n \cdot \left[ \dfrac{a}{n} \right]^+ \right)$ as an element in $\dfrac{\#\widetilde{E}_{\mathrm{ns}}(\mathbb{F}_p)}{p^{1+t}} \mathbb{Z}_p \otimes \mathbb{Z}[\zeta_n] \subseteq \mathbb{Q}_p \otimes \mathbb{Q}(\zeta_n)$
via (\ref{eqn:dual-exponential-mod-I_n})
 in order to compute $\mathrm{exp}^*_{\omega_E} \circ \mathrm{loc}^s_p ( \kappa^{\mathrm{Kato}}_n )$.

We record all the integral lattices in each step of the following computation to avoid confusion.
\begin{align*}
& \mathrm{exp}^*_{\omega_E} \circ \mathrm{loc}^s_p ( \kappa^{\mathrm{Kato}}_n ) \in \dfrac{\#\widetilde{E}_{\mathrm{ns}}(\mathbb{F}_p)}{p^{1+t}}\mathbb{Z}_p/ \dfrac{\#\widetilde{E}_{\mathrm{ns}}(\mathbb{F}_p)}{p^{1+t}} I_n\mathbb{Z}_p  \\
 & =  \mathrm{exp}^*_{\omega_E} \circ \mathrm{loc}^s_p ( D_{\mathbb{Q}(\zeta_n)} z^{\mathrm{Kato}}_{\mathbb{Q}(\zeta_n)} \pmod{I_n} )  \in \dfrac{\#\widetilde{E}_{\mathrm{ns}}(\mathbb{F}_p)}{p^{1+t}}\mathbb{Z}_p/ \dfrac{\#\widetilde{E}_{\mathrm{ns}}(\mathbb{F}_p)}{p^{1+t}} I_n\mathbb{Z}_p \\
& = \mathrm{exp}^*_{\omega_E} \circ \mathrm{loc}^s_p ( D_{\mathbb{Q}(\zeta_n)} z^{\mathrm{Kato}}_{\mathbb{Q}(\zeta_n)} ) \pmod{I_n}  \in \dfrac{\#\widetilde{E}_{\mathrm{ns}}(\mathbb{F}_p)}{p^{1+t}} \mathbb{Z}_p \otimes \mathbb{Z}[\zeta_n] / \dfrac{\#\widetilde{E}_{\mathrm{ns}}(\mathbb{F}_p)}{p^{1+t}} I_n \mathbb{Z}_p \otimes \mathbb{Z}[\zeta_n] \\
& = E_p(\sigma_p)  \cdot D_{\mathbb{Q}(\zeta_n)} \left( \sum_{a \in (\mathbb{Z}/n\mathbb{Z})^\times } \zeta^a_n \cdot \left[ \dfrac{a}{n} \right]^+ \right) \in \dfrac{\#\widetilde{E}_{\mathrm{ns}}(\mathbb{F}_p)}{p^{1+t}} \mathbb{Z}_p \otimes \mathbb{Z}[\zeta_n] / \dfrac{\#\widetilde{E}_{\mathrm{ns}}(\mathbb{F}_p)}{p^{1+t}} I_n \mathbb{Z}_p \otimes \mathbb{Z}[\zeta_n] \\
& = E_p(\sigma_p)  \cdot \widedelta_n \in \dfrac{\#\widetilde{E}_{\mathrm{ns}}(\mathbb{F}_p)}{p^{1+t}} \mathbb{Z}_p  / \dfrac{\#\widetilde{E}_{\mathrm{ns}}(\mathbb{F}_p)}{p^{1+t}} I_n \mathbb{Z}_p  \\
& = E_p  \cdot \widedelta_n \in \dfrac{\#\widetilde{E}_{\mathrm{ns}}(\mathbb{F}_p)}{p^{1+t}} \mathbb{Z}_p  / \dfrac{\#\widetilde{E}_{\mathrm{ns}}(\mathbb{F}_p)}{p^{1+t}} I_n \mathbb{Z}_p 
\end{align*}
where 
$E_p = \left\lbrace \begin{array}{ll}
1 - a_p(E) p^{-1} + p^{-1} & \textrm{if } p \nmid N \\
1 - a_p(E) p^{-1} &\textrm{if } p \Vert N \\
1 & \textrm{if } p^2 \vert N .
\end{array} \right. $
Due to the difference between the torsion dual exponential map and the na\"{i}ve mod $I_n$ reduction of the dual exponential map, the image of $D_{\mathbb{Q}(\zeta_n)} z^{\mathrm{Kato}}_{\mathbb{Q}(\zeta_n)}$ under $\mathrm{exp}^*_{\omega_E} \circ \mathrm{loc}^s_p$ actually lies in $\dfrac{\#\widetilde{E}_{\mathrm{ns}}(\mathbb{F}_p)}{p} \mathbb{Z}_p \otimes \mathbb{Z}[\zeta_n]$. By using the fixed isomorphism $\xi$, the conclusion follows.
\end{proof}
Theorem  \ref{thm:from-ks-to-kn} can be understood as Kolyvagin derivatives of (an equivariant refinement of) Kato's explicit reciprocity law \cite[Thm. 12.5]{kato-euler-systems}, \cite[Thm. 6.1]{kataoka-thesis}.

\subsubsection{}
\begin{lem} \label{lem:ks-kn-equi-non-vanishing}
Let $E$ be an elliptic curve over $\mathbb{Q}$ and $p\geq 5$ is a prime such that
 $\overline{\rho}$ is surjective and the Manin constant is prime to $p$.
Then the non-vanishing of $p^t \cdot \kn$ and the non-vanishing of $\kn$ are equivalent.
In particular, $\mathrm{ord}(p^t \cdot \kn) = \mathrm{ord}(\kn)$.
\end{lem}
\begin{proof}
This can be checked directly.
\end{proof}
\begin{prop} \label{prop:ks-kn-equi-non-vanishing}
Let $E$ be an elliptic curve over $\mathbb{Q}$ and $p\geq 5$ is a prime such that
 $\overline{\rho}$ is surjective and the Manin constant is prime to $p$.
Then the non-triviality of $\ks^{\mathrm{Kato}}$ and the non-vanishing of $\kn$ are equivalent.
\end{prop}
\begin{proof}
The $p^t \cdot \kn \neq 0 \Rightarrow \ks^{\mathrm{Kato}} \neq 0$ direction follows from Theorem \ref{thm:from-ks-to-kn}.
The opposite direction follows from Theorem \ref{thm:core-rank-elliptic-curves} and Lemma \ref{lem:ks-kn-equi-non-vanishing}.
\end{proof}
Proposition \ref{prop:ks-kn-equi-non-vanishing} does \emph{not} mean that $\kappa^{\mathrm{Kato}}_n \neq 0$ if and only if $\widedelta_n \neq 0$ for each $n \in \mathcal{N}_1$.

\subsection{Functional equations and vanishing of Kurihara numbers}
Following \cite[Lem. 4 (Page 347)]{kurihara-iwasawa-2012} and \cite[Page 220]{kurihara-munster}, we have
\begin{equation} \label{eqn:functional-equation-delta_n}
w(E) \cdot (-1)^{\nu(n)} \cdot \widetilde{\delta}_n = \widetilde{\delta}_n \in \mathbb{Z}_p/I_n \mathbb{Z}_p
\end{equation}
where $w(E)$ is the root number of $E$. Comparing with Proposition \ref{prop:non-vanishing-kappa-n}, the following statement shows the fundamental difference between $\ks^{\mathrm{Kato}}$ and $\kn$.
\begin{prop} \label{prop:vanishing-delta-n}
If $(-1)^{\nu(n)} \neq w(E)$, then $\widetilde{\delta}_n  = 0$.
In particular, if $\widedelta_n \neq 0$, then $\widedelta_{n\ell} = 0 \in \mathbb{Z}_p/I_{n\ell}\mathbb{Z}_p$ for all $\ell \in \mathcal{N}_1$ with $(n, \ell) =1$.
\end{prop}

\section{{$\Lambda$}-adic Kolyvagin systems for elliptic curves: Proof of Theorem \ref{thm:main-technical}} \label{sec:Lambda-adic-KS}
Throughout this section, we assume that $\overline{\rho}$ is surjective.
\subsection{The Iwasawa-theoretic set up}
\subsubsection{}
Recall the notation in $\S$\ref{subsubsec:cyclotomic-setup}.
Let $\mathbb{Q}_\infty$ be the cyclotomic $\mathbb{Z}_p$-extension of $\mathbb{Q}$ and $\mathbb{Q}_m \subseteq \mathbb{Q}_\infty$ the cyclic subextension of $\mathbb{Q}$ of degree $p^m$ in $\mathbb{Q}_\infty$.
Let $$\Lambda = \mathbb{Z}_p\llbracket  \mathrm{Gal}(\mathbb{Q}_\infty/\mathbb{Q}) \rrbracket = \varprojlim_m \mathbb{Z}_p[  \mathrm{Gal}(\mathbb{Q}_m/\mathbb{Q}) ]$$
be the Iwasawa algebra.
\subsubsection{}
Fix a finite set  $\Sigma$  of the places of $\mathbb{Q}$ containing $p$, $\infty$, and the primes where $T$ is ramified.
Denote by $\mathbb{Q}_{\Sigma}$ the maximal extension of $\mathbb{Q}$ unramified outside $\Sigma$.
\begin{lem} \label{lem:iwasawa-cohomology}
\begin{enumerate}
\item $\mathrm{H}^1(\mathbb{Q}_{\Sigma}/\mathbb{Q}, T \otimes \Lambda) \simeq \varprojlim_m \mathrm{H}^1(\mathbb{Q}_{\Sigma}/\mathbb{Q}_m, T) $.
\item If $\ell \neq p$, then $\mathrm{H}^1(\mathbb{Q}_\ell, T \otimes \Lambda) = \mathrm{H}^1_{\mathrm{ur}}(\mathbb{Q}_\ell, T \otimes \Lambda)$.
\item $\mathrm{H}^1(\mathbb{Q}_{\Sigma}/\mathbb{Q}, T \otimes \Lambda)$ is independent of the choice of $\Sigma$.
\end{enumerate}
\end{lem}
\begin{proof}
See \cite[Lem. 5.3.1]{mazur-rubin-book}.
\end{proof}
The Selmer structure $\mathcal{F}_{\Lambda}$ on $T \otimes \Lambda$ is defined by
$$\mathrm{H}^1_{\mathcal{F}_{\Lambda}}(\mathbb{Q}_v, T \otimes \Lambda) = \mathrm{H}^1(\mathbb{Q}_v, T \otimes \Lambda)$$
for all $v \in \Sigma$.
Also, by Lemma \ref{lem:iwasawa-cohomology}.(2), we have
$\mathrm{H}^1_{f}(\mathbb{Q}_v, T \otimes \Lambda) = \mathrm{H}^1(\mathbb{Q}_v, T \otimes \Lambda)$
for all $v \not\in \Sigma$.
Thus, this Selmer structure is independent of the choice of $\Sigma$, and
$$\mathrm{Sel}_{\mathcal{F}_{\Lambda}}(\mathbb{Q}, T\otimes \Lambda) = \mathrm{H}^1(\mathbb{Q}, T \otimes \Lambda).$$
\begin{thm}[Mazur--Rubin] \label{thm:euler-to-kolyvagin-Lambda-adic}
Suppose that $\mathcal{K}$ contains the maximal abelian $p$-extension of $\mathbb{Q}$ which is unramified outside $p$ and $\mathcal{P}$, and
\begin{enumerate}
\item $T/(\mathrm{Fr}_\ell - 1)T$ is a cyclic $\mathbb{Z}_p$-module for every $\ell \in \mathcal{P}$,
\item $\mathrm{Fr}^{p^k}_\ell - 1$ is injective on $T$ for every $\ell \in \mathcal{P}$ and every $k \geq 0$. 
\end{enumerate}
Then there exists a canonical homomorphism 
\[
\xymatrix@R=0em{
\ES(T, \mathcal{K}, \mathcal{P}) \ar[r] & \KSbar(T \otimes \Lambda, \mathcal{F}_\Lambda, \mathcal{P}) \\
\mathbf{z}  \ar@{|->}[r] & \ks^{\infty}
}
\]
such that
$\kappa^\infty_1 = z_{\mathbb{Q}_\infty} = \varprojlim_m z_{\mathbb{Q}_m} \in \mathrm{H}^1(\mathbb{Q}, T\otimes \Lambda)$.
\end{thm}
\begin{proof}
See \cite[Thm. 5.3.3]{mazur-rubin-book}.
Under our working hypotheses, all the assumptions are satisfied.
\end{proof}
\subsubsection{}
\begin{lem}  \label{lem:iwasawa-cohomology-finitely-generated}
For every $i \geq 0$, $\mathrm{H}^1(\mathbb{Q}_{\Sigma}/\mathbb{Q}, T \otimes \Lambda)$ and 
$\mathrm{H}^i(\mathbb{Q}_p, T \otimes \Lambda)$
are finitely generated, and
$\mathrm{H}^1(\mathbb{Q}_{\Sigma}/\mathbb{Q}, (T \otimes \Lambda)^*)$ is co-finitely generated.
Furthermore, $\mathrm{H}^2(\mathbb{Q}_p, T \otimes \Lambda)$ is a torsion $\Lambda$-module.
\end{lem}
\begin{proof}
See \cite[Lem. 5.3.4]{mazur-rubin-book}.
\end{proof}
\begin{thm}[Kato--Rohrlich]
Kato's zeta element $\kappa^{\mathrm{Kato},\infty}_{1} = z^{\mathrm{Kato}}_{\mathbb{Q}_\infty}$ over $\mathbb{Q}_\infty$ is non-trivial.
\end{thm}
\begin{proof}
By using Kato's explicit reciprocity law \cite[Thm. 12.5]{kato-euler-systems}, 
the conclusion follows from the generic non-vanishing of twisted $L$-values \cite{rohrlich-nonvanishing, rohrlich-nonvanishing-2}.
\end{proof}
\begin{thm}[Kato] \label{thm:kato-12-4}
\begin{enumerate}
\item 
If $\overline{\rho}$ is irreducible, then 
 $\mathrm{Sel}_{\mathcal{F}_\Lambda} (\mathbb{Q}, T \otimes \Lambda) = \mathrm{H}^1 (\mathbb{Q}, T \otimes \Lambda)$ is free of rank one over $\Lambda$.
\item 
$\mathrm{Sel}_{\mathcal{F}^*_\Lambda} (\mathbb{Q}, (T \otimes \Lambda)^*) = \mathrm{Sel}_0(\mathbb{Q}_\infty, E[p^\infty])$ is a co-finitely generated co-torsion $\Lambda$-module.
\end{enumerate}
\end{thm}
\begin{proof}
See \cite[Thm. 12.4]{kato-euler-systems}. See also \cite[Lem. 5.3.5 and Thm. 5.3.6]{mazur-rubin-book}.
\end{proof}
We recall the notion in $\S$\ref{subsubsec:blind-spot-Lambda-primitivity}.
We say that \textbf{$\ks^{\infty}$ is $\Lambda$-primitive} if $\ks^{\infty} \pmod{\mathfrak{P}}$ does not vanish for every height one prime $\mathfrak{P}$ of $\Lambda$ \cite[Def. 5.3.9]{mazur-rubin-book}.
\begin{thm}[Mazur--Rubin] \label{thm:mazur-rubin-main-conjecture}
Let $\ks^{\infty} \in \KS(T \otimes \Lambda)$ be a $\Lambda$-adic Kolyvagin system.
\begin{enumerate}
\item $ \mathrm{char}_{\Lambda} \left( \dfrac{\mathrm{H}^1_{\mathrm{Iw}}(\mathbb{Q}, T)}{\Lambda \kappa^{\infty}_1 }  \right) \subseteq 
\mathrm{char}_{\Lambda} \left( \mathrm{Sel}_0(\mathbb{Q}_{\infty}, E[p^\infty] )^\vee  \right) .$
\item Suppose that $\kappa^\infty_1 \neq 0$.
If $\mathfrak{P}$ is not a blind spot of $\ks^{\infty}$, then 
$$\mathrm{ord}_{\mathfrak{P}} \left( \mathrm{char}_{\Lambda} \left( \dfrac{\mathrm{H}^1_{\mathrm{Iw}}(\mathbb{Q}, T)}{\Lambda \kappa^{\infty}_1 }  \right) \right)
=
\mathrm{ord}_{\mathfrak{P}} \left( \mathrm{char}_{\Lambda} \left( \mathrm{Sel}_0(\mathbb{Q}_{\infty}, E[p^\infty] )^\vee  \right) \right) .$$
\item Suppose that $\kappa^\infty_1 \neq 0$.
If $\ks$  is $\Lambda$-primitive, then
$$ \mathrm{char}_{\Lambda} \left( \dfrac{\mathrm{H}^1_{\mathrm{Iw}}(\mathbb{Q}, T)}{\Lambda \kappa^{\infty}_1 }  \right) = 
\mathrm{char}_{\Lambda} \left( \mathrm{Sel}_0(\mathbb{Q}_{\infty}, E[p^\infty] )^\vee  \right) .$$
\end{enumerate}
\end{thm}
\begin{proof}
See \cite[Thm. 5.3.10]{mazur-rubin-book}.
\end{proof}
Theorem \ref{thm:main-technical} proves the converse to the second and the third statements of Theorem \ref{thm:mazur-rubin-main-conjecture}.
Thanks to \cite[Rem. 5.3.11]{mazur-rubin-book}, both $\KS(T \otimes \Lambda)$ and $\KSbar(T \otimes \Lambda)$ work for all the argument below.

\subsection{Specializations at height one primes}
Let $\mathfrak{P}$ be a height one prime ideal of $\Lambda$.
Denote by $S_{\mathfrak{P}}$ the integral closure of $\Lambda/ \mathfrak{P}$, which is a discrete valuation ring.
Then $[S_{\mathfrak{P}}:\mathbb{Z}_p]$ is finite and $T \otimes \Lambda \otimes_{\Lambda} S_{\mathfrak{P}} = T  \otimes_{\mathbb{Z}_p} S_{\mathfrak{P}}$.
The canonical Selmer structure on $T  \otimes_{\mathbb{Z}_p} S_{\mathfrak{P}}$
is induced from the $\Lambda$-adic Selmer structure on $T  \otimes \Lambda$.
Let
$$\Sigma_{\Lambda} = \left\lbrace \mathfrak{P} : \mathrm{H}^2(\mathbb{Q}_\Sigma/\mathbb{Q}, T \otimes \Lambda)[\mathfrak{P}] \textrm{ is infinite} \right\rbrace \cup
\left\lbrace \mathfrak{P} : \mathrm{H}^2(\mathbb{Q}_p, T \otimes \Lambda)[\mathfrak{P}] \textrm{ is infinite} \right\rbrace \cup \left\lbrace p\Lambda \right\rbrace$$
be the exceptional set of height one primes of $\Sigma_{\Lambda}$ and it is finite by Lemma \ref{lem:iwasawa-cohomology-finitely-generated}.

\begin{prop}
For every height one prime $\mathfrak{P}$ of $\Lambda$,
the composition map $T \otimes \Lambda \twoheadrightarrow T \otimes \Lambda / \mathfrak{P} \hookrightarrow T \otimes S_{\mathfrak{P}}$ induces maps
\[
\xymatrix@R=0em{
\dfrac{\mathrm{H}^1(\mathbb{Q}, T \otimes \Lambda)}{\mathfrak{P}\mathrm{H}^1(\mathbb{Q}, T \otimes \Lambda)}
\ar@{^{(}->}[r]^-{\pi_{\mathfrak{P}}} & \mathrm{Sel}_{\mathrm{rel}}(\mathbb{Q},T \otimes S_{\mathfrak{P}}) , \\
\mathrm{Sel}_{0}(\mathbb{Q},(T \otimes S_{\mathfrak{P}})^*) \ar[r]^-{\pi^*_{\mathfrak{P}}} &
\mathrm{Sel}_{0}(\mathbb{Q}_\infty,E[p^\infty])[\mathfrak{P}] .
}
\]
For every $\mathfrak{P}$, the map $\pi_{\mathfrak{P}}$ is injective.
If $\mathfrak{P} \not\in \Sigma_{\Lambda}$, then
$\mathrm{coker}(\pi_{\mathfrak{P}})$,
$\mathrm{ker}(\pi^*_{\mathfrak{P}})$, and
$\mathrm{coker}(\pi^*_{\mathfrak{P}})$
are all finite with order bounded by a constant depending only on $T$ and $[S_{\mathfrak{P}}:\Lambda/\mathfrak{P}]$.
\end{prop}
\begin{proof}
See \cite[Prop. 5.3.14]{mazur-rubin-book}.
\end{proof}
\begin{cor} \label{cor:specialization}
For every height one prime $\mathfrak{P}$ of $\Lambda$, there is a natural map
$$\KS(T \otimes \Lambda, \mathcal{F}_{\Lambda}) \to
\KS(T \otimes S_{\mathfrak{P}}, \mathcal{F}_{\mathrm{can}})$$
\end{cor}
\begin{proof}
See \cite[Cor. 5.3.15]{mazur-rubin-book}.
\end{proof}
The core rank of the ($\Lambda$-adic) Kolyvagin system for $T\otimes\Lambda$ is defined by the common value of the core rank of the Kolyvagin system for $T \otimes S_{\mathfrak{P}}$ for every $\mathfrak{P} \not\in \Sigma_{\Lambda}$.
It is well-defined and it is one in our case.

Let $\ks^{\infty} \in \KS(T \otimes \Lambda, \mathcal{F}_{\Lambda})$ and $\mathfrak{P}$ a height one prime ideal of $\Lambda$.
Denote by $\ks^{(\mathfrak{P})}$ the image of $\ks^{\infty}$ in $\KS(T \otimes S_{\mathfrak{P}}, \mathcal{F}_{\mathrm{can}})$ under the natural map in Corollary \ref{cor:specialization}.
\begin{cor}
Let $\ks^{\infty} \in \KS(T \otimes \Lambda)$ with $\kappa^\infty_1 \neq 0$.
Then for all but finitely many height one primes $\mathfrak{P}$ of $\Lambda$, 
the class $\kappa^{(\mathfrak{P})}_1 \in \mathrm{H}^1(\mathbb{Q}, T \otimes S_{\mathfrak{P}})$ is non-trivial.
\end{cor}
\begin{proof}
See \cite[Cor. 5.3.19]{mazur-rubin-book}.
\end{proof}

\begin{prop} \label{prop:non-triviality-equivalence}
Let $\ks^{\infty} \in \KS(T \otimes \Lambda, \mathcal{F}_{\Lambda})$ and $\mathfrak{P}$ a height one prime ideal of $\Lambda$.
Then the following statements are equivalent.
\begin{enumerate}
\item $\ks^{\infty} \pmod{\mathfrak{P}} \in \KS(T \otimes \Lambda / \mathfrak{P}, \mathcal{F}_{\mathrm{can}})$ is non-trivial, i.e. $\mathfrak{P}$ is not a blind spot of $\ks^{\infty}$.
\item $\ks^{(\mathfrak{P})} \in \KS(T \otimes S_{\mathfrak{P}}, \mathcal{F}_{\mathrm{can}})$ is non-trivial.
\end{enumerate}
\end{prop}
\begin{proof}
See \cite[Lem. 5.3.20]{mazur-rubin-book}.
\end{proof}

\subsection{Proof of Theorem \ref{thm:main-technical}} \label{subsec:proof-of-main-technical}
Recall that $(-)^* = \mathrm{Hom}(-, \mu_{p^\infty})$ means the Cartier dual.

Let $\Lambda$ be the Iwasawa algebra and identify it with $\mathbb{Z}_p\llbracket X \rrbracket$.
Let $\mathfrak{P}$ be a height one prime ideal of $\Lambda$ with $\mathfrak{P} \neq p\Lambda$.
Write $f_{\mathfrak{P}}(X)$ to be a distinguished polynomial such that $\mathfrak{P} = \left( f_{\mathfrak{P}}(X) \right)$ in $\Lambda$.
For an integer $M \geq 1$, let
$$\mathfrak{P}_M = (f_{\mathfrak{P}}(X) + p^M) \Lambda .$$
Fix a pseudo-isomorphism
\begin{equation} \label{eqn:pseudo-isom}
\mathrm{Sel}_0(\mathbb{Q}_{\infty}, E[p^\infty] )^\vee \to \bigoplus_i \Lambda / \mathfrak{P}^{m_i} \oplus \bigoplus_j \Lambda / f_j \Lambda
\end{equation}
where each $f_j$ is prime to $\mathfrak{P}$. We write
$$\mathrm{ord}_{\mathfrak{P}} ( \mathrm{char}_{\Lambda} (\mathrm{Sel}_0(\mathbb{Q}_{\infty}, E[p^\infty] )^\vee) ) = \sum_i m_i.$$
For convenience, we also write
$T \otimes S_{\mathfrak{P}_M} = T \otimes_{\mathbb{Z}_p} \Lambda \otimes_{\Lambda} S_{\mathfrak{P}_M}$.
\begin{lem} \label{lem:higher-power}
If $M$ is sufficiently large, $\mathfrak{P}_M$ satisfies the following properties:
\begin{enumerate}
\item $\mathfrak{P}_M$ is a height one prime ideal of $\Lambda$ and $\Lambda / \mathfrak{P} \simeq \Lambda / \mathfrak{P}_M$,
\item the image of $\kappa^{\mathrm{Kato},\infty}_1$ in $\mathrm{H}^1(\mathbb{Q}, T  \otimes S_{\mathfrak{P}_M})$ is non-zero, i.e. $\kappa^{\mathrm{Kato},(\mathfrak{P}_M)}_1 \neq 0$,
\item the cokernel of the injection
$$\mathrm{H}^1(\mathbb{Q}, T \otimes \Lambda) / \mathfrak{P}_M \mathrm{H}^1(\mathbb{Q}, T \otimes \Lambda) \hookrightarrow \mathrm{Sel}_{\mathrm{rel}} (\mathbb{Q}, T \otimes S_{\mathfrak{P}_M})$$
 is finite with bounded order by a constant independent of $M$, and
\item $\mathfrak{P}_M$ is prime to each $f_j$ in (\ref{eqn:pseudo-isom}), and $\mathfrak{P}_M \not\in \Sigma_{\Lambda}$. 
\end{enumerate}
\end{lem}
\begin{proof}
See \cite[Proof of Thm. 5.3.10, pp. 66]{mazur-rubin-book}.
\end{proof}
Let $f_\kappa(X)$ be a generator of $\mathrm{char}_{\Lambda} \left( \dfrac{\mathrm{H}^1_{\mathrm{Iw}}(\mathbb{Q}, T)}{\Lambda \kappa^{\mathrm{Kato}, \infty}_1 } \right)$.
If a height one prime $\mathfrak{P} = ( f_{\mathfrak{P}}(X) )$ has no common zero with $f_\kappa(X)$, then
$\kappa^{\mathrm{Kato}, (\mathfrak{P})}_1 = \kappa^{\mathrm{Kato},\infty}_1 \pmod{\mathfrak{P}}$ is non-zero, so there is nothing to prove.

Now we assume that $f_\kappa(X)$ and  $f_{\mathfrak{P}}(X)$  have a common zero.

By taking large $M \gg 0$, we have $\kappa^{\mathrm{Kato},(\mathfrak{P}_M)}_1 \neq 0$.
In other words, $f_\kappa(X)$ and $f_{\mathfrak{P}_M}(X)$ have no common zero
where $f_{\mathfrak{P}_M}(X) = f_{\mathfrak{P}}(X) +p^M$.
Then we have
$$\langle \kappa^{\mathrm{Kato},(\mathfrak{P}_M)}_1 \rangle = \mathfrak{m}^{j_{\mathfrak{P}_M}+ \lambda(1, S_{\mathfrak{P}_M}) }_{\mathfrak{P}_M} \mathrm{Sel}_{\mathrm{rel}}(\mathbb{Q}, T \otimes S_{\mathfrak{P}_M})$$
for some constant $j_{\mathfrak{P}_M}$
where $\lambda(1, S_{\mathfrak{P}_M}) = \mathrm{length}_{S_{\mathfrak{P}_M}} \mathrm{Sel}_0(\mathbb{Q}, (T \otimes S_{\mathfrak{P}_M})^* )$ as in Proposition \ref{prop:kolyvagin-system-location}.
In particular, we have
\begin{equation} \label{eqn:j-lambda-partial-zero}
\xymatrix{
\partial^{(0)} (\ks^{\mathrm{Kato},(\mathfrak{P}_M)})  = j_{\mathfrak{P}_M}+ \lambda(1, S_{\mathfrak{P}_M}) , & \partial^{(\infty)} (\ks^{\mathrm{Kato},(\mathfrak{P}_M)}) = j_{\mathfrak{P}_M}
}
\end{equation}
by \cite[Thm 5.2.12]{mazur-rubin-book}, which is a slightly more general version of Theorem \ref{thm:kato-kolyvagin-main}.

Let $e(S_{\mathfrak{P}_M}/\mathbb{Z}_p)$ be the ramification degree of $S_{\mathfrak{P}_M}/\mathbb{Z}_p$.
Since $$\Lambda/ (\mathfrak{P}^{\mathrm{ord}_{\mathfrak{P}} ( f_{\kappa}(X) )  } , \mathfrak{P}_M ) = \Lambda/ (p^{M \cdot \mathrm{ord}_{\mathfrak{P}}(  f_{\kappa}(X))}, \mathfrak{P}_M ) ,$$
we obtain
\begin{equation} \label{eqn:j-lambda-M-e-ord}
\partial^{(0)} (\ks^{\mathrm{Kato},(\mathfrak{P}_M)})  = M \cdot  e(S_{\mathfrak{P}_M}/\mathbb{Z}_p) \cdot \mathrm{ord}_{\mathfrak{P}} (  f_{\kappa}(X)) + O(1)
\end{equation}
from the properties of $M$ (Lemma \ref{lem:higher-power}.(2) and (3))
where $O(1)$ means an integer bounded independently of $M$.
By a property of $M$ (Lemma \ref{lem:higher-power}.(4)) again, we also have
\begin{equation} \label{eqn:length-control}
 \dfrac{r(S_{\mathfrak{P}_M}/\mathbb{Z}_p)}{e(S_{\mathfrak{P}_M}/\mathbb{Z}_p)} \cdot \mathrm{length}_{S_{\mathfrak{P}_M}} \mathrm{Sel}_0(\mathbb{Q}, (T \otimes S_{\mathfrak{P}_M})^* )
=
\mathrm{length}_{\mathbb{Z}_p} \mathrm{Sel}_0(\mathbb{Q}, (T \otimes \Lambda)^* )[\mathfrak{P}_M]
 + O(1) .
\end{equation}
Following the computation in \cite[Page 67]{mazur-rubin-book}, we obtain
\begin{align} \label{eqn:length-comparison-pM}
\begin{split}
& \mathrm{length}_{\mathbb{Z}_p} \mathrm{Sel}_0(\mathbb{Q}, (T \otimes \Lambda)^* )[\mathfrak{P}_M] \\
&  = M \cdot r(S_{\mathfrak{P}_M}/\mathbb{Z}_p) \cdot \mathrm{ord}_{\mathfrak{P}} ( \mathrm{char}_{\Lambda} (\mathrm{Sel}_0(\mathbb{Q}_{\infty}, E[p^\infty] )^\vee) ) + O(1) .
 \end{split}
\end{align}
We now use the Iwasawa main conjecture localized at $\mathfrak{P}$.
In other words, we assume
\begin{equation} \label{eqn:imc-P_M}
\mathrm{ord}_{\mathfrak{P}} (  f_{\kappa}(X)) = \mathrm{ord}_{\mathfrak{P}} ( \mathrm{char}_{\Lambda} (\mathrm{Sel}_0(\mathbb{Q}_{\infty}, E[p^\infty] )^\vee) ) .
\end{equation}
Combining all the computations, we obtain the following equalities up to some constants bounded independently of $M$.
\begin{align*}
& M \cdot e(S_{\mathfrak{P}_M}/\mathbb{Z}_p) \cdot \mathrm{ord}_{\mathfrak{P}} ( f_{\kappa}(X)) + O(1) \\
& = \partial^{(0)} (\ks^{\mathrm{Kato},(\mathfrak{P}_M)})  + O(1)  & (\ref{eqn:j-lambda-M-e-ord})  \\
& = j_{\mathfrak{P}_M}+ \lambda(1, S_{\mathfrak{P}_M}) + O(1)  & (\ref{eqn:j-lambda-partial-zero})   \\
& =  j_{\mathfrak{P}_M} +  \dfrac{e(S_{\mathfrak{P}_M}/\mathbb{Z}_p)}{r(S_{\mathfrak{P}_M}/\mathbb{Z}_p)} \cdot \mathrm{length}_{\mathbb{Z}_p} \mathrm{Sel}_0(\mathbb{Q}, (T \otimes \Lambda)^* )[\mathfrak{P}_M] +O(1) & (\ref{eqn:length-control}) \\
& =  j_{\mathfrak{P}_M} + 
M \cdot e(S_{\mathfrak{P}_M}/\mathbb{Z}_p) \cdot \mathrm{ord}_{\mathfrak{P}} ( \mathrm{char}_{\Lambda} (\mathrm{Sel}_0(\mathbb{Q}_{\infty}, E[p^\infty] )^\vee) ) +O(1)
 & (\ref{eqn:length-comparison-pM}) \\
& =  j_{\mathfrak{P}_M} + 
M \cdot e(S_{\mathfrak{P}_M}/\mathbb{Z}_p) \cdot \mathrm{ord}_{\mathfrak{P}} (  f_{\kappa}(X)) +O(1)
& (\ref{eqn:imc-P_M})
\end{align*}
This computation shows that $j_{\mathfrak{P}_M}$ is also a constant independent of $M$.
Since $p^M = \mathfrak{m}^{M \cdot e(S_{\mathfrak{P}_M}/\mathbb{Z}_p)}_{S_{\mathfrak{P}_M}}$ in $S_{\mathfrak{P}_M}$ and $j_{\mathfrak{P}_M}$ is constant,
we are able to choose a large $M$ so that 
$$j_{\mathfrak{P}_M} < M \cdot e(S_{\mathfrak{P}_M}/\mathbb{Z}_p) .$$
With this choice of $M$, $\ks^{ \mathrm{Kato}, (\mathfrak{P}_M) } \pmod{p^M}$ does not vanish since $\partial^{(\infty)} (\ks^{\mathrm{Kato},(\mathfrak{P}_M)}) = j_{\mathfrak{P}_M}$.
Also, since
$$\ks^{ \mathrm{Kato}, (\mathfrak{P}_M) } \equiv \ks^{ \mathrm{Kato}, (\mathfrak{P}) } \pmod{p^M}$$
 by definition,
the mod $p^M$ non-vanishing of $\ks^{ \mathrm{Kato}, (\mathfrak{P}_M) }$ is equivalent to the mod $p^M$ non-vanishing of $\ks^{ \mathrm{Kato}, (\mathfrak{P}) }$.
Thus, we obtain Theorem \ref{thm:main-technical} for $\mathfrak{P} \neq p\Lambda$.

The same argument works for $\mathfrak{P} = p\Lambda$ by taking $\mathfrak{P}_{M} = (X^{M}+p)\Lambda$.

\section{The structure of Selmer groups: Proof of Theorem \ref{thm:refined-bsd}} \label{sec:proof-refined-bsd}
The goal of this section is to give a complete proof of Theorem \ref{thm:refined-bsd}.
In the Kolyvagin system argument (Proposition \ref{prop:structure-fine-selmer-p^k}), it is essential to change the local conditions of $p$-strict Selmer groups at carefully chosen Kolyvagin primes in order to reduce the number of generators of $p$-strict Selmer groups.
We develop a similar reduction argument for classical Selmer groups by using the comparison with the argument for $p$-strict Selmer groups.

We summarize the strategy of the proof before going through the details.
We first measure the difference between $n$-transverse Selmer groups and $n$-transverse $p$-strict Selmer groups for every $n \in \mathcal{N}_1$ 
by using the global Poitou--Tate duality as in (\ref{eqn:kappa_n-widedelta_n-sequence}).
Then we measure the difference between the valuations of $\kappa^{\mathrm{Kato}}_n$ and $\widedelta_n$ for every $n \in \mathcal{N}_1$ as in (\ref{eqn:kappa_n-widedelta_n}), and these two differences match perfectly well. 
Thus, we would like to transplant the structure theorem of $p$-strict Selmer groups (Theorem \ref{thm:kato-kolyvagin-main}) to the case of classical Selmer groups.
However, an ``half" of $\kn$ are forced to vanish by functional equation (Proposition \ref{prop:vanishing-delta-n}).
Fortunately, by using Flach's generalized Cassels--Tate pairing (Theorem \ref{thm:cassels-tate}), it suffices to know the \emph{other} half of $\kn$ in order to determine the structure of Selmer groups.
In particular, the computation shows that the vanishing half of $\kn$ \emph{helps} to deduce the non-vanishing of the other half of $\kn$.
Of course, the argument also contains careful choices of Kolyvagin primes based on Chebotarev density theorem (Proposition \ref{prop:non-vanishing-kappa-n}).

This strategy can be viewed as a structural refinement of the four term exact sequence argument comparing two different main conjectures in Iwasawa theory (e.g. \cite[$\S$17.13]{kato-euler-systems}).

\subsection{A generalized Cassels--Tate pairing} \label{subsec:cassels-tate}
We use the following variant of the generalized Cassels--Tate pairing\cite{flach-cassels-tate, howard-kolyvagin}.
\begin{thm}[Flach, Howard] \label{thm:cassels-tate}
\begin{enumerate}
\item Let $s, t$ be positive integers with $s+t\leq k$. Then there exists a pairing
$$\mathrm{Sel}_{\mathcal{F}}(\mathbb{Q} ,T/p^sT) \times \mathrm{Sel}_{\mathcal{F}^*}(\mathbb{Q}, E[p^t]) \to \mathbb{Z}/p^k\mathbb{Z}$$
whose kernels on the left and right are the image of
\[
\xymatrix@R=0em{
\mathrm{Sel}_{\mathcal{F}}(\mathbb{Q} ,T/p^{s+t}T) \ar[r] & \mathrm{Sel}_{\mathcal{F}}(\mathbb{Q} ,T/p^{s}T) , \\
\mathrm{Sel}_{\mathcal{F}^*}(\mathbb{Q}, E[p^{s+t}]) \ar[r]^-{p^s \cdot } & \mathrm{Sel}_{\mathcal{F}^*}(\mathbb{Q}, E[p^t]) .
}
\]
\item Let $n \in \mathcal{N}_k$.
Then we have
$$\mathrm{Sel}_{n}(\mathbb{Q}, E[I_{n}] ) \simeq \left( \mathbb{Z}_p / I_{n}\mathbb{Z}_p \right)^{r_{n}} \oplus M_{n} \oplus M_{n} $$
where 
$r_{n}$ is a non-negative integer and $M_{n}$ is a finite abelian $p$-group.
In particular, we can choose $r_{n}$ is 0 or 1.
If we further assume that $\mathrm{length}_{\mathbb{Z}_p} \mathrm{Sel}_{n}(\mathbb{Q}, E[I_{n}] ) < \mathrm{length}_{\mathbb{Z}_p} (\mathbb{Z}_p/I_n \mathbb{Z}_p)$, then $r_n = 0$.
\end{enumerate}
\end{thm}
\begin{proof}
\begin{enumerate}
\item See \cite[Prop. 1.4.1]{howard-kolyvagin}. See also \cite{flach-cassels-tate} for the detailed construction.
\item This is a simple variant of \cite[Thm. 1.4.2]{howard-kolyvagin}.
\end{enumerate}
\end{proof}

\subsection{Kurihara numbers and Kato's Kolyvagin systems}
We write 
\begin{align*}
\mathrm{coker} ( \mathrm{loc}^s_p(\mathrm{Sel}_{\mathrm{rel}}) )^\vee & = \left( \dfrac{  \mathrm{H}^1_{/f}(\mathbb{Q}_p, T)  }{ \mathrm{loc}^s_p \left( \mathrm{Sel}_{\mathrm{rel},n}(\mathbb{Q}, T) \right) } \right)^\vee  , \\
\mathrm{coker} ( \mathrm{loc}^s_p(\mathrm{Sel}_{\mathrm{rel},n}) )^\vee & = \left( \dfrac{  \mathrm{H}^1_{/f}(\mathbb{Q}_p, T/I_nT)  }{ \mathrm{loc}^s_p \left( \mathrm{Sel}_{\mathrm{rel},n}(\mathbb{Q}, T/I_nT) \right) } \right)^\vee
\end{align*}
for convenience. These modules are cyclic.
By the global Poitou--Tate duality, we have the exact sequences
\begin{equation} \label{eqn:kappa_n-widedelta_n-sequence}
\begin{split}
\xymatrix@R=0em@C=1em{
		0 \ar[r] & \mathrm{Sel}_{0}(\mathbb{Q}, E[p^\infty]) 
		 \ar[r]  & \mathrm{Sel}(\mathbb{Q}, E[p^\infty]) 
       	\ar[r]	^-{\mathrm{loc}_p} & \mathrm{coker} ( \mathrm{loc}^s_p(\mathrm{Sel}_{\mathrm{rel}}) )^\vee \ar[r]   & 0 , \\
       			0 \ar[r] & \mathrm{Sel}_{0,n}(\mathbb{Q}, E[I_n]) 
		 \ar[r]  & \mathrm{Sel}_n(\mathbb{Q}, E[I_n]) 
       	\ar[r]	^-{\mathrm{loc}_p} & \mathrm{coker} ( \mathrm{loc}^s_p(\mathrm{Sel}_{\mathrm{rel},n}) )^\vee \ar[r]   & 0  .
}
\end{split}
\end{equation}

Following (\ref{eqn:exp-loc}), we have the map
\begin{equation} \label{eqn:dual-exponential-valuation}
\begin{split}
\xymatrix@R=0em{
\mathrm{Sel}_{\mathrm{rel},n}(\mathbb{Q}, T/I_nT) \ar[r]^-{\mathrm{loc}^s_p} & \mathrm{H}^1_{/f}(\mathbb{Q}_p, T/I_nT) \ar[r]^-{\xi \circ \mathrm{exp}^*_{\omega_E}}_-{\simeq} & \mathbb{Z}_p/I_n\mathbb{Z}_p  \\
\kappa^{\mathrm{Kato}}_n \ar@{|->}[r] & \mathrm{loc}^s_p \kappa^{\mathrm{Kato}}_n \ar@{|->}[r] & p^t \cdot \widedelta_n  .
}
\end{split}
\end{equation}
We ignore $u \in (\mathbb{Z}_p /  I_n \mathbb{Z}_p)^\times$ in Theorem \ref{thm:from-ks-to-kn} since we focus only on the divisibilities of the elements.
For notational convenience, we write
\begin{align*}
\mathrm{ord}_p(\kappa^{\mathrm{Kato}}_n) & = \mathrm{min} \lbrace j : \kappa^{\mathrm{Kato}}_n \in p^j  \mathrm{Sel}_{\mathrm{rel},n}(\mathbb{Q}, T/I_nT) \rbrace  , \\
\mathrm{ord}_p(\mathrm{loc}^s_p \kappa^{\mathrm{Kato}}_n) & = \mathrm{min} \lbrace j : \mathrm{loc}^s_p \kappa^{\mathrm{Kato}}_n \in p^j  \mathrm{H}^1_{/f}(\mathbb{Q}_p, T/I_nT)  \rbrace , \\
\mathrm{ord}_p(\widedelta_n) & = \mathrm{min} \lbrace j : \widedelta_n \in p^j \mathbb{Z}_p/I_n \mathbb{Z}_p \rbrace .
\end{align*}
Then (\ref{eqn:dual-exponential-valuation}) shows that
\begin{align} \label{eqn:kappa_n-widedelta_n}
\begin{split}
& \mathrm{ord}_p(\kappa^{\mathrm{Kato}}_1)
+ \mathrm{length}_{\mathbb{Z}_p} ( \mathrm{coker} ( \mathrm{loc}^s_p(\mathrm{Sel}_{\mathrm{rel},n}) )^\vee )
= \mathrm{ord}_p(\mathrm{loc}^s_p \kappa^{\mathrm{Kato}}_1) = \mathrm{ord}_p(\widedelta_1) +t , \\
& \mathrm{ord}_p(\kappa^{\mathrm{Kato}}_n)
+ \mathrm{length}_{\mathbb{Z}_p} ( \mathrm{coker} ( \mathrm{loc}^s_p(\mathrm{Sel}_{\mathrm{rel},n}) )^\vee )
= \mathrm{ord}_p(\mathrm{loc}^s_p \kappa^{\mathrm{Kato}}_n) = \mathrm{ord}_p(\widedelta_n) +t .
\end{split} 
\end{align}
To sum up, we have the following statement.
\begin{prop} \label{prop:behavior-kappa-n-localization}
Assume that the Manin constant is prime to $p$.
Let $n \in \mathcal{N}_k$.
The followings are equivalent.
\begin{enumerate}
\item $ p^t \cdot\widedelta_n =0$.
\item $\kappa^{\mathrm{Kato}}_n \in \mathrm{Sel}_{n}(\mathbb{Q}, T/I_nT)$.
\item $\mathrm{ord}_p(\kappa^{\mathrm{Kato}}_n)
+ \mathrm{length}_{\mathbb{Z}_p}  ( \mathrm{coker} ( \mathrm{loc}^s_p(\mathrm{Sel}_{\mathrm{rel},n}) )^\vee )  \geq   \mathrm{length}_{\mathbb{Z}_p}  \left( \mathbb{Z}_p/I_n\mathbb{Z}_p \right)$.
\end{enumerate}
\end{prop}
\begin{proof}
$(1) \Leftrightarrow (2)$: It follows from that $ p^t \cdot\widedelta_n  = 0$ is equivalent to $\mathrm{loc}^s_p (\kappa^{\mathrm{Kato}}_n) = 0$.\\
$(1) \Leftrightarrow (3)$: It follows from (\ref{eqn:kappa_n-widedelta_n}). 
\end{proof}

\begin{lem} \label{lem:vanishing-localization}
Assume that  $\ks^{\mathrm{Kato}}$ is non-trivial.
Then
$$\partial^{(\infty)}(\ks^{\mathrm{Kato}}) = \partial^{(\infty)}(\mathrm{loc}^s_p\ks^{\mathrm{Kato}}) $$
where
$\partial^{(\infty)}(\mathrm{loc}^s_p\ks^{\mathrm{Kato}}) = \mathrm{min} \left\lbrace \mathrm{ord}_\pi(\mathrm{loc}^s_p\kappa^{\mathrm{Kato}}_n) : n \in \mathcal{N}_1 \right\rbrace $.
\end{lem}
\begin{proof}
Suppose that
$\partial^{(\infty)}(\ks^{\mathrm{Kato}}) < \partial^{(\infty)}(\mathrm{loc}^s_p\ks^{\mathrm{Kato}}) $.
Let $n \in \mathcal{N}_1$ satisfying
$\mathrm{ord}_\pi \kappa^{\mathrm{Kato}}_n = \partial^{(\infty)}(\ks^{\mathrm{Kato}}) < \infty$.
By Theorem \ref{thm:kato-kolyvagin-main}, we have $\mathrm{Sel}_{0,n}(\mathbb{Q}, E[I_n]) = 0$.
Thus, we obtain $\mathrm{Sel}_{n}(\mathbb{Q}, E[I_n]) \simeq \mathrm{coker} ( \mathrm{loc}^s_p(\mathrm{Sel}_{\mathrm{rel},n}) )^\vee$ from (\ref{eqn:kappa_n-widedelta_n-sequence}).
By (\ref{eqn:kappa_n-widedelta_n}) and the inequality, it is always non-trivial.
Therefore, $\mathrm{Sel}_{n}(\mathbb{Q}, E[I_n])[p] \simeq \mathrm{Sel}_{n}(\mathbb{Q}, E[p])$ is non-trivial for every $n \in \mathcal{N}_1$.
However, we can always find $n_0 \in \mathcal{N}_1$ such that 
$\mathrm{Sel}_{n_0}(\mathbb{Q}, E[p]) = 0$
by applying Proposition \ref{prop:chebotarev} and Lemma \ref{lem:surjectivity-at-ell}. The conclusion follows.
\end{proof}

\subsection{Proof of Theorem \ref{thm:refined-bsd}.(1): the corank part} \label{subsec:selmer-corank-part}
In this subsection, we prove Theorem \ref{thm:refined-bsd}.(1), i.e.
if $\overline{\rho}$ is surjective, the Manin constant is prime to $p$, and $\mathrm{ord}(\kn) < \infty$, then
$$\mathrm{ord} (\kn) = \mathrm{cork}_{\mathbb{Z}_p} (\mathrm{Sel}(\mathbb{Q}, E[p^\infty])) .$$

\subsubsection{When the corank is zero}
One direction follows from the theorem of Gross--Zagier and Kolyvagin and Kato \cite{gross-zagier-original,kolyvagin-euler-systems,kato-euler-systems}.
\begin{thm}[Gross--Zagier, Kolyvagin, Kato]
If $L(E, 1) \neq 0$, then $\mathrm{Sel}(\mathbb{Q}, E[p^\infty]) $ is finite.
\end{thm}
We now assume that $\mathrm{Sel}(\mathbb{Q}, E[p^\infty])$ is finite.
Then both
$\mathrm{Sel}_0(\mathbb{Q}, E[p^\infty])$
and
$\mathrm{coker} ( \mathrm{loc}^s_p(\mathrm{Sel}_{\mathrm{rel}}) )^\vee$
are finite due to (\ref{eqn:kappa_n-widedelta_n-sequence}).
Since $\kn$ is non-zero, $\ks^{\mathrm{Kato}}$ is also non-trivial (Proposition \ref{prop:ks-kn-equi-non-vanishing}).
Therefore, we have $\kappa^{\mathrm{Kato}}_1 \neq 0$ by the finiteness of $\mathrm{Sel}_0(\mathbb{Q}, E[p^\infty])$ and Corollary \ref{cor:kato-kolyvagin-main}.

Suppose $ p^t \cdot \widedelta_1 = 0$. Then Proposition \ref{prop:behavior-kappa-n-localization} implies that
$$\kappa^{\mathrm{Kato}}_1 \in \mathrm{Sel}(\mathbb{Q}, T) = \mathrm{ker} (\mathrm{loc}^s_p : \mathrm{Sel}_{\mathrm{rel}}(\mathbb{Q}, T) \to \mathrm{H}^1_{/f}(\mathbb{Q}_p, T)) .$$
Write $\kappa^{\mathrm{Kato},(k)}_1 = \kappa^{\mathrm{Kato}}_1 \pmod{p^k}$ for an integer $k > 0$.
We assume that $k$ is large enough to have $\kappa^{\mathrm{Kato},(k)}_1 \neq 0$.
Then we have
\begin{align*}
\kappa^{\mathrm{Kato},(k)}_1 & \in \mathrm{Sel}(\mathbb{Q}, T) / p^k \\
 & \subseteq \mathrm{Sel}(\mathbb{Q}, T/ p^kT)  \\
 & \simeq \mathrm{Sel}(\mathbb{Q}, E[p^k])  \\
 & \simeq \mathrm{Sel}(\mathbb{Q}, E[p^\infty])[p^k] 
\end{align*}
where the last isomorphism follows from \cite[Lem. 3.5.3]{mazur-rubin-book}.
Since $\kappa^{\mathrm{Kato},(k)}_1$ generates $p^j \mathcal{H}'(1)$ (Proposition \ref{prop:kolyvagin-system-location}), we have
$$\langle \kappa^{\mathrm{Kato},(k)}_1 \rangle \simeq \mathbb{Z}_p/p^{k-\lambda(1, E[p^k]) - j} \mathbb{Z}_p$$
where $\lambda(1, E[p^k]) = \mathrm{length}_{\mathbb{Z}_p} \mathrm{Sel}_0(\mathbb{Q}, E[p^k])$.
Since $\mathrm{Sel}_0(\mathbb{Q}, E[p^\infty])$ is finite, $\lambda(1, E[p^k])$ stabilizes as $k \to \infty$.
Hence, the size of $\langle \kappa^{\mathrm{Kato},(k)}_1 \rangle$ can be arbitrarily large as $k$ increases.
This shows that $\mathrm{Sel}(\mathbb{Q}, E[p^\infty])$ must be infinite, so we get contradiction.

Since the non-vanishing properties of $\widedelta_1$ and  $p^t \cdot \widedelta_1$ are equivalent, we are done.

\subsubsection{When the corank is positive} \label{subsubsec:three-conditions-on-n}
Let $k \gg 0$ be a sufficiently large integer.
Let $n \in \mathcal{N}_k$ such that
\begin{itemize}
\item $I_n = p^k \mathbb{Z}_p$,
\item $\mathrm{ord}(\ks^{\mathrm{Kato}}) = \nu(n) = \mathrm{cork}_{\mathbb{Z}_p} \mathrm{Sel}_{0}(\mathbb{Q}, E[p^\infty])$, and
\item $\kappa^{\mathrm{Kato}}_n \neq 0$.
\end{itemize}
These conditions mean that our choice of $n$ is compatible with that in the Kolyvagin system argument \cite[Prop. 4.5.8]{mazur-rubin-book}.
By Proposition \ref{prop:kolyvagin-system-location}, we have
$$\langle \kappa^{\mathrm{Kato}}_n \rangle = p^{\lambda(n, E[I_n]) + j} \mathrm{Sel}_{\mathrm{rel}, n}(\mathbb{Q}, T/I_nT) \simeq \mathbb{Z}_p/p^{k-\lambda(n, E[I_n]) - j} \mathbb{Z}_p .$$
Since $k \gg 0$, we may assume that $\mathrm{Sel}_{0,n}(\mathbb{Q}, E[I_n]) \simeq \mathrm{Sel}_0(\mathbb{Q}, E[p^\infty])_{/\mathrm{div}}$.
Therefore, $\lambda(n, E[I_n])$ is actually bounded independent of $n$ as long as 
we choose $n$ satisfying the above three conditions.
Also, $j = \partial^{(\infty)} (\ks^{\mathrm{Kato}})$ since $k \gg 0$; thus, $j$ is also independent of $n$.

With our choice of $n$, we have exact sequence
\begin{equation} \label{eqn:chebotarev-density-argument}
\begin{split}
\xymatrix{
0 \ar[r] & 
\mathrm{Sel}_{n\textrm{-}\mathrm{str}}(\mathbb{Q}, E[I_n]) \ar[r] \ar@{=}[d]_-{\textrm{Lem. \ref{lem:surjectivity-at-ell}}} & 
\mathrm{Sel}(\mathbb{Q}, E[I_n]) \ar[r] & 
\bigoplus_{\ell \vert n} E(\mathbb{Q}_\ell) / I_n E(\mathbb{Q}_\ell) \ar[r] & 0  \\
 & \mathrm{Sel}_{n}(\mathbb{Q}, E[I_n]) .
}
\end{split}
\end{equation}
Following Proposition \ref{prop:behavior-kappa-n-localization}, $\mathrm{loc}^s_p \kappa^{\mathrm{Kato}}_n = 0$ if and only if $\kappa^{\mathrm{Kato}}_n \in \mathrm{Sel}_n(\mathbb{Q}, E[I_n])$. We consider two possible cases separately.
\begin{enumerate}
\item $\mathrm{loc}^s_p \kappa^{\mathrm{Kato}}_n \neq 0$.
\item $\mathrm{loc}^s_p \kappa^{\mathrm{Kato}}_n = 0$.
\end{enumerate}

\subsubsection{}
\underline{Suppose that  $\mathrm{loc}^s_p \kappa^{\mathrm{Kato}}_n \neq 0$.}
Then
\begin{align*}
& \mathrm{length}_{\mathbb{Z}_p} ( \mathrm{Sel}_{n}(\mathbb{Q}, E[I_n]) ) \\
& = \mathrm{length}_{\mathcal{O}} ( \mathrm{Sel}_{0,n}(\mathbb{Q}, E[I_n]) )
+ \mathrm{length}_{\mathcal{O}} ( \mathrm{coker} ( \mathrm{loc}^s_p(\mathrm{Sel}_{\mathrm{rel},n}) )^\vee ) \\
& =\mathrm{ord}_p ( \kappa^{\mathrm{Kato}}_n) - \partial^{(\infty)}( \ks^{\mathrm{Kato}})
+ \mathrm{length}_{\mathbb{Z}_p} ( \mathrm{coker} ( \mathrm{loc}^s_p(\mathrm{Sel}_{\mathrm{rel},n}) )^\vee ) \\
& = \mathrm{ord}_\pi (\mathrm{loc}^s_p \kappa^{\mathrm{Kato}}_n) - \partial^{(\infty)}(\mathrm{loc}^s_p \ks^{\mathrm{Kato}}) < k-j .
\end{align*}
This computation with (\ref{eqn:chebotarev-density-argument}) implies that 
$\mathrm{cork}_{\mathbb{Z}_p} \mathrm{Sel}(\mathbb{Q}, E[p^\infty]) = \nu(n)$.
Therefore, we have
$$\mathrm{ord} ( \mathrm{loc}^s_p \ks^{\mathrm{Kato}} ) =
\mathrm{ord} (  \ks^{\mathrm{Kato}} ) = \nu(n) =
\mathrm{cork}_{\mathcal{O}} \mathrm{Sel}_{0}(\mathbb{Q}, E[p^\infty]) =  
\mathrm{cork}_{\mathcal{O}} \mathrm{Sel}(\mathbb{Q}, E[p^\infty]) .$$
We also have $p^t \cdot \widedelta_n \neq 0$ since $k \gg 0$, so
$\mathrm{ord} ( \mathrm{loc}^s_p \ks^{\mathrm{Kato}} ) = \mathrm{ord} ( \kn ) = \nu(n)$.

\subsubsection{}
\underline{Suppose that  $\mathrm{loc}^s_p \kappa^{\mathrm{Kato}}_n = 0$.}
Thus, we have
$\kappa^{\mathrm{Kato}}_n \in  \mathrm{Sel}_n(\mathbb{Q}, E[I_n])$. Proposition \ref{prop:kolyvagin-system-location}  implies
\begin{align} \label{eqn:kappa_n-order}
\begin{split}
\mathrm{ord}_p \kappa^{\mathrm{Kato}}_n & = k - j + \lambda(n, E[I_n]) \\
& = k - \partial^{(\infty)}( \ks^{\mathrm{Kato}}) - \mathrm{length}_{\mathbb{Z}_p}  \mathrm{Sel}_{0}(\mathbb{Q}, E[p^\infty])_{/\mathrm{div}} .
\end{split}
\end{align}
Suppose that $\mathrm{cork}_{\mathbb{Z}_p} \mathrm{Sel}(\mathbb{Q}, E[p^\infty]) = \nu(n)$.
Then (\ref{eqn:chebotarev-density-argument}) implies that
$\mathrm{Sel}_{n}(\mathbb{Q}, E[I_n]) \simeq \mathrm{Sel}(\mathbb{Q}, E[p^\infty])_{/\mathrm{div}}$,
so its length is bounded independently on $n$ as long as 
we choose $n$ satisfying the three conditions in \S\ref{subsubsec:three-conditions-on-n} as before.
However, it is impossible because $k$ can be arbitrarily large in (\ref{eqn:kappa_n-order}).

Thus, we have $\mathrm{cork}_{\mathbb{Z}_p} \mathrm{Sel}(\mathbb{Q}, E[p^\infty]) = \nu(n) +1 $ by (\ref{eqn:kappa_n-widedelta_n-sequence}).
By (\ref{eqn:chebotarev-density-argument}) again, $\mathrm{Sel}_n(\mathbb{Q}, E[I_n])$ is of rank one over $\mathbb{Z}_p/I_n\mathbb{Z}_p$.
Since $k \gg 0$, we may assume that
$\mathrm{length}_{\mathbb{Z}_p} \mathrm{Sel}_{0,n}(\mathbb{Q}, E[I_n]) < k$.
We also have
$$\mathrm{Sel}_{0,n}(\mathbb{Q}, E[I_n]) 
\subseteq 
\mathrm{Sel}_{n}(\mathbb{Q}, E[I_n]) 
\subseteq
\mathrm{Sel}_{\mathrm{rel},n}(\mathbb{Q}, E[I_n]) \simeq \mathbb{Z}_p/I_n\mathbb{Z}_p \oplus \mathrm{Sel}_{0,n}(\mathbb{Q}, E[I_n])$$
due to Theorem \ref{thm:splitting-mazur-rubin} and the self-duality $E[I_n]\simeq T/I_nT$,
so the $\mathbb{Z}_p/I_n \mathbb{Z}_p$-component in $\mathrm{Sel}_{\mathrm{rel},n}(\mathbb{Q}, E[I_n])$
is also contained in $\mathrm{Sel}_{n}(\mathbb{Q}, E[I_n])$.

By using Chebotarev density argument (Proposition \ref{prop:non-vanishing-kappa-n}), we choose a useful prime $\ell$ for $\kappa^{\mathrm{Kato}}_n$ such that
\begin{itemize}
\item $I_{n\ell} = I_n = p^k\mathbb{Z}_p$,
\item $\kappa^{\mathrm{Kato}}_{n\ell} \neq 0$, and
\item the restriction map $\mathrm{Sel}_{n}(\mathbb{Q}, E[I_n]) \to E(\mathbb{Q}_\ell) \otimes \mathbb{Z}_p/I_n\mathbb{Z}_p$ is surjective.
\end{itemize}
The last condition implies that
$\mathrm{Sel}_{n\ell }(\mathbb{Q}, E[I_n]) = \mathrm{Sel}_{n, \ell\textrm{-str}}(\mathbb{Q}, E[I_n])$ due to Lemma \ref{lem:surjectivity-at-ell}.
Hence, $\mathrm{Sel}_{n\ell }(\mathbb{Q}, E[I_n])$ is of rank zero over $\mathbb{Z}_p/I_n\mathbb{Z}_p$, and
$\mathrm{Sel}_{n\ell }(\mathbb{Q}, E[I_n]) \subseteq \mathrm{Sel}(\mathbb{Q}, E[p^\infty])_{/\mathrm{div}}$.
Here, we regard $\mathrm{Sel}(\mathbb{Q}, E[p^\infty])_{/\mathrm{div}}$ as a submodule of $\mathrm{Sel}(\mathbb{Q}, E[p^\infty])$.
This shows that $\kappa^{\mathrm{Kato}}_{n\ell}  \not\in \mathrm{Sel}_{n\ell }(\mathbb{Q}, E[I_n])$, so
$\mathrm{loc}^s_p \kappa^{\mathrm{Kato}}_{n\ell} \neq 0$.
Therefore, we have
$$\mathrm{ord} ( \mathrm{loc}^s_p \ks^{\mathrm{Kato}} ) =
\mathrm{ord} (  \ks^{\mathrm{Kato}} ) +1 = \nu(n) +1 =
\mathrm{cork}_{\mathbb{Z}_p} \mathrm{Sel}_{0}(\mathbb{Q}, E[p^\infty]) +1 =  
\mathrm{cork}_{\mathbb{Z}_p} \mathrm{Sel}(\mathbb{Q}, E[p^\infty]) .$$
We also have $p^t \cdot \widedelta_{n\ell} \neq 0$ since $k \gg 0$; thus, we have
$$\mathrm{ord} (\kn)= \mathrm{ord} ( p^t \cdot \kn) = \mathrm{ord} ( \mathrm{loc}^s_p \ks^{\mathrm{Kato}} )  = \nu(n)+1 .$$

\subsection{Proof of Theorem \ref{thm:refined-bsd}.(2): the structure of ``Tate--Shafarevich" groups} \label{subsec:selmer-corank-zero}
In this subsection, we determine the structure of $\mathrm{Sel}(\mathbb{Q}, E[p^\infty])_{/\mathrm{div}}$ in terms of $\kn$.

\subsubsection{}
For $k \gg 0$, we fix $n \in \mathcal{N}_1$ such that
\begin{itemize}
\item $\mathbb{Z}_p/I_n\mathbb{Z}_p = \mathbb{Z}/p^k\mathbb{Z}$
\item $\widedelta_n \neq 0$
\item $\nu(n) =\mathrm{ord}(\kn)$
\item $\mathrm{ord}_p  \widedelta_n = \partial^{(\nu(n))} (\kn)$
\item $\mathrm{Sel}_{n}(\mathbb{Q}, E[I_n]) = \mathrm{Sel}_{n\textrm{-}\mathrm{str}}(\mathbb{Q}, E[I_n]) =\mathrm{Sel}(\mathbb{Q}, E[p^\infty])_{/\mathrm{div}}$
\item $\mathrm{length}_{\mathbb{Z}_p} \mathrm{Sel}(\mathbb{Q}, E[p^\infty])_{/\mathrm{div}} < k = \mathrm{length}_{\mathbb{Z}_p}\mathbb{Z}_p/I_n\mathbb{Z}_p$
\end{itemize}

\subsubsection{}
By Flach's generalized Cassels--Tate pairing (Theorem \ref{thm:cassels-tate}) and the fifth condition of $n$ above, we have
$$\mathrm{Sel}_n(\mathbb{Q}, E[I_n]) \simeq M_n \oplus M_n$$
for a finite abelian $p$-group $M_n$.
This determines the structure of $\mathrm{Sel}_n(\mathbb{Q}, E[I_n])$ uniquely
due to the following reasons:
\begin{itemize}
\item we know the structure of $\mathrm{Sel}_{0,n}(\mathbb{Q}, E[I_n])$ by applying Proposition \ref{prop:structure-fine-selmer-p^k} to
$\ks^{\mathrm{Kato},n, (k)} = \left\lbrace \kappa^{\mathrm{Kato},n, (k)}_m \right\rbrace$ with $\kappa^{\mathrm{Kato},n, (k)}_m = \kappa^{\mathrm{Kato},(k)}_{n \cdot m}$ where $m \in \mathcal{N}_1$ with $(m,n)=1$.
\item as in (\ref{eqn:kappa_n-widedelta_n-sequence}), we have exact sequence \[
\xymatrix{
	0 \ar[r] & \mathrm{Sel}_{0,n}(\mathbb{Q}, E[I_n]) 
		 \ar[r]  & \mathrm{Sel}_n(\mathbb{Q}, E[I_n]) 
       	\ar[r]	^-{\mathrm{loc}_p} & \mathrm{coker} ( \mathrm{loc}^s_p(\mathrm{Sel}_{\mathrm{rel},n}) )^\vee \ar[r]   & 0  .
}
\]
\item $\mathrm{coker} ( \mathrm{loc}^s_p(\mathrm{Sel}_{\mathrm{rel},n}) )^\vee$ is cyclic.
\end{itemize}

\subsubsection{}
Following Proposition \ref{prop:structure-fine-selmer-p^k},
we write
$$\mathrm{Sel}_{0,n}(\mathbb{Q}, E[p^k]) \simeq \bigoplus_{i \geq 1}  \mathbb{Z} / p^{d_i} \mathbb{Z}$$
with non-negative integers $d_1 \geq d_2 \geq \cdots \geq 0$.

By Chebotarev density argument (Proposition \ref{prop:chebotarev}),
we are able to choose a useful Kolyvagin prime $\ell_1 \in \mathcal{N}_1$ for $\kappa^{\mathrm{Kato}}_n$ such that
\begin{itemize}
\item  $I_{n\ell_1} \mathbb{Z}_p = p^k \mathbb{Z}_p$,
\item $\kappa^{\mathrm{Kato}}_{n\ell_1} \neq 0$,
\item $\mathrm{loc}_{\ell_1} : p^{k-1}\mathrm{Sel}_{\mathrm{rel},n}(\mathbb{Q}, T/p^kT) \to E(\mathbb{Q}_{\ell_1}) \otimes \mathbb{Z}/p^k\mathbb{Z}$ is non-zero, and
\item $\mathrm{loc}_{\ell_1} : p^{d_1 - 1} \mathrm{Sel}_{0,n}(\mathbb{Q}, E[p^k]) \to E(\mathbb{Q}_{\ell_1}) \otimes \mathbb{Z}/p^k\mathbb{Z}$ is non-zero.
\end{itemize}
This choice of $\ell_1$ is compatible with that in the Kolyvagin system argument \cite[Prop. 4.5.8]{mazur-rubin-book}, i.e.
$$\mathrm{Sel}_{0,n\ell_1}(\mathbb{Q}, E[p^k]) \simeq \bigoplus_{i \geq 2}  \mathbb{Z} / p^{d_i} \mathbb{Z} .$$
The third condition implies that 
$$\mathrm{loc}_{\ell_1} : \mathrm{Sel}_{\mathrm{rel},n}(\mathbb{Q}, T/p^kT) \to E(\mathbb{Q}_{\ell_1}) \otimes \mathbb{Z}/p^k\mathbb{Z}$$
is surjective, so we have
$\mathrm{Sel}_{\mathrm{rel},n,\ell_1\textrm{-}\mathrm{str}}(\mathbb{Q}, T/p^kT)
= 
\mathrm{Sel}_{\mathrm{rel},n\ell_1}(\mathbb{Q}, T/p^kT)$
by Lemma \ref{lem:surjectivity-at-ell}.
\subsubsection{}
Consider the inclusions
\begin{align*}
\mathrm{Sel}_{0,n\ell_1}(\mathbb{Q}, E[p^k]) \subseteq
\mathrm{Sel}_{n\ell_1}(\mathbb{Q}, E[p^k]) & \subseteq
\mathrm{Sel}_{\mathrm{rel},n\ell_1}(\mathbb{Q}, E[p^k]) \\
& \simeq \mathrm{Sel}_{0,n\ell_1}(\mathbb{Q}, E[p^k]) \oplus \mathbb{Z}/p^k\mathbb{Z}
\end{align*}
where the last isomorphism follows from Theorem \ref{thm:splitting-mazur-rubin} and the self-duality $E[p^k] \simeq T/p^kT$.

We explicitly characterize the ``free component of $\mathrm{Sel}_{n\ell_1}(\mathbb{Q}, E[p^k])$".

We have $ p^t \cdot \widedelta_{n\ell_1} = 0$ due to Proposition \ref{prop:vanishing-delta-n} and $ p^t \cdot \widedelta_n \neq  0$.
Thus, we have
\begin{align} \label{eqn:kappa_ell_1_local_condition}
\begin{split}
\kappa^{\mathrm{Kato}}_{n\ell_1} & \in \mathrm{Sel}_{n\ell_1}(\mathbb{Q}, T/I_{n\ell_1}T) \\
& \simeq \mathrm{Sel}_{n\ell_1}(\mathbb{Q}, E[I_{n\ell_1}]) \\
& = \mathrm{Sel}_{n\ell_1}(\mathbb{Q}, E[p^k]) .
\end{split}
\end{align}
Proposition \ref{prop:kolyvagin-system-location} implies that
\begin{align*}
\langle \kappa^{\mathrm{Kato}}_{n\ell_1} \rangle & = p^{\lambda(n\ell_1, E[I_{n\ell_1}]) + j} \mathrm{Sel}_{\mathrm{rel}, n\ell_1}(\mathbb{Q}, T / I_{n\ell_1} T) \\
& \simeq \mathbb{Z} / p^{k - \lambda(n\ell_1, E[I_{n\ell_1}]) - j } \mathbb{Z} .
\end{align*}
Write 
$$k_2 = k - \lambda(n\ell_1, E[I_{n\ell_1}]) - j$$
for convenience.
Then we have
\begin{align} \label{eqn:kappa_ell_1_order}
\begin{split}
 p^{\lambda(n\ell_1, E[I_{n\ell_1}]) + j} \mathrm{Sel}_{\mathrm{rel}, n\ell_1}(\mathbb{Q}, T / I_{n\ell_1} T) 
&\simeq p^{\lambda(n\ell_1, E[I_{n\ell_1}]) + j} \mathrm{Sel}_{\mathrm{rel}, n\ell_1}(\mathbb{Q}, E[I_{n\ell_1}])  \\
& \subseteq \mathrm{Sel}_{\mathrm{rel}, \ell_1}(\mathbb{Q}, E[I_{n\ell_1}]) [p^{k - \lambda(n\ell_1, E[I_{n\ell_1}]) + j}] \\
& = \mathrm{Sel}_{\mathrm{rel}, n\ell_1}(\mathbb{Q}, E[p^{k_2}]) .
\end{split}
\end{align}
By (\ref{eqn:kappa_ell_1_local_condition}) and (\ref{eqn:kappa_ell_1_order}), we have
$$\langle \kappa^{\mathrm{Kato}}_{n\ell_1} \rangle \simeq \mathbb{Z}/p^{k_2}\mathbb{Z} \subseteq \mathrm{Sel}_{n\ell_1}(\mathbb{Q}, E[p^{k_2}]) .$$
Since $k$ is sufficiently large, we may assume $d_2 < k_2 = k - \lambda(n\ell_1, E[I_{n\ell_1}]) - j$ without loss of generality.
Thus, $\mathrm{Sel}_{n\ell_1}(\mathbb{Q}, E[p^{k_2}])$ is of rank one over $\mathbb{Z}/p^{k_2}\mathbb{Z}$.

\subsubsection{}
By using Chebotarev density argument (Proposition \ref{prop:non-vanishing-kappa-n}) again, we are able to choose a useful Kolyvagin prime $\ell_2 \in \mathcal{N}_{1}$ for $\kappa^{\mathrm{Kato}}_{n\ell_1}$ such that
\begin{itemize}
\item $I_{n\ell_1 \ell_2} \mathbb{Z}_p = p^{k_2} \mathbb{Z}_p$, 
\item $\kappa^{\mathrm{Kato}}_{n\ell_1\ell_2} \neq 0$,
\item $\mathrm{loc}_{\ell_2} : p^{k_2-1}\mathrm{Sel}_{\mathrm{rel}, n\ell_1}(\mathbb{Q}, T/p^{k_2}T) \to E(\mathbb{Q}_{\ell_2}) \otimes \mathbb{Z} / p^{k_2} \mathbb{Z}$, is non-zero, and
\item $\mathrm{loc}_{\ell_2} : p^{d_2-1} \mathrm{Sel}_{0, n\ell_1}(\mathbb{Q}, E[p^{k_2}]) \to E(\mathbb{Q}_{\ell_2}) \otimes \mathbb{Z} / p^{k_2} \mathbb{Z}$ is non-zero.
\end{itemize}
Since the $\mathbb{Z}/p^{k_2}\mathbb{Z}$-component of $\mathrm{Sel}_{\mathrm{rel}, n\ell_1}(\mathbb{Q}, T/p^{k_2}T)$ is contained in $\mathrm{Sel}_{n\ell_1}(\mathbb{Q}, E[p^{k_2}])$ (with self-duality $E[p^{k_2}] \simeq T/p^{k_2}T$), the third condition on $\ell_2$ implies that we have exact sequence
\[
\xymatrix{
0 \ar[r] & \mathrm{Sel}_{n\ell_1, \ell_2\textrm{-}\mathrm{str}}(\mathbb{Q}, E[p^{k_2}]) \ar[r] & \mathrm{Sel}_{n\ell_1}(\mathbb{Q}, E[p^{k_2}]) \ar[r]^-{\mathrm{loc}_{\ell_2}} & E(\mathbb{Q}_{\ell_2}) \otimes \mathbb{Z}_p / p^{k_2} \mathbb{Z}_p \ar[r] & 0 
}
\]
and isomorphism
 $$\mathrm{Sel}_{n\ell_1, \ell_2\textrm{-}\mathrm{str}}(\mathbb{Q}, E[p^{k_2}]) \simeq \mathrm{Sel}_{n\ell_1 \ell_2}(\mathbb{Q}, E[p^{k_2}]) $$ by Lemma \ref{lem:surjectivity-at-ell}.
In particular,
 $\mathrm{Sel}_{n\ell_1 \ell_2}(\mathbb{Q}, E[p^{k_2}])$ is of rank zero over $\mathbb{Z}_p /p^{k_2} \mathbb{Z}_p$.
By Flach's generalized Cassels--Tate pairing (Theorem \ref{thm:cassels-tate}), we have
$$\mathrm{Sel}_{n\ell_1 \ell_2}(\mathbb{Q}, E[p^{k_2}]) \simeq M_{n\ell_1 \ell_2}  \oplus M_{n\ell_1 \ell_2}$$  for a finite abelian $p$-group $M_{n\ell_1 \ell_2}$.
As before, the structure of $\mathrm{Sel}_{n\ell_1 \ell_2}(\mathbb{Q}, E[p^{k_2}])$ is uniquely determined again
due to the following reasons:
\begin{itemize}
\item we know the structure of $\mathrm{Sel}_{0,n\ell_1 \ell_2}(\mathbb{Q}, E[p^{k_2}])$ by applying Proposition \ref{prop:structure-fine-selmer-p^k} to
$\ks^{\mathrm{Kato},n\ell_1 \ell_2, (k_2)} = \left\lbrace \kappa^{\mathrm{Kato},n\ell_1 \ell_2, (k_2)}_m \right\rbrace$ with $\kappa^{\mathrm{Kato},n\ell_1 \ell_2, (k_2)}_m = \kappa^{\mathrm{Kato},(k_2)}_{n\ell_1 \ell_2 \cdot m}$ where $m \in \mathcal{N}_{1}$ with $(m,n\ell_1 \ell_2)=1$.
\item as in (\ref{eqn:kappa_n-widedelta_n-sequence}), we have exact sequence \[
\xymatrix{
	0 \ar[r] & \mathrm{Sel}_{0,n\ell_1 \ell_2}(\mathbb{Q}, E[p^{k_2}]) 
		 \ar[r]  & \mathrm{Sel}_{n\ell_1 \ell_2}(\mathbb{Q}, E[p^{k_2}]) 
       	\ar[r]	^-{\mathrm{loc}_p} & \mathrm{coker} ( \mathrm{loc}^s_p(\mathrm{Sel}_{\mathrm{rel},n\ell_1 \ell_2 }) )^\vee \ar[r]   & 0  .
}
\]
\item $\mathrm{coker} ( \mathrm{loc}^s_p(\mathrm{Sel}_{\mathrm{rel},n \ell_1 \ell_2}) )^\vee$ is cyclic.
\end{itemize}
\subsubsection{}
We now compare
$\mathrm{Sel}_{n\ell_1 \ell_2}(\mathbb{Q}, E[p^{k_2}]) \simeq M_{n\ell_1\ell_2} \oplus M_{n\ell_1\ell_2}$
with $\mathrm{Sel}_{n}(\mathbb{Q}, E[p^{k}]) \simeq M_n \oplus M_n$ more explicitly.
Write $\mathrm{Sel}_{0,n}(\mathbb{Q}, E[p^{k}]) \simeq M_n \oplus M'_n$ with $M'_n \subseteq M_n$, and then $M_n/M'_n$ is cyclic since $\mathrm{coker} ( \mathrm{loc}^s_p(\mathrm{Sel}_{\mathrm{rel},n }) )^\vee \simeq  M_n/M'_n$.
\begin{itemize}
\item Suppose that $M_n/M'_n$ is not isomorphic to the largest cyclic factor of $M_n$.
Write
$\mathrm{Sel}_{0,n\ell_1 \ell_2}(\mathbb{Q}, E[p^{k_2}]) \simeq M_{n\ell_1\ell_2} \oplus M'_{n\ell_1\ell_2}$.
Then, by the Kolyvagin system argument for $\mathrm{Sel}_{0,n}(\mathbb{Q}, E[p^{k}])$,
it is not difficult to see that $M_{n}/M_{n\ell_1\ell_2}$ is isomorphic to the largest cyclic factor of $M_n$.
\item Suppose that $M_n/M'_n$ is isomorphic to the largest cyclic factor of $M_n$.
Then we have $\mathrm{Sel}_{0,n\ell_1 \ell_2}(\mathbb{Q}, E[p^{k_2}]) \hookrightarrow M'_n \oplus M'_n$ with cyclic cokernel.
This implies that $\mathrm{Sel}_{n\ell_1 \ell_2}(\mathbb{Q}, E[p^{k_2}]) \simeq M'_n \oplus M'_n$.
\end{itemize}
Therefore, $\mathrm{Sel}_{n\ell_1 \ell_2}(\mathbb{Q}, E[p^{k_2}])$ is (non-canonically) isomorphic to
the quotient of $\mathrm{Sel}_{n}(\mathbb{Q}, E[p^{k}])$ by the two copies of its largest cyclic factor.

\subsubsection{}
Putting it all together, we have
\begin{align*}
& \mathrm{length}_{\mathbb{Z}_p}   \mathrm{Sel}_n(\mathbb{Q}, E[p^k]) - \mathrm{length}_{\mathbb{Z}_p}   \mathrm{Sel}_{n \ell_1 \ell_2}(\mathbb{Q}, E[p^{k_2}]) \\
& = \left( \mathrm{length}_{\mathbb{Z}_p}   \mathrm{Sel}_{0,n}(\mathbb{Q}, E[p^k])  +  \mathrm{length}_{\mathbb{Z}_p}  \mathrm{coker} ( \mathrm{loc}^s_p(\mathrm{Sel}_{\mathrm{rel},n}) )^\vee \right) \\
& \ \ \  - \left( \mathrm{length}_{\mathbb{Z}_p}   \mathrm{Sel}_{n \ell_1 \ell_2}(\mathbb{Q}, E[p^{k_2}]) + \mathrm{length}_{\mathbb{Z}_p}  \mathrm{coker} ( \mathrm{loc}^s_p(\mathrm{Sel}_{\mathrm{rel},n \ell_1 \ell_2}) )^\vee \right)
\\
& =
\left( \mathrm{ord}_p (\kappa^{\mathrm{Kato}}_n) + \mathrm{length}_{\mathbb{Z}_p}  \mathrm{coker} ( \mathrm{loc}^s_p(\mathrm{Sel}_{\mathrm{rel},n}) )^\vee \right)
- \left( \mathrm{ord}_p (\kappa^{\mathrm{Kato}}_{n \ell_1 \ell_2}) + \mathrm{length}_{\mathbb{Z}_p}  \mathrm{coker} ( \mathrm{loc}^s_p(\mathrm{Sel}_{\mathrm{rel},n \ell_1 \ell_2}) )^\vee \right)
\\
& =
\mathrm{ord}_p (\widedelta_n) 
-  \mathrm{ord}_p (\widedelta_{n \ell_1 \ell_2}).
\end{align*}
By replacing $n$ by $n\ell_1\ell_2$, we repeat the same process until we get $\mathrm{Sel}_n(\mathbb{Q}, E[I_n]) = 0$. This completes the proof of Theorem \ref{thm:refined-bsd}.(2).

\section{The conjecture of Kurihara: Proof of Theorem \ref{thm:kurihara-conjecture}} \label{sec:kurihara-conjecture}
We first recall the main result of \cite{kazim-Lambda-adic}.
\begin{thm}[B\"{u}y\"{u}kboduk] \label{thm:rigidity}
Let $E$ be an elliptic curve over $\mathbb{Q}$ and $p \geq 5$ a prime such that
 $\overline{\rho}$ is surjective,
 $E(\mathbb{Q}_p)[p] = 0$, and
 all the Tamagawa factors are prime to $p$.
Let $\KSbar(T \otimes \Lambda)$ be the generalized module of $\Lambda$-adic Kolyvagin systems as in \cite[$\S$5.3]{mazur-rubin-book} and $\KS(T)$ the module of Kolyvagin systems.
Then
\begin{enumerate}
\item $\KSbar(T \otimes \Lambda)$ is free of rank one over $\Lambda$, and
\item the natural map $\KSbar(T \otimes \Lambda) \to \KS(T)$ is surjective.
\end{enumerate}
\end{thm}
\begin{proof}
See \cite[Thm. 3.23]{kazim-Lambda-adic}.
\end{proof}
\begin{rem}
When $p$ divides Tamagawa factors, the natural map in Theorem \ref{thm:rigidity}.(2) is not surjective \cite[Rem. 3.25]{kazim-Lambda-adic}.
In particular, there exists a Kolyvagin system which cannot lift to a $\Lambda$-adic Kolyvagin system (e.g. the primitive Kolyvagin system).
In this sense, Tamagawa factors can be understood as an obstruction to the existence of the $\Lambda$-adic lift of Kolyvagin systems. 
Since every Kolyvagin system arising from an Euler system lifts to a $\Lambda$-adic Kolyvagin system, Kato's Kolyvagin system cannot be primitive when $p$ divides Tamagawa factors.
Due to Theorem \ref{thm:main-technical},  the primitivity of Kato's Kolyvagin system is strictly stronger than the $\Lambda$-primitivity of the $\Lambda$-adic Kato's Kolyvagin system are not equivalent when $p$ divides any Tamagawa factor.
\end{rem}

\begin{cor} \label{cor:primitivity-iff-Lambda-primitivity}
Let $E$ be an elliptic curve over $\mathbb{Q}$ and $p \geq 5$  a prime such that
$\overline{\rho}$ is surjective,
 $E(\mathbb{Q}_p)[p] = 0$, and
 all the Tamagawa factors are prime to $p$.
Then $\ks^{\mathrm{Kato},(1)} \in \KS(T /pT, \mathcal{F}_{\mathrm{can}}, \mathcal{P})$ is non-zero if and only if $\ks^{\mathrm{Kato}, \infty}$ is a generator of $\KSbar(T \otimes \Lambda, \mathcal{F}_{\mathrm{can}}, \mathcal{P})$ over $\Lambda$.
\end{cor}
\begin{proof}
Since $\KS(T)$ is free of rank one over $\mathbb{Z}_p$ (Theorem \ref{thm:core-rank-elliptic-curves}), it follows from Theorem \ref{thm:rigidity}.
\end{proof}
\begin{proof}[Proof of Theorem \ref{thm:kurihara-conjecture}]
If $\widedelta^{(1)}_n \neq 0$, then $\kappa^{\mathrm{Kato}}_n \pmod{p}$ also does not vanish; thus, $\ks^{\mathrm{Kato}}$ is primitive. By the argument in \cite[Props. 4.1 and 4.2]{kazim-Lambda-adic} (see also \cite[Prop. 4.19]{kks}), the corresponding $\Lambda$-adic Kato's Kolyvagin system $\ks^{\mathrm{Kato},\infty}$ is $\Lambda$-primitive.
Then the Iwasawa main conjecture without $p$-adic $L$-functions follows \cite[Thm. 5.3.10]{mazur-rubin-book}.
In other words,
$$\mathrm{char}_{\Lambda} \left( \dfrac{\mathrm{H}^1_{\mathrm{Iw}}(\mathbb{Q}, T)}{\Lambda \cdot \kappa^{\mathrm{Kato},\infty}_{1} } \right)  = \mathrm{char}_{\Lambda} ( \mathrm{Sel}_0(\mathbb{Q}_\infty, E[p^\infty])^\vee ).$$

Suppose that $\widedelta^{(1)}_n = 0$ for every $n \in \mathcal{N}_1$. Then each $\kappa^{\mathrm{Kato},(1)}_n$ lies in the Selmer group whose Selmer structure at $p$ is the usual Selmer structure. However, the Kolyvagin system with the classical Selmer structure has core rank zero, and so 
$\ks^{\mathrm{Kato},(1)}$ is trivial (Theorem \ref{thm:core-rank-elliptic-curves}).

By Corollary \ref{cor:primitivity-iff-Lambda-primitivity}, $\ks^{\mathrm{Kato},(1)}$ is non-trivial if and only if $\ks^{\mathrm{Kato}, \infty}$ is a $\Lambda$-generator of $\KSbar(T \otimes \Lambda)$.

Let $\ks^{\mathrm{prim}, \infty}$ be a $\Lambda$-generator of $\KSbar(T \otimes \Lambda)$.
Then $\ks^{\mathrm{prim},\infty} = f \cdot \ks^{\mathrm{Kato}, \infty}$ with $f \in \Lambda$.
Since $\ks^{\mathrm{prim},\infty}$ is $\Lambda$-primitive by definition,
we have
$$\mathrm{char}_{\Lambda} \left( \dfrac{\mathrm{H}^1_{\mathrm{Iw}}(\mathbb{Q}, T)}{\Lambda \cdot \kappa^{\mathrm{prim},\infty}_{1} } \right)  = \mathrm{char}_{\Lambda} ( \mathrm{Sel}_0(\mathbb{Q}_\infty, E[p^\infty])^\vee ).$$
Since the Iwasawa main conjecture is equivalent to $f \in \Lambda^\times$, the equivalence statement follows.

Now we suppose that $\widedelta^{(1)}_n \neq 0$ with minimal $n$, i.e. $\nu(n) = \mathrm{ord}(\kn^{(1)})$.
The kernel of the canonical map
$$\mathrm{loc}_{n} : \mathrm{Sel}(\mathbb{Q}, E[p]) \to \bigoplus_{\ell \mid n} E(\mathbb{Q}_\ell) \otimes \mathbb{Z}/p\mathbb{Z}$$ 
is contained in $\mathrm{Sel}_{n}(\mathbb{Q}, E[p])$.
Since $\widedelta^{(1)}_n \neq 0$ implies $\kappa^{\mathrm{Kato}, (1)}_n \neq 0$, we have $\mathrm{Sel}_{0,n}(\mathbb{Q}, E[p])=0$.
Since $\widedelta^{(1)}_n \neq 0$, we also have
$ \left( \dfrac{\mathrm{H}^1_{/f}(\mathbb{Q}, T/pT)}{ \mathrm{loc}^s_p ( \mathrm{Sel}_{\mathrm{rel}, n} (\mathbb{Q}, T/pT) ) } \right)^\vee = 0$. 
Thus, $\mathrm{Sel}_{n}(\mathbb{Q}, E[p]) = 0$ by (\ref{eqn:kappa_n-widedelta_n-sequence}), so $\mathrm{loc}_n$ is injective under $\widedelta^{(1)}_n \neq 0$.

In order to prove the surjectivity, it suffices to show that
$\mathrm{dim}_{\mathbb{F}_p} \mathrm{Sel}(\mathbb{Q}, E[p]) = \mathrm{ord}(\kn^{(1)}) $.
By \cite[Thm. 5.1.1.(iii)]{mazur-rubin-book}, we have
$\mathrm{dim}_{\mathbb{F}_p} \mathrm{Sel}_0(\mathbb{Q}, E[p]) = \mathrm{ord}(\ks^{\mathrm{Kato},(1)}) $.
Say $\kappa^{\mathrm{Kato},(1)}_{n_0} \neq 0$ with $\nu(n') = \mathrm{ord}(\ks^{\mathrm{Kato},(1)})$.
Then $\mathrm{Sel}_{0,n_0}(\mathbb{Q}, E[p])  = 0$ by \cite[Thm. 5.1.1.(ii)]{mazur-rubin-book}.
Thus, we have an isomorphism
\begin{equation} \label{eqn:loc_p-isom}
\xymatrix{
 \mathrm{Sel}_{n_0}(\mathbb{Q}, E[p]) 
       	\ar[r]	^-{\mathrm{loc}_p}_-{\simeq} & \left( \dfrac{  \mathrm{H}^1_{/f}(\mathbb{Q}_p, T/pT)  }{ \mathrm{loc}^s_p \left( \mathrm{Sel}_{\mathrm{rel},n_0}(\mathbb{Q}, T/pT) \right) } \right)^\vee .
	}
\end{equation}
Since $\kappa^{\mathrm{Kato},(1)}_{n_0} \neq 0$, 
the following statements are equivalent.
\begin{enumerate}
\item $\widedelta^{(1)}_{n_0} \neq 0$.
\item $\left( \dfrac{  \mathrm{H}^1_{/f}(\mathbb{Q}_p, T/pT)  }{ \mathrm{loc}^s_p \left( \mathrm{Sel}_{\mathrm{rel},n_0}(\mathbb{Q}, T/pT) \right) } \right)^\vee = 0$.
\item (\ref{eqn:loc_p-isom}) is the zero map between the zero spaces.
\item $\mathrm{Sel}_{0,n_0}(\mathbb{Q}, E[p]) = \mathrm{Sel}_{n_0}(\mathbb{Q}, E[p]) = 0$.
\end{enumerate}
In this case, we have
$$\mathrm{ord}(\kn^{(1)}) = \mathrm{ord}(\ks^{\mathrm{Kato},(1)}) = \nu(n_0) = \mathrm{dim}_{\mathbb{F}_p} \mathrm{Sel}_{0}(\mathbb{Q}, E[p]) = \mathrm{dim}_{\mathbb{F}_p} \mathrm{Sel}(\mathbb{Q}, E[p]) ,$$
so we are done.

Now we suppose $\widedelta^{(1)}_{n_0} = 0$.
Then we have
 $$\mathrm{dim}_{\mathbb{F}_p} \mathrm{Sel}(\mathbb{Q}, E[p]) - \nu(n_0) = \mathrm{dim}_{\mathbb{F}_p} \mathrm{Sel}_{n_0}(\mathbb{Q}, E[p]) = 1 .$$
By Proposition \ref{prop:non-vanishing-kappa-n}, we are able to choose a useful Kolyvagin prime $\ell$ for $\kappa^{\mathrm{Kato},(1)}_{n_0}$
such that the localization map
$\mathrm{loc}_\ell : \mathrm{Sel}_{n_0}(\mathbb{Q}, E[p]) \to E(\mathbb{Q}_\ell) \otimes \mathbb{Z}/p\mathbb{Z}$
is surjective, so it is an isomorphism.
Furthermore, the kernel of $\mathrm{loc}_\ell$ is $\mathrm{Sel}_{n_0, \ell\textrm{-}\mathrm{str}}(\mathbb{Q}, E[p]) = \mathrm{Sel}_{n_0 \ell}(\mathbb{Q}, E[p]) = 0$.
Since $\kappa^{\mathrm{Kato}, (1)}_{n_0 \ell} \neq 0$, we have $\widedelta^{(1)}_{n_0 \ell} \neq 0$ with $n_0 \ell = \mathrm{ord}(\kn^{(1)})$ due to the above equivalence.
Thus, we have
$$\mathrm{ord}(\kn^{(1)}) = \mathrm{ord}(\ks^{\mathrm{Kato},(1)}) +1= \nu(n_0) + 1 = \mathrm{dim}_{\mathbb{F}_p} \mathrm{Sel}_{0}(\mathbb{Q}, E[p]) +1 = \mathrm{dim}_{\mathbb{F}_p} \mathrm{Sel}(\mathbb{Q}, E[p]) .$$

Regarding the rank formula, it immediately follows from
$\mathrm{dim}_{\mathbb{F}_p} \mathrm{Sel}(\mathbb{Q}, E[p]) = \mathrm{rk}_{\mathbb{Z}}E(\mathbb{Q}) + \mathrm{dim}_{\mathbb{F}_p} \sha(E/\mathbb{Q})[p]$ under the running hypotheses.
\end{proof}

\section{Iwasawa modules and $p$-adic BSD conjectures: Proof of Theorem \ref{thm:main-structure-fine-iwasawa}} \label{sec:iwasawa-modules}
The goal of this section is to prove Theorem \ref{thm:main-structure-fine-iwasawa}.
The idea of proof is similar to that of Theorem \ref{thm:main-technical}.
\subsection{One-sided divisibility}
We recall Kato's result on Conjecture \ref{conj:p-adic-BSD-kato-zeta}.
\begin{prop}[Kato] \label{prop:kato-inequality-p-adic-bsd}
The following inequality is valid
\begin{align*}
\mathrm{cork}_{\mathbb{Z}_p} \mathrm{Sel}_0(\mathbb{Q}, E[p^\infty]) & \leq 
\mathrm{ord}_{X\Lambda} \left( \mathrm{char}_{\Lambda} \left( \mathrm{Sel}_0(\mathbb{Q}_\infty, E[p^\infty])^\vee \right) \right) \\
 & \leq \mathrm{ord}_{X\Lambda} \left( \mathrm{char}_{\Lambda} \left( \dfrac{\mathrm{H}^1_{\mathrm{Iw}}(\mathbb{Q}, T)}{\kappa^{\mathrm{Kato}, \infty}_1} \right) \right) .
\end{align*}
\end{prop}
\begin{proof}
See \cite[Lem. 18.7]{kato-euler-systems}. The first statement basically follows from the control theorem for fine Selmer groups and the second statement follows from the one-sided divisibility of the Iwasawa main conjecture.
\end{proof}

\subsection{A general reduction}
We adapt the notation in $\S$\ref{subsec:proof-of-main-technical}.

Let $\mathfrak{P} = (f_{\mathfrak{P}}(X) )$ be a hight one prime of $\Lambda$ with $\mathfrak{P} \neq p\Lambda$.
Recall the fixed pseudo-isomorphism (\ref{eqn:pseudo-isom-XLambda})
\begin{equation*} 
\mathrm{Sel}_0(\mathbb{Q}_{\infty}, E[p^\infty] )^\vee \to \bigoplus_i \Lambda / \mathfrak{P}^{m_i} \Lambda \oplus \bigoplus_j \Lambda / f_j \Lambda
\end{equation*}
where each $f_j$ is prime to $\mathfrak{P}$ and $\mathfrak{P}_M = ( f_{\mathfrak{P}}(X) +p^M )$ by choosing sufficiently large $M$ as explained in Lemma \ref{lem:higher-power}.
Then
\begin{align} \label{eqn:sel-0-p-M}
\begin{split}
\left( \mathrm{Sel}_0(\mathbb{Q}_{\infty}, E[p^\infty])[\mathfrak{P}_M] \right)^\vee & \simeq \bigoplus_i \Lambda / (\mathfrak{P}_M, \mathfrak{P}^{m_i}) \\
& \simeq \bigoplus_i \Lambda / (\mathfrak{P}_M, p^{Mm_i})
\end{split}
\end{align}
up to a finite abelian group whose size is independent of $M$.
Thus, we have
$$\mathrm{Sel}_0(\mathbb{Q}, (T \otimes S_{\mathfrak{P}_M} )^* )^\vee \simeq \bigoplus_i S_{\mathfrak{P}_M} /p^{Mm_i} S_{\mathfrak{P}_M}$$
up to a finite abelian group whose size is independent of $M$ again. In terms of the sizes, we have
$$\mathrm{length}_{S_{\mathfrak{P}_M}} \mathrm{Sel}_0(\mathbb{Q}, (T \otimes S_{\mathfrak{P}_M} )^* )^\vee = M \cdot \sum_i m_i +O(1) .$$
On the other hand, by taking mod $p^M$ reduction of (\ref{eqn:sel-0-p-M}), we obtain
\begin{align*}
\left( \mathrm{Sel}_0(\mathbb{Q}_{\infty}, E[p^\infty])[\mathfrak{P}_M, p^M] \right)^\vee & \simeq \bigoplus_i \Lambda / (\mathfrak{P}_M, \mathfrak{P}^{m_i}, p^M) \\
& \simeq \bigoplus_i \Lambda / (\mathfrak{P}_M, p^M) \\
& \simeq \bigoplus_i \Lambda / (\mathfrak{P}, p^M)
\end{align*}
up to a finite abelian group whose size is independent of $M$.
This shows that
\begin{align*}
\mathrm{Sel}_0(\mathbb{Q}, (T/p^MT \otimes S_{\mathfrak{P}_M} )^* )^\vee & \simeq \bigoplus_i S_{\mathfrak{P}_M} /p^M S_{\mathfrak{P}_M} \\
& \simeq \bigoplus_i S_{\mathfrak{P}} /p^M S_{\mathfrak{P}} \\
& \simeq \mathrm{Sel}_0(\mathbb{Q}, (T/p^MT \otimes S_{\mathfrak{P}} )^* )^\vee 
\end{align*}
up to a finite abelian group whose size is independent of $M$.
Thus, we have
$$\mathrm{Sel}_0(\mathbb{Q}, (T \otimes S_{\mathfrak{P}} )^* )^\vee \simeq \bigoplus_i S_{\mathfrak{P}} $$
up to a finite abelian group whose size is independent of $M$ again, so we have
$$\mathrm{cork}_{S_{\mathfrak{P}}} \mathrm{Sel}_0(\mathbb{Q}, (T \otimes S_{\mathfrak{P}} )^* )  = \sum_i 1.$$
We now put $\mathfrak{P} =X\Lambda$ and assume Conjecture \ref{conj:p-adic-BSD-kato-zeta}.
Then by Proposition \ref{prop:kato-inequality-p-adic-bsd}, we have
\begin{align*}
\sum_i 1 & = \mathrm{cork}_{\mathbb{Z}_p} \mathrm{Sel}_0(\mathbb{Q}, E[p^\infty] )  \\
& = \mathrm{ord}_{X\Lambda} \left( \mathrm{char}_{\Lambda} \left( \mathrm{Sel}_0(\mathbb{Q}_\infty, E[p^\infty])^\vee \right) \right) \\
& = \sum_i m_i.
\end{align*}
Thus, Theorem \ref{thm:main-structure-fine-iwasawa} follows.
\begin{rem}
It seems that the same argument works for the classical Selmer groups with $p$-adic $L$-functions or signed $p$-adic $L$-functions when $E$ has good ordinary reduction or good supersingular reduction at $p$, respectively.
\end{rem}

\section{Numerical examples} \label{sec:examples}
We illustrate some numerical examples 
regarding Theorem \ref{thm:refined-bsd} and Theorem \ref{thm:kurihara-conjecture}
based on \cite{lmfdb-2021}.
We fix $p = 5$.
\begin{enumerate}
\item Let $E_{\textrm{389.a1}}$ be the elliptic curve defined by minimal Weierstrass equation $y^2 +y=x^3 +x^2 -2x $.
Then we have $\widedelta^{(1)}_{41 \cdot 61} (E_{\textrm{389.a1}}) \neq 0 \in \mathbb{F}_5$.
The following statements follow from the Kurihara number computation.
\begin{itemize}
\item $\mathrm{cork}_{\mathbb{Z}_5}\mathrm{Sel}(\mathbb{Q}, E_{\textrm{389.a1}}[5^\infty]) \leq 2$. Since it is the elliptic curve of rank 2 with the smallest conductor, the inequality becomes the equality.
\item All the Tamagawa factors of $E_{\textrm{389.a1}}$ are not divisible by 5.
\item $\sha(E_{\textrm{389.a1}}/\mathbb{Q})[5^\infty]$ is trivial.
\item There exists a canonical isomorphism 
$$\mathrm{Sel}(\mathbb{Q}, E_{\textrm{389.a1}}[5]) \simeq E_{\textrm{389.a1}}(\mathbb{Q}_{41}) \otimes \mathbb{Z}/5\mathbb{Z} \oplus E_{\textrm{389.a1}}(\mathbb{Q}_{61}) \otimes \mathbb{Z}/5\mathbb{Z} .$$
\end{itemize}
\item Let $E_{\textrm{5077.a1}}$  be the elliptic curve defined by minimal Weierstrass equation $y^2+y=x^3-7x+6$.
Then we have $\widedelta^{(1)}_{71 \cdot 401 \cdot 631}(E_{\textrm{5077.a1}}) \neq 0 \in \mathbb{F}_5$.
The following statements follow from the Kurihara number computation.
\begin{itemize}
\item $\mathrm{cork}_{\mathbb{Z}_5}\mathrm{Sel}(\mathbb{Q}, E_{\textrm{5077.a1}}[5^\infty]) \leq 3$. Since it is the elliptic curve of rank 3 with the smallest conductor, the inequality becomes the equality.
\item All the Tamagawa factors of $E_{\textrm{5077.a1}}$ are not divisible by 5.
\item $\sha(E_{\textrm{5077.a1}}/\mathbb{Q})[5^\infty]$ is trivial.
\item There exists a canonical isomorphism 
\begin{align*}
& \mathrm{Sel}(\mathbb{Q}, E_{\textrm{5077.a1}}[5])  \\
& \simeq E_{\textrm{5077.a1}}(\mathbb{Q}_{71}) \otimes \mathbb{Z}/5\mathbb{Z} \oplus E_{\textrm{5077.a1}}(\mathbb{Q}_{401}) \otimes \mathbb{Z}/5\mathbb{Z} \oplus E_{\textrm{5077.a1}}(\mathbb{Q}_{631}) \otimes \mathbb{Z}/5\mathbb{Z}.
\end{align*}
\end{itemize}
\item Let $E_{\textrm{1058.e1}}$ be the elliptic curve defined by minimal Weierstrass equation $y^2 +xy=x^3-x^2 -   332311x-73733731$.
Then we have 
\[
\xymatrix{
\widedelta_1 = \dfrac{L(E_{\textrm{1058.e1}},1)}{\Omega^+_{E_{\textrm{1058.e1}}}} = 25 , & \widedelta^{(1)}_{131 \cdot 151}(E_{\textrm{1058.e1}}) \neq 0 \in \mathbb{F}_5 .
}
\]
It is not difficult to observe that  all the Tamagawa factors of $E_{\textrm{1058.e1}}$ are not divisible by 5.
By Proposition \ref{prop:vanishing-delta-n}, we have $\mathrm{ord}(\kn^{(1)}) =2$ in this case.
The following statements follow from the above discussion.
\begin{itemize}
\item $\mathrm{cork}_{\mathbb{Z}_5}\mathrm{Sel}(\mathbb{Q}, E_{\textrm{1058.e1}}[5^\infty]) =0$, so $\mathrm{rk}_{\mathbb{Z}}E_{\textrm{1058.e1}}(\mathbb{Q}) = 0$.
\item $\sha(E_{\textrm{1058.e1}}/\mathbb{Q})[5^\infty] \simeq (\mathbb{Z}/5\mathbb{Z})^{\oplus 2}$.
\item There exists a canonical isomorphism 
$$\mathrm{Sel}(\mathbb{Q}, E_{\textrm{1058.e1}}[5]) \simeq E_{\textrm{1058.e1}}(\mathbb{Q}_{131}) \otimes \mathbb{Z}/5\mathbb{Z} \oplus E_{\textrm{1058.e1}}(\mathbb{Q}_{151}) \otimes \mathbb{Z}/5\mathbb{Z} .$$
\end{itemize}
\item 
Let $E_{\textrm{196794.bf1}}$ be the elliptic curve defined by minimal Weierstrass equation $y^2+xy=x^3-x^2-672055191x-6705708066275$.
Then we have
\[
\xymatrix@R=0em{
\widedelta_1 = \dfrac{L(E_{\textrm{196794.bf1}},1)}{\Omega^+_{E_{\textrm{196794.bf1}}}} = 0 , & \widedelta^{(3)}_{93251}(E_{\textrm{196794.bf1}}) = 5^2 \cdot u \in \mathbb{Z}/5^3 \mathbb{Z} , \\
\widedelta^{(1)}_{11 \cdot 13 \cdot 131}(E_{\textrm{196794.bf1}}) \neq 0 \in \mathbb{F}_5 
}
\]
where $u \in (\mathbb{Z}/5^3 \mathbb{Z})^\times$.
The following statements follow from the above computation.
\begin{itemize}
\item
$\mathrm{cork}_{\mathbb{Z}_5}\mathrm{Sel}(\mathbb{Q}, E_{\textrm{196794.bf1}}[5^\infty]) =1$; indeed, $\mathrm{rk}_{\mathbb{Z}}E_{\textrm{196794.bf1}}(\mathbb{Q}) = 1$.
\item All the Tamagawa factors of $E_{\textrm{196794.bf1}}$ are not divisible by 5.
\item $\sha(E_{\textrm{196794.bf1}}/\mathbb{Q})[5^\infty] \simeq (\mathbb{Z}/5\mathbb{Z})^{\oplus 2}$.
\item There exists a canonical isomorphism 
\begin{align*}
& \mathrm{Sel}(\mathbb{Q}, E_{\textrm{196794.bf1}}[5])  \\
& \simeq E_{\textrm{196794.bf1}}(\mathbb{Q}_{11}) \otimes \mathbb{Z}/5\mathbb{Z} \oplus E_{\textrm{196794.bf1}}(\mathbb{Q}_{13}) \otimes \mathbb{Z}/5\mathbb{Z} \oplus E_{\textrm{196794.bf1}}(\mathbb{Q}_{131}) \otimes \mathbb{Z}/5\mathbb{Z}.
\end{align*}
\end{itemize}
\item \label{exam:625}
Let $E_{\textrm{423801.ci1}}$ be the elliptic curve defined by minimal Weierstrass equation $y^2+y=x^3-17034726259173x-27061436852750306309$.
Then we have
\[
\xymatrix@R=0em{
\widedelta_1 = \dfrac{L(E_{\textrm{423801.ci1}},1)}{\Omega^+_{E_{\textrm{423801.ci1}}}} = 10000 = 2^4 \cdot 5^4, &
\widedelta^{(1)}_{11 \cdot 41}(E_{\textrm{423801.ci1}}) \neq 0 \in \mathbb{F}_5 .
}
\]
The following statements follow from the above computation.
\begin{itemize}
\item
$\mathrm{cork}_{\mathbb{Z}_5}\mathrm{Sel}(\mathbb{Q}, E_{\textrm{423801.ci1}}[5^\infty]) =0$; indeed, $\mathrm{rk}_{\mathbb{Z}}E_{\textrm{423801.ci1}}(\mathbb{Q}) = 0$.
\item All the Tamagawa factors of $E_{\textrm{423801.ci1}}$ are not divisible by 5.
\item $\sha(E_{\textrm{423801.ci1}}/\mathbb{Q})[5^\infty] \simeq (\mathbb{Z}/25\mathbb{Z})^{\oplus 2} (\not\simeq (\mathbb{Z}/5\mathbb{Z})^{\oplus 4} )$.
\item There exists a canonical isomorphism 
\begin{align*}
& \mathrm{Sel}(\mathbb{Q}, E_{\textrm{423801.ci1}}[5])  \\
& \simeq E_{\textrm{423801.ci1}}(\mathbb{Q}_{11}) \otimes \mathbb{Z}/5\mathbb{Z} \oplus E_{\textrm{423801.ci1}}(\mathbb{Q}_{41}) \otimes \mathbb{Z}/5\mathbb{Z}.
\end{align*}
\end{itemize}
It is remarkable that $\widedelta^{(1)}_{11 \cdot 41}(E_{\textrm{423801.ci1}}) \neq 0$ implies $\sha(E_{\textrm{423801.ci1}}/\mathbb{Q})[5^\infty] \simeq  (\mathbb{Z}/25\mathbb{Z})^{\oplus 2}$ and this structural information is not observed in Birch and Swinnerton-Dyer conjecture.
We took this example from \cite[Ex. 4]{kurihara-analytic-quantities}, and the question was raised by D. Prasad--Shekhar \cite[Ex. 2]{prasad-shekhar}.
\end{enumerate}
See also \cite[$\S$3.8]{grigorov-thesis}, \cite[$\S$10.15]{kurihara-munster}, \cite[$\S$5.3]{kurihara-iwasawa-2012}, \cite[$\S$8]{kks}, \cite[Appendix A]{kim-survey}, and \cite{kurihara-analytic-quantities} for various examples on the computations of $\widedelta_n$'s.
Both the Iwasawa main conjecture and the low analytic rank assumption in Theorem \ref{thm:main-intro} are \emph{not} essential \emph{at all} when we compute numerical examples in practice.

\section*{Acknowledgement}
As will be clear to the reader, this article is strongly inspired from the philosophy of refined Iwasawa theory of
Kurihara \cite{kurihara-fitting, kurihara-documenta, kurihara-plms, kurihara-munster, kurihara-iwasawa-2012} 
and the ideas in the work of Mazur--Rubin on Kolyvagin systems \cite{mazur-rubin-book}.
The author would like to thank Masato Kurihara and Karl Rubin for extensive and very helpful discussions.
Especially, Masato Kurihara pointed out a serious gap in an earlier version.
Indeed, the initial motivation of this work is to develop a precise comparison between 
\cite[Thm. B]{kurihara-munster} and \cite[Thm. 5.2.12]{mazur-rubin-book}.
The author are also benefited from the discussions with K\^{a}z{\i}m B\"{u}y\"{u}kboduk, Francesc Castella, Tatsuya Ohshita, and Chris Wuthrich.

We are greatly benefited from the discussion with Alberto Angurel Andres.
The discussion with him led to a significant simplification and clarification of the proof of the main result.
The revision was essentially made while the author worked at Ewha Institute of Mathematical Sciences.
The author gratefully acknowledges their support and warm hospitality.

We would also like to heartily thank for the anonymous referee for a very careful reading of the paper and pointing out many inaccuracies in an earlier version.
\bibliographystyle{amsalpha}
\bibliography{library}

\providecommand{\bysame}{\leavevmode\hbox to3em{\hrulefill}\thinspace}
\providecommand{\MR}{\relax\ifhmode\unskip\space\fi MR }
\providecommand{\MRhref}[2]{%
  \href{http://www.ams.org/mathscinet-getitem?mr=#1}{#2}
}
\providecommand{\href}[2]{#2}
\begin{thebibliography}{{LMF}21}

\bibitem[BBV16]{berti-bertolini-venerucci}
A.~Berti, M.~Bertolini, and R.~Venerucci, \emph{Congruences between modular
  forms and the {B}irch and {S}winnerton-{D}yer conjecture}, Elliptic Curves,
  Modular Forms and Iwasawa Theory (David Loeffler and Sarah~Livia Zerbes,
  eds.), Springer Proc. Math. Stat., vol. 188, Springer, 2016, In Honour of
  John H. Coates' 70th Birthday, Cambridge, UK, March 2015, pp.~1--31.

\bibitem[BCDT01]{bcdt}
C.~Breuil, B.~Conrad, F.~Diamond, and R.~Taylor, \emph{On the modularity of
  elliptic curves over {$\mathbb{Q}$}: wild 3-adic exercises}, J. Amer. Math.
  Soc. \textbf{14} (2001), no.~4, 843--939.

\bibitem[BCGS]{burungale-castella-grossi-skinner-indivisibility}
A.~A. Burungale, F.~Castella, G.~Grossi, and C.~Skinner, \emph{Non-vanishing of
  {K}olyvagin systems and {I}wasawa theory}, preprint,
  \href{https://arxiv.org/abs/2312.09301}{arXiv:2312.09301}.

\bibitem[BCS25]{burungale-castella-skinner-gl2}
A.~A. Burungale, F.~Castella, and C.~Skinner, \emph{Base change and {I}wasawa
  main conjectures for {$\mathrm{GL}_2$}}, Int. Math. Res. Not. IMRN
  \textbf{2025} (2025), no.~8, rnaf082.

\bibitem[BDV22]{bertolini-darmon-venerucci}
M.~Bertolini, H.~Darmon, and R.~Venerucci, \emph{Heegner points and
  {B}eilinson--{K}ato elements: a conjecture of {P}errin-{R}iou}, Adv. Math.
  \textbf{398} (2022), 108172.

\bibitem[BKS24]{burns-kurihara-sano}
D.~Burns, M.~Kurihara, and T.~Sano, \emph{On derivatives of {K}ato's {E}uler
  system for elliptic curves}, J. Math. Soc. Japan \textbf{76} (2024), no.~3,
  855--919.

\bibitem[BMS16]{balakrishnan-muller-stein}
J.~Balakrishnan, S.~M{\"{u}}ller, and W.~Stein, \emph{A $p$-adic analogue of
  the conjecture of {B}irch and {S}winnerton-{D}yer for modular abelian
  varieties}, Math. Comp. \textbf{85} (2016), 983--1016.

\bibitem[BPS]{kazim-pollack-sasaki}
K.~B{\"{u}}y{\"{u}}kboduk, R.~Pollack, and S.~Sasaki, \emph{{$p$}-adic
  {G}ross--{Z}agier formula at critical slope and a conjecture of
  {P}errin-{R}iou}, preprint,
  \href{https://arxiv.org/abs/1811.08216}{arXiv:1811.08216}.

\bibitem[BST21]{burungale-skinner-tian-survey}
A.~A. Burungale, C.~Skinner, and Y.~Tian, \emph{The {B}irch and
  {S}winnerton-{D}yer conjecture: a brief survey}, Nine mathematical
  challenges--an elucidation (Providence, RI) (A.~Kechris, N.~Makarov,
  D.~Ramakrishnan, and X.~Zhu, eds.), Proc. Sympos. Pure Math., vol. 104, Amer.
  Math. Soc., 2021, pp.~11--29.

\bibitem[BSTW]{burungale-skinner-tian-wan}
A.~A. Burungale, C.~Skinner, Y.~Tian, and X.~Wan, \emph{Zeta elements for
  elliptic curves and application}, preprint,
  \href{https://arxiv.org/abs/2409.01350}{arXiv:2409.01350}.

\bibitem[BT20]{burungale-tian-p-converse}
A.~A. Burungale and Y.~Tian, \emph{{$p$}-converse to a theorem of {G}ross--{Z}agier, {K}olyvagin
  and {R}ubin}, Invent. Math. \textbf{220} (2020), no.~1, 211--253.

\bibitem[BT22]{burungale-tian-jnt}
\bysame, \emph{The even parity {G}oldfeld conjecture: congruent number elliptic
  curves}, J. Number Theory \textbf{230} (2022), 161--195.

\bibitem[BT]{burungale-tian-rank-zero-p-converse}
\bysame, \emph{A rank zero {$p$}-converse to a theorem of
  {G}ross--{Z}agier, {K}olyvagin and {R}ubin}, preprint.

\bibitem[B{\"{u}}y11]{kazim-Lambda-adic}
K.~B{\"{u}}y{\"{u}}kboduk, \emph{{$\Lambda$}-adic {K}olyvagin systems}, Int.
  Math. Res. Not. IMRN (2011), no.~14, 3141--3206.

\bibitem[Cai07]{bryden-cais-thesis}
B.~Cais, \emph{Correspondences, integral structures, and compatibilities in
  {$p$}-adic cohomology}, Ph.D. thesis, The University of Michigan, 2007, under
  the supervision of Brian Conrad.

\bibitem[Cas18]{castella-cambridge}
F.~Castella, \emph{On the $p$-part of the {B}irch--{S}winnerton-{D}yer formula
  for multiplicative primes}, Camb. J. Math. \textbf{6} (2018), no.~1, 1--23.

\bibitem[CGLS22]{castella-grossi-lee-skinner}
F.~Castella, G.~Grossi, J.~Lee, and C.~Skinner, \emph{On the anticyclotomic
  {I}wasawa theory of rational elliptic curves at {E}isenstein primes}, Invent.
  Math. \textbf{227} (2022), 517--580.

\bibitem[CW77]{coates-wiles-bsd-1977}
J.~Coates and A.~Wiles, \emph{On the conjecture of {B}irch and
  {S}winnerton-{D}yer}, Invent. Math. \textbf{39} (1977), no.~3, 223--252.

\bibitem[DW08]{david-weston}
C.~David and T.~Weston, \emph{Local torsion of elliptic curves and the
  deformation theory of {G}alois representations}, Math. Res. Lett. \textbf{15}
  (2008), no.~3, 599--611.

\bibitem[Fla90]{flach-cassels-tate}
M.~Flach, \emph{A generalisation of the {C}assels--{T}ate pairing}, J. Reine
  Angew. Math. \textbf{412} (1990), 113--127.

\bibitem[Gha05]{ghate-2005}
E.~Ghate, \emph{Ordinary forms and their local {G}alois representations},
  Algebra and Number Theory (R.~Tandon, ed.), Hindustan Book Agency, 2005,
  pp.~226--242.

\bibitem[Gre99]{greenberg-lnm}
R.~Greenberg, \emph{Iwasawa theory for elliptic curves}, Arithmetic theory of
  elliptic curves ({C}etraro, 1997) (Berlin) (C.~Viola, ed.), Lecture Notes in
  Math., vol. 1716, Centro Internazionale Matematico Estivo (C.I.M.E.),
  Florence, Springer-Verlag, 1999, Lectures from the 3rd {C.I.M.E.}~{S}ession
  held in {C}etraro, {J}uly 12-–19, 1997, pp.~51–--144.

\bibitem[Gri05]{grigorov-thesis}
G.~T. Grigorov, \emph{Kato's {E}uler {S}ystems and the {M}ain {C}onjecture},
  Ph.D. thesis, Harvard, May 2005, under the supervision of R. Taylor.

\bibitem[Gro90]{gross-tameness}
B.~Gross, \emph{A tameness criterion for {G}alois representations associated to
  modular forms (mod {$p$})}, Duke Math. J. \textbf{61} (1990), no.~2,
  445--517.

\bibitem[GZ86]{gross-zagier-original}
B.~Gross and D.~Zagier, \emph{Heegner points and derivatives of {$L$}-series},
  Invent. Math. \textbf{84} (1986), no.~2, 225--320.

\bibitem[How04]{howard-kolyvagin}
B.~Howard, \emph{The {H}eegner point {K}olyvagin system}, Compos. Math.
  \textbf{140} (2004), no.~6, 1439--1472.

\bibitem[JSW17]{jetchev-skinner-wan}
D.~Jetchev, C.~Skinner, and X.~Wan, \emph{The {B}irch and {S}winnerton-{D}yer
  formula for elliptic curves of analytic rank one}, Camb. J. Math. \textbf{5}
  (2017), no.~3, 369--434.

\bibitem[Kat99]{kato-euler-iwasawa-selmer}
K.~Kato, \emph{Euler systems, {I}wasawa theory, and {S}elmer groups}, Kodai
  Math. J. \textbf{22} (1999), no.~3, 313--372.

\bibitem[Kat04]{kato-euler-systems}
\bysame, \emph{{$p$}-adic {H}odge theory and values of zeta functions of
  modular forms}, Ast\'{e}risque \textbf{295} (2004), 117--290.

\bibitem[Kat17]{kato-survey}
\bysame, \emph{Some topics in recent developments in number theory
  ({J}apanese)}, S{\={u}}gaku \textbf{69} (2017), no.~4, 413--428.

\bibitem[Kat21]{kataoka-thesis}
T.~Kataoka, \emph{Equivariant {I}wasawa theory for elliptic curves}, Math. Z.
  \textbf{298} (2021), 1653--1725.

\bibitem[Kim21]{kim-survey}
C.-H. Kim, \emph{Indivisibility of {K}ato's {E}uler systems and {K}urihara
  numbers}, RIMS K{\^{o}}ky{\^{u}}roku Bessatsu \textbf{B86} (2021), 63--86,
  with an appendix by A. Ghitza.

\bibitem[Kim23]{kim-p-converse}
\bysame, \emph{On the soft {$p$}-converse to a theorem of {G}ross--{Z}agier and
  {K}olyvagin}, Math. Ann. \textbf{387} (2023), 1961--1968.

\bibitem[Kim24]{kim-gross-zagier}
\bysame, \emph{A higher {G}ross--{Z}agier formula and the structure of {S}elmer
  groups}, Trans. Amer. Math. Soc. \textbf{377} (2024), no.~5, 3691--3725.

\bibitem[KKS20]{kks}
C.-H. Kim, M.~Kim, and H.-S. Sun, \emph{On the indivisibility of derived
  {K}ato's {E}uler systems and the main conjecture for modular forms}, Selecta
  Math. (N.S.) \textbf{26} (2020), no.~31, 47 pages.

\bibitem[KN20]{kim-nakamura}
C.-H. Kim and K.~Nakamura, \emph{Remarks on {K}ato's {E}uler systems for
  elliptic curves with additive reduction}, J. Number Theory \textbf{210}
  (2020), 249--279.

\bibitem[Kob06]{kobayashi-elementary}
S.~Kobayashi, \emph{An elementary proof of the {M}azur--{T}ate--{T}eitelbaum
  conjecture for elliptic curves}, Doc. Math. (2006), 567--575, Extra Volume:
  John H. Coates' Sixtieth Birthday.

\bibitem[Kob13]{kobayashi-gross-zagier}
\bysame, \emph{The {$p$}-adic {G}ross--{Z}agier formula for elliptic curves at
  supersingular primes}, Invent. Math. \textbf{191} (2013), no.~3, 527--629.

\bibitem[Kol90]{kolyvagin-euler-systems}
V.~Kolyvagin, \emph{Euler systems}, The {G}rothendieck {F}estschrift {V}olume
  {II} (P.~Cartier, L.~Illusie, N.~M. Katz, G.~Laumon, Y.~Manin, and K.~A.
  Ribet, eds.), Progr. Math., vol.~87, Birkh\"{a}user {B}oston, 1990,
  pp.~435--483.

\bibitem[Kol91]{kolyvagin-selmer}
\bysame, \emph{On the structure of {S}elmer groups}, Math. Ann. \textbf{291}
  (1991), no.~2, 253--259.

\bibitem[KP]{kosters-pannekoek}
M.~Kosters and R.~Pannekoek, \emph{On the structure of elliptic curves over
  finite extensions of {$\mathbb{Q}_p$} with additive reduction}, preprint,
  \href{https://arxiv.org/abs/1703.07888}{arXiv:1703.07888}.

\bibitem[Kur03a]{kurihara-fitting}
M.~Kurihara, \emph{Iwasawa theory and {F}itting ideals}, J. Reine Angew. Math.
  \textbf{561} (2003), 39--86.

\bibitem[Kur03b]{kurihara-documenta}
\bysame, \emph{On the structure of ideal class groups of {CM}-fields}, Doc.
  Math. (2003), 539--563, Extra {V}olume: {K}azuya {K}ato's {F}iftieth
  {B}irthday.

\bibitem[Kur12]{kurihara-plms}
\bysame, \emph{Refined {I}wasawa theory and {K}olyvagin systems of {G}auss sum
  type}, Proc. Lond. Math. Soc. (3) \textbf{104} (2012), no.~4, 728--769.

\bibitem[Kur14a]{kurihara-munster}
\bysame, \emph{Refined {I}wasawa theory for {$p$}-adic representations and the
  structure of {S}elmer groups}, M{\"{u}}nster J. of Math. \textbf{7} (2014),
  no.~1, 149--223.

\bibitem[Kur14b]{kurihara-iwasawa-2012}
\bysame, \emph{The structure of {S}elmer groups of elliptic curves and modular
  symbols}, Iwasawa Theory 2012: State of the Art and Recent Advances
  (T.~Bouganis and O.~Venjakob, eds.), Contrib. Math. Comput. Sci., vol.~7,
  Springer, 2014, pp.~317--356.

\bibitem[Kur24]{kurihara-analytic-quantities}
\bysame, \emph{Some anayltic quantities yielding arithemtic information about
  elliptic curves}, Arithmetic Geometry---{P}roceedings of the {I}nternational
  {C}olloquium, {M}umbai, 2020, Tata Inst. Fundam. Res. Stud. Math., vol.~24,
  Hindustan Book Agency, New Delhi, 2024, pp.~345--384.

\bibitem[{LMF}21]{lmfdb-2021}
The {LMFDB Collaboration}, \emph{The {L}-functions and modular forms database},
  \url{http://www.lmfdb.org}, 2021, [Online; accessed 18 March 2021].

\bibitem[Lor11]{lorenzini-torsion-tamagawa}
D.~Lorenzini, \emph{Torsion and {T}amagawa numbers}, Ann. Inst. Fourier
  (Grenoble) \textbf{61} (2011), no.~5, 1995--2037.

\bibitem[Maz78]{mazur-rational-isogenies}
B.~Mazur, \emph{Rational isogenies of prime degree}, Invent. Math. \textbf{44}
  (1978), 129--162, with an appendix by D. Goldfeld.

\bibitem[MR04]{mazur-rubin-book}
B.~Mazur and K.~Rubin, \emph{{K}olyvagin {S}ystems}, Mem. Amer. Math. Soc.,
  vol. 168, American {M}athematical {S}ociety, March 2004.

\bibitem[MR05]{mazur-rubin-organizing}
\bysame, \emph{Organizing the arithmetic of elliptic curves}, Adv. Math.
  \textbf{198} (2005), 504--546.

\bibitem[MT87]{mazur-tate}
B.~Mazur and J.~Tate, \emph{Refined conjectures of the ``{B}irch and
  {S}winnerton-{D}yer type"}, Duke Math. J. \textbf{54} (1987), no.~2,
  711--750.

\bibitem[MTT86]{mtt}
B.~Mazur, J.~Tate, and J.~Teitelbaum, \emph{On {$p$}-adic analogues of the
  conjectures of {B}irch and {S}winnerton-{D}yer}, Invent. Math. \textbf{84}
  (1986), no.~1, 1--48.

\bibitem[Ota18]{ota-thesis}
K.~Ota, \emph{Kato's {E}uler system and the {M}azur--{T}ate refined conjecture
  of {BSD} type}, Amer. J. Math. \textbf{140} (2018), no.~2, 495--542.

\bibitem[Ots09]{otsuki}
R.~Otsuki, \emph{Construction of a homomorphism concerning {E}uler systems for
  an elliptic curve}, Tokyo J. Math. \textbf{32} (2009), no.~1, 253--278.

\bibitem[PR93]{perrin-riou-rational-pts}
B.~Perrin-Riou, \emph{Fonctions {$L$} {$p$}-adiques d'une courbe elliptique et
  points rationnels}, Ann. Inst. Fourier (Grenoble) \textbf{43} (1993), no.~4,
  945--995.

\bibitem[PR98]{perrin-riou-euler-systems}
\bysame, \emph{Syst\`{e}mes d'{E}uler {$p$}-adiques et th\'{e}orie
  d'{I}wasawa}, Ann. Inst. Fourier (Grenoble) \textbf{48} (1998), no.~5,
  1231--1307.

\bibitem[PS21]{prasad-shekhar}
D.~Prasad and S.~Shekhar, \emph{Relating the {T}ate--{S}hafarevich group of an
  elliptic curve with the class group}, Pacific J. Math. \textbf{312} (2021),
  no.~1, 203--218.

\bibitem[Roh84]{rohrlich-nonvanishing}
D.~Rohrlich, \emph{On {$L$}-functions of elliptic curves and cyclotomic
  towers}, Invent. Math. \textbf{75} (1984), 409--423.

\bibitem[Roh88]{rohrlich-nonvanishing-2}
\bysame, \emph{{$L$}-functions and division towers}, Math. Ann. \textbf{281}
  (1988), 611--632.

\bibitem[Rub87]{rubin-tate-shafarevich}
K.~Rubin, \emph{Tate--{S}hafarevich groups and {$L$}-functions of elliptic
  curves with complex multiplication}, Invent. Math. \textbf{89} (1987),
  527--559.

\bibitem[Rub91]{rubin-main-conj-cm}
\bysame, \emph{The ``main conjectures" of {I}wasawa theory for imaginary
  quadratic fields}, Invent. Math. \textbf{103} (1991), no.~1, 25--68.

\bibitem[Rub97]{rubin-modularity-mod-5}
\bysame, \emph{Modularity of mod 5 representations}, Modular {F}orms and
  {F}ermat's {L}ast {T}heorem (G.~Cornell, J.~Silverman, and G.~Stevens, eds.),
  Springer, 1997, pp.~463--474.

\bibitem[Rub00]{rubin-book}
\bysame, \emph{Euler {S}ystems}, Ann. of Math. Stud., vol. 147, Princeton
  {U}niversity {P}ress, 2000.

\bibitem[Sak22]{sakamoto-p-selmer}
R.~Sakamoto, \emph{$p$-{S}elmer groups and modular symbols}, Doc. Math.
  \textbf{27} (2022), 1891--1922.

\bibitem[Sak24]{sakamoto-p-3}
\bysame, \emph{The theory of {K}olyvagin systems for {$p=3$}}, J. Th{\'{e}}or.
  Nombres Bordeaux \textbf{36} (2024), no.~3, 919--946.

\bibitem[Sch85]{schneider-height-2}
P.~Schneider, \emph{{$p$}-adic height pairings. {II}}, Invent. Math.
  \textbf{79} (1985), 329--374.

\bibitem[Sil99]{silverman2}
J.~Silverman, \emph{Advanced {T}opics in the {A}rithmetic of {E}lliptic
  {C}urves}, corrected second printing ed., Grad. Texts in Math., vol. 151,
  Springer-{V}erlag, 1999.

\bibitem[Sil09]{silverman}
\bysame, \emph{The {A}rithmetic of {E}lliptic {C}urves}, 2nd ed., Grad. Texts
  in Math., vol. 106, Springer-{V}erlag, 2009.

\bibitem[Ski16]{skinner-pacific}
C.~Skinner, \emph{Multiplicative reduction and the cyclotomic main conjecture
  for {$\mathrm{GL}_2$}}, Pacific J. Math. \textbf{283} (2016), no.~1,
  171--200.

\bibitem[Ski20]{skinner-converse}
\bysame, \emph{A converse to a theorem of {G}ross, {Z}agier, and {K}olyvagin},
  Ann. of Math. \textbf{191} (2020), no.~2, 329--354.

\bibitem[SU14]{skinner-urban}
C.~Skinner and E.~Urban, \emph{The {I}wasawa main conjectures for
  {$\mathrm{GL}_2$}}, Invent. Math. \textbf{195} (2014), no.~1, 1--277.

\bibitem[SW13]{stein-wuthrich}
W.~Stein and C.~Wuthrich, \emph{Algorithms for the arithmetic of elliptic
  curves using {I}wasawa theory}, Math. Comp. \textbf{82} (2013), no.~283,
  1757--1792.

\bibitem[Tat75]{tate-algorithm}
J.~Tate, \emph{Algorithm for determining the type of a singular fiber in an
  elliptic pencil}, Modular {F}unctions of {O}ne {V}ariable {IV}: Proceedings
  of the International Summer School, University of Antwerp, RUCA, July
  17--August 3, 1972 (B.~J. Birch and W.~Kuyk, eds.), Lecture Notes in Math.,
  vol. 476, Springer, 1975, pp.~33--52.

\bibitem[Wan15]{wan_hilbert}
X.~Wan, \emph{The {I}wasawa main conjecture for {H}ilbert modular forms}, Forum
  Math. Sigma \textbf{3} (2015), e18 (95 pages).

\bibitem[Wan20]{wan-rankin-selberg}
\bysame, \emph{Iwasawa main conjecture for {R}ankin--{S}elberg {$p$}-adic
  {$L$}-functions}, Algebra Number Theory \textbf{14} (2020), no.~2, 383--483.

\bibitem[Zha14a]{wei-zhang-cdm}
W.~Zhang, \emph{The {B}irch--{S}winnerton-{D}yer conjecture and {H}eegner
  points: a survey}, Current {D}evelopments in {M}athematics, vol. 2013, 2014,
  pp.~169--203.

\bibitem[Zha14b]{wei-zhang-mazur-tate}
\bysame, \emph{Selmer groups and the indivisibility of {H}eegner points}, Camb.
  J. Math. \textbf{2} (2014), no.~2, 191--253.

\end{thebibliography}

\end{document}